\documentclass[11pt,a4paper]{article}

\usepackage[utf8]{inputenc}
\usepackage[T1]{fontenc}
\usepackage{lmodern}
\usepackage{dsfont}
\usepackage{indentfirst}

\usepackage{amsfonts}
\usepackage{amsthm}
\usepackage{amsmath}
\usepackage{amssymb}
\usepackage{amscd}
\usepackage{mathrsfs}
\usepackage[mathscr]{eucal}
\usepackage{graphicx}
\usepackage{epstopdf}
\usepackage{pict2e}
\usepackage{epic}
\usepackage{xcolor}
\usepackage{float}
\usepackage{mathtools}
\usepackage{centernot}
\usepackage{theoremref}
\usepackage{xcolor}

\usepackage{cite}
\usepackage{hyperref}

\hypersetup{colorlinks=true, urlcolor= black, linkcolor=black, citecolor=black}
\usepackage{tikz}
\usepackage{orcidlink}

\numberwithin{equation}{section}

\theoremstyle{plain}
\newtheorem{Thm}{Theorem}[section]
\newtheorem{Lem}[Thm]{Lemma}
\newtheorem{Coro}[Thm]{Corollary}
\newtheorem{Prop}[Thm]{Proposition}

\theoremstyle{definition}
\newtheorem{Def}[Thm]{Definition}

\newtheorem{Rem}[Thm]{Remark}

\newtheorem{Ex}[Thm]{Example}

\newtheorem{Ass}{Assumption}

\usepackage[margin=3cm]{geometry}

\newcommand\1{\mathds{1}}

\newcommand\bfc{\mathbf{c}}
\newcommand\bfC{\mathbf{C}}
\newcommand\bfco{\mathbf{c}_0}
\newcommand\bfCo{\mathbf{C}_0}
\newcommand\Crw{\mathsf{C}_{\textup{RW}}}
\newcommand\crw{\mathsf{c}_{\textup{RW}}}
\newcommand\CGreen{\mathsf{C}}
\newcommand\cGreen{\mathsf{c}}
\newcommand\CM{\mathsf{C}_{\textup{reg}}}
\newcommand\cM{\mathsf{c}_{\textup{reg}}}
\newcommand\creg{\mathsf{c}_{\textup{reg}}}
\newcommand\Creg{\mathsf{C}_{\textup{reg}}}
\newcommand{\Cstab}{C_{\textup{stab}}}

\newcommand{\CAC}{\mathsf C_{\textup{ac}}}
\newcommand{\deltastab}{\delta_{\textup{stab}}}
\newcommand{\deltareg}{\delta_{\textup{reg}}}
\newcommand{\deltamain}{\boldsymbol{\delta}}

\newcommand{\D}{\mathrm{d}}

\newcommand{\connect}{\xleftrightarrow}

\renewcommand{\b}{\beta}

\definecolor{darkmagenta}{rgb}{0.55, 0.0, 0.55}

\title{A random walk approach to \\ high-dimensional critical phenomena}
\begin{document}

\author{Hugo Duminil-Copin
\orcidlink{0000-0002-7609-2816}
\footnotemark[1]\footnote{Section de Math\'ematiques, Universit\'e de Gen\`eve, 7-9 Rue du Conseil G\'en\'eral, 1205 Gen\`eve, Suisse. \url{hugo.duminil@unige.ch}, \url{aman.markar@unige.ch}}\:\:\footnotemark[2]\footnote{Institut des Hautes Études Scientifiques, 35 Rte de Chartres, 91440 Bures-sur-Yvette, France. \url{duminil@ihes.fr}}\:,
Aman Markar
\orcidlink{0009-0009-1387-3605}
\footnotemark[1]\:,
Romain Panis
\orcidlink{0009-0001-4604-8398}
\thanks{Universit\'e Lyon 1, Centrale Lyon, INSA Lyon, Universit\'e Jean Monnet, CNRS, ICJ UMR 5208, 69622, Villeurbanne, France. \url{panis@math.univ-lyon1.fr}.}\:, Gordon Slade
\orcidlink{0000-0001-9389-9497}
  \thanks{Department of Mathematics,
     University of British Columbia,
     Vancouver BC, Canada V6T 1Z2.
       \url{slade@math.ubc.ca}}
}

\date{\vspace{-5ex}} 

\maketitle

\begin{abstract}
We present a ``black box'' proof of mean-field near-critical behaviour for
a family of functions on $\mathbb Z^d$ (${d>2}$) satisfying a short list of assumptions.
The functions represent two-point functions of a lattice statistical mechanical
model in the subcritical or critical regimes, and are proved to have
decay of the form $|x|^{-d+2+\varepsilon}\exp[-c|x|/\xi]$, for any $\varepsilon>0$.
The black box applies to
several models for which commonplace methods can be used to verify the
assumptions.
Applications include models of self-avoiding walk, percolation, spins (Ising, XY, $|\varphi|^4$), and  lattice trees, all above their upper critical dimensions.
The proof is based on random walk techniques, and
provides a new, unified, probabilistic,
and relatively simple proof of mean-field near-critical behaviour.
\end{abstract}

\tableofcontents

\section{Introduction and results}

\subsection{Background and motivation}
\label{sec:motivation}

A central goal in statistical mechanics is to understand the
behaviour of lattice models which undergo a \emph{phase transition} at a \emph{critical point}. Of particular interest is
the intricate fractal behaviour of the model at and near its critical point.
This behaviour can be described by \emph{critical exponents}.
To prove the existence of these critical exponents, let alone to compute their values,
is in general an open problem.
A profound interplay between the dimension of the underlying lattice and the critical behaviour is expected.
We focus on the case of models defined on the
hypercubic lattice $\mathbb Z^d$ in dimensions $d\geq 2$, and particularly on high $d$.

A striking observation was made in the physics literature around half a century
ago \cite{Ginz61,WF72}: above an \emph{upper critical} dimension $d_c$, the critical behaviour of the model simplifies in the sense that its critical exponents match those obtained when the
model is formulated on a tree rather than on $\mathbb Z^d$. The regime $d\geq d_c$ constitutes the \emph{mean-field} regime of the model. It is characterised by the emergence of Gaussian features, with logarithmic corrections at the upper critical dimension $d=d_c$.

Our purpose in this paper is to provide a general analysis of dimensions $d>d_c$ for a
wide variety of models---a black box for proving mean-field behaviour.
Models covered by our method include:
the self-avoiding walk,
Bernoulli bond percolation,
the Ising model,
the lattice $|\varphi|^4$ model (with one or two components),
the XY model, and lattice trees, all above their upper critical dimension.
The self-avoiding walk and spin models are predicted to have $d_c=4$,
for percolation the prediction is $d_c=6$, and for lattice trees it is $d_c=8$.

A fundamental object is the \emph{two-point function} $G_\beta(x)$, which for parameter
$\beta$ (e.g., inverse temperature) expresses correlation (e.g., of spins) between
the point $x\in\mathbb Z^d$ and the origin.
The two-point function typically decays exponentially for $\beta$ below
a critical value $\beta_c$, and decays algebraically at $\beta_c$.
The \emph{susceptibilty} $\chi(\beta)$ and \emph{correlation length of order two}
$\xi_2(\beta)$
are defined for $\beta< \beta_c$ by
\begin{equation}
\label{eq:xiG}
   \chi(\beta) := \sum_{x\in\mathbb Z^d} G_\beta(x),
     \qquad
    \xi_2(\beta)^2 := \frac{1}{\chi(\beta)}\sum_{x\in\mathbb Z^d} |x|_2^2 G_\beta(x),
\end{equation}
where $|x|_2:=(x_1^2+\ldots+x_d^2)^{1/2}$ is the  Euclidean norm.
For a \emph{second-order} phase transition,
three critical exponents $\eta,\gamma,\nu$ are associated with these quantities
via
\begin{equation}
    G_{\beta_c}(x) \approx \frac{1}{|x|_2^{d-2+\eta}},
    \qquad
    \chi(\beta) \approx \frac{1}{(\beta_c-\beta)^\gamma},
    \qquad
    \xi_2(\beta) \approx \frac{1}{(\beta_c-\beta)^\nu},
\end{equation}
in the limits $|x|_2 \to \infty$ (Euclidean norm) and $\beta \uparrow \beta_c$.
For the moment we leave the approximation symbol ``$\approx$'' undefined,
but we will be more careful later.  The three exponents are predicted to be
related by Fisher's scaling relation: $\gamma=(2-\eta)\nu$.  Fisher's relation arises from
the ansatz that the near-critical behaviour of the two-point function is
given by
\begin{equation}
\label{eq:Gscaling}
    G_\beta(x) \approx \frac{1}{|x|_2^{d-2+\eta}}g(|x|_2/\xi_2(\beta)),
\end{equation}
for some ``nice'' function $g$ of rapid decay.
The relation \eqref{eq:Gscaling} is
meant to capture the joint behaviour of the two-point function for $\beta$ near
$\beta_c$ and large $|x|_2$. For a fixed $\beta<\beta_c$, as $|x|_2\to\infty$,
it is in many cases proved that there is instead
\emph{Ornstein--Zernike decay} \cite{OZ14,Zern16}, i.e., with power $|x|_2^{-(d-1)/2}$; the crossover
between these two forms of decay is the subject of \cite{MS22,LS26OZ}.
The mean-field values of the critical exponents, for the aforementioned
models except lattice trees, are
\begin{equation}
    \eta=0, \qquad \gamma =1, \qquad \nu = \frac 12.
\end{equation}
For lattice trees, instead we have $\eta=0$, $\gamma =\frac 12$, $\nu = \frac 14$.

In this paper, we identify general hypotheses on a class of functions
$G_\beta: \mathbb Z^d \to [0,\infty]$ that
imply an upper bound version of \eqref{eq:Gscaling}. Namely, we obtain that
for any $\varepsilon >0$ there exist constants $c,C>0$ such that for
every $\beta \le \beta_c$ and every $x\in \mathbb Z^d$,
\begin{equation}
\label{eq:Gscaling-precise}
    G_\beta(x)
    \le
    C\delta_{0}(x)
    +
    \frac{C}{(1\vee |x|_2)^{d-2-\varepsilon}}
    \exp\left(-c\frac{|x|_2}{\xi_2(\beta)}\right).
\end{equation}
Here $\delta_0$ is the Kronecker function at $0$, i.e., $\delta_{0}(x)=1$ if $x=0$ and $\delta_{0}(x)=0$ otherwise.
Our hypotheses also imply that the susceptibility and correlation length
of order $2$ obey 
\begin{equation}
\label{eq:critical-exponents}
    \frac{c}{\beta_c-\beta} \le \chi(\beta) \le \frac{C}{\beta_c-\beta},
    \qquad
    \frac{c}{(\beta_c-\beta)^{1/2-\delta}} \le \xi_2 (\beta) \le \frac{C}{(\beta_c-\beta)^{1/2+\delta}},
\end{equation}
where $\delta$ is a small parameter we require to be present in the model.
For example, $\delta$ can be associated with the spread parameter of a spread-out model, or the
small parameter defining the weakly self-avoiding walk model, etc. The above statements can be interpreted as saying that $\eta \ge -\varepsilon$ for every $\varepsilon>0$, that $\gamma=1$, and that $\nu \in [\frac 12 -\delta,
\frac 12 +\delta]$.
Lattice trees, with their different critical exponents, are handled via a change of
variable in order to fit this framework.

These bounds are not new and are not
as strong as the best currently available results. However, the novelty lies in our method, which is radically
different from previous approaches.
Our proof:
\begin{itemize}
\item is relatively extremely simple,
\item
applies generally and essentially simultaneously to many models,
and
\item
is more connected to standard probability theory
than previous approaches as it is propelled by elementary random walk theory.
\end{itemize}

Our hypotheses on the functions $G_\beta$ involve a finite-difference inequality
for the two-point function which is
easily verifiable in practice, in any dimension. For percolation it already appeared in \cite{Hutc19_hyperbolic}, and for the Ising model it appeared in \cite{Pani24_thesis}.
It
implies an upper bound on the derivative $\partial_\beta G_\beta(x)$
which has been a standard part of the theory since the 1980s.
The finite-difference inequality is supplemented by
a lower bound on $\partial_\beta G_\beta$, which has also been a standard
part of the theory since the 1980s.  The lower bound involves a \emph{diagram} (bubble
or triangle or square diagram, depending on the model) which, to be useful, we must prove to be small.  This forces
the dimension to exceed $d_c$ in order for the diagram to be finite,
and also requires a small parameter in the
formulation of the model to make the diagram small.

The weakly self-avoiding walk and the weakly-coupled $|\varphi|^4$ model
have built-in small parameters.  For other models, a small parameter can
be introduced by considering a \emph{spread-out} version of the model.
For percolation, this was initiated in \cite{HS90a}. In the spread-out models, nearest-neighbour edges are replaced by long-range
  edges up to a distance of order $R$, where $R$ is the \emph{spread} parameter. The reciprocal $R^{-1}$
  provides the small parameter we require.
  The spread-out model is still expected to offer a valid description of the critical regime of the nearest-neighbour model thanks to the \emph{universality} hypothesis.

As a brief summary, our method applies to prove \eqref{eq:Gscaling-precise}
and \eqref{eq:critical-exponents} for the following:
\begin{itemize}
\item
nearest-neighbour weakly self-avoiding walk
and spread-out strictly self-avoiding
walk for $d>4$
(continuous-time models are included),
\item
spread-out bond percolation for $d>6$,
\item
spread-out Ising and XY models for $d>4$,
\item
nearest-neighbour weakly-coupled lattice $|\varphi|^4$ model, or spread-out
$|\varphi|^4$ model with arbitrary coupling, for $d>4$
($1$-component and $2$-component),
\item
spread-out lattice trees for $d>8$ (in this case the exponents are different as discussed above).
\end{itemize}

We discuss applications in more detail in Section~\ref{sec:statmech}.
Precise definitions of the models are given in Section~\ref{sec:applications}.

\subsection{Main result: the black box}
\label{sec:blackbox}

We now present the model-independent results.

For $x\in \mathbb R^d$, we write $|x|_2:=(x_1^2+\ldots+x_d^2)^{1/2}$ for the   Euclidean norm, $|x|=\max_{1\le i\le d}|x_1|$ for the infinity norm.  The two norms are
equivalent, but it is convenient to work with both. For $k\geq 1$, we define the box $\Lambda_k:=[-k,k]^d\cap \mathbb Z^d$,
and write $\Lambda_k(x)=\Lambda_k +x$ for the box centred at $x$.
Also, for $G:\mathbb Z^d\rightarrow [0,\infty)$, we write
\begin{equation}\label{eq:def of function norms}
	\Vert G\Vert_1:=\sum_{x\in \mathbb Z^d}G(x), \qquad \Vert |x|_2^2\cdot G\Vert_1:=\sum_{x\in \mathbb Z^d}|x|_2^2 G(x).
\end{equation}
The next definition introduces the kernels $J$ which for spin models
will be the Hamiltonian's spin-spin coupling.

\begin{Def}
\label{Def:J}
(Admissible kernel.)
An \emph{admissible kernel} is a function $J : \mathbb{Z}^d \to [0,\infty)$ with the
following properties:
\begin{enumerate}
\item[$(i)$]
The kernel $J$
satisfies $J_0=0$,
is normalised in the sense that
$\sum_{x\in\mathbb{Z}^d}J_x =1$, and is
\emph{$\mathbb{Z}^d$-symmetric} in the sense that $J_x=J_{x'}$ if $x'$ is obtained by permuting components of $x$ and/or multiplying any component $x$  by $\pm 1$.
\item[$(ii)$]
$J$ is \emph{finite-range} in the sense that
the \emph{range} $R_J$ obeys
\begin{equation}\label{eq:def range}
R_J=\max \{|x|: J_{x} >0\} \in [1,\infty).
\end{equation}
\item[$(iii)$]
The \emph{variance} of $J$ is defined by
\begin{equation}
\label{eq:Jvariance}
    \sigma_J^2 = \sum_{x\in \mathbb Z^d}|x|_2^2J_x.
\end{equation}
By definition, $\sigma_J \le R_J$. There exists a constant $c_0\in (0,1]$
such that, for all $x \in \mathbb{Z}^d$,
\begin{equation}
\label{eq:R_J comparable to xi(0)}
    c_0R_J \le \sigma_J,
    \qquad
    J_{x} \le c_0^{-1}R_J^{-d}.
\end{equation}
\end{enumerate}
\end{Def}

\begin{Ex}
Two important examples of admissible kernels are:
\begin{enumerate}
\item[$(i)$] (Nearest-neighbour model).
For every $x\in \mathbb Z^d$,
\begin{equation}
\label{eq:Jnn}
	J_x=\frac{1}{2d}\mathds{1}_{| x|_2=1}.
\end{equation}
For the nearest-neighbour model, $R_J=\sigma_J=1$ and we may take $c_0=1$.
\item[$(ii)$] (Uniform spread-out model).
For some $R\geq 1$ and for every $x\in \mathbb Z^d$,
\begin{equation}
\label{eq:Jso}
	J_x=\frac{1}{|\Lambda_R|-1}\mathds{1}_{0<|x|\leq R}.
\end{equation}
Here $R_J=R$,
$J_x \le R^{-d}$,
and there is a $c_0>0$ such that
$\sigma_J \ge c_0 R$.
More general spread-out examples are given, e.g., in \cite{HS90a,HHS03,Saka07}.
\end{enumerate}
\end{Ex}

Next, we introduce the class $\mathcal G$ of functions $G_\beta$ to which our analysis applies.
Our motivation stems from statistical mechanics where $G_\beta$ is the two-point function of some underlying model which undergoes a phase transition.
As we will see in Section~\ref{sec:applications}, it is elementary and standard that
the two-point functions of
all the models mentioned at the end of Section~\ref{sec:motivation} do belong to $\mathcal G$.
In Definition~\ref{Def:G}, we use the notation $\beta_c$ normally reserved for
a critical point.  In our applications, $\beta_c$ will indeed be a (finite)
critical point of a statistical mechanical model, but that interpretation is
not relevant until we reach Theorem~\ref{Thm:gamma-nu} below.

\begin{Def}
\label{Def:G}
We define $\mathcal G$ to be the family of functions
$G_\beta:\mathbb Z^d\rightarrow [0,\infty]$, indexed by $\beta \in [0,\beta_c]$
for some $\beta_c\in [0,\infty]$,
which satisfy the following conditions:
\begin{enumerate}
	\item[$(i)$] (Initial condition.) $G_0=\delta_{0}$.
	\item[$(ii)$] (Regularity.) For every $x\in \mathbb Z^d$, the function
    $\beta\mapsto G_\beta(x)$  is monotone
    non-decreasing and differentiable on the interval $[0,\beta_c)$.
	\item[$(iii)$] (Symmetry.) For every $\beta\geq 0$, $G_\beta$ is $\mathbb Z^d$-symmetric.
	\item[$(iv)$] (Exponential decay.)
	 For every
    $\beta\in [0,\beta_c)$, the function $x \mapsto G_\beta(x)$
    decays exponentially.
    \item[$(v)$] (Limit as $\beta\nearrow\beta_c$ when $\beta_c<\infty$.)  For every $x\in\mathbb R^d$, $\lim_{\beta \uparrow \beta_c}G_{\beta}(x) = G_{\beta_c}(x)$, but
    we do \emph{not} assume that $G_{\beta_c}(x)$ is finite.
	\end{enumerate}
\end{Def}

As we discuss in more detail in Remark~\ref{Rem:A}
below, the assumption that $G_0=\delta_0$ can be
relaxed to $G_0=A\delta_0$ with $A>0$, by a simple change of variables.
The more fundamental restrictions appear in the following assumption.
Here and later, the convolution of two
absolutely summable functions $f,g:\mathbb Z^d\rightarrow \mathbb R$ is defined by
$(f*g)(x):=\sum_{y\in \mathbb Z^d}f(y)g(x-y)$.

\begin{Ass}
\label{Ass:G}
Let $J$ be an admissible kernel and let $G\in \mathcal G$. We assume that $G$ satisfies the following:
\begin{enumerate}

	\item[$(i)$] (Finite-difference upper bound.) For every $0\leq\beta'\leq \beta<\beta_c$,
	\begin{equation}\label{eq:SL assumption}\tag{$\mathbf{I.1}$}		G_\beta\leq G_{\beta'}+(\beta-\beta')(G_{\beta'}*J*G_\beta).
	\end{equation}
    An immediate and important consequence of
\eqref{eq:SL assumption} is the differential inequality
\begin{align}
\label{eq:dGub}
    \partial_\beta G_\beta
    = \lim_{\beta' \uparrow \beta}\frac{G_\beta - G_{\beta'}}{\beta - \beta'}
    &
    \le
    G_{\beta}*J*G_\beta.
\end{align}
	\item[$(ii)$] (Differential lower bound.) For every $\beta<\beta_c$, there exists a
     function $H_\beta:\mathbb Z^d\rightarrow [0,\infty)$ (which may depend on
     $J$ and $G$) such that
    \begin{equation}\label{eq:Diff inequ assumption}\tag{$\mathbf{I.2}$}
    \partial_\beta G_\beta
    \ge
    G_\beta*(J-H_\beta)*G_\beta  .
	\end{equation}
    In addition, for all $x\in \mathbb Z^d$, $H_0(x)= 0$, $H_\beta(x)=H_\beta(-x)$, and
    $\beta\in [0,\beta_c)\mapsto H_\beta(x)$ is continuous and decays exponentially
    in $x$ for each fixed $\beta$.
\end{enumerate}
\end{Ass}

The pointwise bound \eqref{eq:Diff inequ assumption} is actually stronger than
what we require: our proofs only use the weaker assumption that \eqref{eq:Diff inequ assumption} holds when multiplied by
$|x|^p$ for $p=0,2$ and then summed over
$x\in \mathbb Z^d$.  In applications, the pointwise bound is verified instead of its
weaker summed counterparts, so we have assumed \eqref{eq:Diff inequ assumption} for simplicity.

A basic example satisfying Assumption~\ref{Ass:G} is the massive lattice Green function,
which is defined as follows.
Let $(Y_n)_{n \ge 0}$ denote the random walk on $\mathbb Z^d$ started from
$Y_0=0$ and with transition probabilities $\mathbb P_J[Y_{n+1}=y\mid Y_n=x]=J_{y-x}$,
where $J$ is an admissible kernel.
The \emph{massive lattice Green function} is defined, for $\beta \in [0,1]$ and
$x \in \mathbb Z^d$, by
\begin{equation}
\label{eq:Green-def}
  \mathbb C_\beta(x) = \sum_{n=0}^\infty \beta^n \mathbb P_J[Y_n=x].
\end{equation}
It is elementary that $\mathbb C=(\mathbb C_\beta)_{\beta\geq 0}$
satisfies the requirements of Definition~\ref{Def:G}, with
$\beta_c=1$, and therefore $\mathbb C \in \mathcal G$.
Also, as we verify in detail in Section~\ref{sec:Green},
$\mathbb C$ satisfies \eqref{eq:SL assumption}--\eqref{eq:Diff inequ assumption}
with \emph{equalities} and with $H=0$:
\begin{align}
\label{eq:Cub-intro}
    \mathbb C_\beta(x) - \mathbb C_{\beta'}(x)
    &=
    (\beta-\beta')(\mathbb C_{\beta'}*J*\mathbb C_\beta)(x),
\\
\label{eq:Clb-intro}
    \partial_\beta \mathbb C_\beta(x)
    &  =
    (\mathbb C_\beta * J * \mathbb C_\beta)(x).
\end{align}
Thus, the inequalities of Assumption~\ref{Ass:G} specify that the functions $G_\beta$
obey inequality variants of the equalities obeyed by the massive lattice Green function.

In applications, the upper bound \eqref{eq:SL assumption}
holds without further assumption in all the models
we consider, in all dimensions.  The lower bound
\eqref{eq:Diff inequ assumption} also holds in all dimensions,
with a model-dependent function $H_\beta$.  As concrete examples,
for self-avoiding walk and percolation we will see in Section~\ref{sec:applications}
that \eqref{eq:Diff inequ assumption} holds with
\begin{equation}
\label{eq:HSAWperc-intro}
    H^{\rm SAW}_\beta(x)=\delta_{0}(x)(G_\beta*J*G_\beta)(0), \qquad
    H^{{\rm{perc}}}_\beta(x) = G_\beta(x)(G_\beta*J*G_\beta)(x).
\end{equation}
The exponential decay of $G_\beta$ ensures that each of
$H^{\rm SAW}_\beta(x)$ and $H^{{\rm{perc}}}_\beta(x)$ is not only
finite but also summable,
for all $\beta < \beta_c$.
When summed over $x\in\mathbb Z^d$, these quantities are respectively the \emph{open} bubble
and \emph{open} triangle diagrams depicted in Figure~\ref{fig:H}. Summability at $\beta_c$ will be a consequence of our results, and for this
we will need to assume that $d$ is above the upper critical dimension.
Furthermore, our method requires $\sum_{x\in\mathbb Z^d}H_\beta(x)$ to be not merely finite,
but also small, and this
forces us to introduce a small parameter into a model's definition.

\begin{figure}[h]
\begin{center}
\includegraphics[scale = 1]{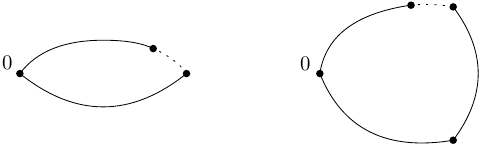}
\caption{Diagrammatic representations of the \emph{open bubble} (left) and the \emph{open triangle} (right). Unlabelled vertices are summed over $\mathbb Z^d$. Solid lines represent a two-point function $G_\beta$, and dotted lines
represent the kernel $J$. The diagrams are called ``open'' due to the dotted line.  The
 $J$ line causes the contribution from all vertices being at the origin to be zero,
as opposed to a contribution $1$ from the closed diagram without a $J$ factor. For spread-out models,
this $J$ is what produces small $\Vert H_\beta\Vert_1$
by taking $R_J$ large.}
\label{fig:H}
\end{center}
\end{figure}

Upper and lower differential inequalities as in \eqref{eq:dGub}
and \eqref{eq:Diff inequ assumption} are a well-established part of the mathematical study of phase
transitions, going back to the 1980s.  A version of the finite-difference upper bound
\eqref{eq:SL assumption} was used for percolation in \cite[Lemma~2.4]{Hutc19_hyperbolic}, and for the Ising model in \cite[Proposition~7.3.3]{Pani24_thesis}.
Here, we use \eqref{eq:SL assumption} as a fundamental and general driving principle for
our analysis.
As noted above, \eqref{eq:SL assumption}
implies the standard differential upper bound \eqref{eq:dGub}.  The verification of
\eqref{eq:SL assumption} in
applications is sometimes in the spirit of the proof of the Simon--Lieb inequality \cite{Simo80,Lieb80}.

Let $F_\beta=J*G_\beta$.
By linearity, Assumption~\ref{Ass:G} implies that $F_\beta$ satisfies
a version of \eqref{eq:SL assumption}, \eqref{eq:dGub}, and \eqref{eq:Diff inequ assumption},
namely, for $0\leq \beta'\leq \beta<\beta_c$,
\begin{align}
\label{eq:F satisfies A2}
	F_\beta
    &\leq F_{\beta'}+(\beta-\beta')(F_{\beta'}*F_{\beta}),
    \\
\label{eq:dFub}
    \partial_\beta F_\beta
    &
    \le
    F_{\beta}* F_\beta,
    \\
    \partial_\beta F_\beta &\geq F_\beta *(J-H_\beta)*G_\beta.\label{eq:dFlb}
\end{align}

For $\beta<\beta_c$, we define the \emph{susceptibility} $\chi$, and
the \emph{correlation length $\xi$ of order two} for $F_\beta$, by
\begin{align}
    \chi(\beta)& =\Vert   G_\beta \Vert_1=\Vert   F_\beta \Vert_1,
    \\
\label{eq:xi2}
    \xi (\beta)^2 &=\frac{\Vert |x|_2^2\cdot F_\beta \Vert_1}{\Vert F_\beta \Vert_1}
    =\frac{\Vert |x|_2^2\cdot F_\beta \Vert_1}{\chi(\beta)}.
\end{align}
By Assumption~\ref{Ass:G}, $G\in \mathcal{G}$ and therefore
$G_\beta$ and $F_\beta$ decay exponentially for $\beta<\beta_c$.  It follows that
both $\chi$ and $\xi$ are finite for $\beta<\beta_c$.  It is possible that
$\chi(\beta)$ diverges to $\infty$ as $\beta \uparrow \beta_c$, but we do not
assume this until Theorem~\ref{Thm:gamma-nu}.
We also define an ``error term'' by
\begin{equation}\label{eq:definition error}
	E(\beta) =
    \sup_{0\leq t\leq \beta}\left(\Vert H_t \Vert_1 + \frac{\Vert |x|_2^2\cdot H_t \Vert_1}{\xi (t)^2}\right).
\end{equation}
The supremum in \eqref{eq:definition error} is present solely
to ensure that $E(\beta)$ is increasing in $\beta$.
By definition, $E(0)=0$.
Since each of the three functions inside the supremum in \eqref{eq:definition error}
is continuous in $t\in [0,\beta_c)$ by the Dominated Convergence Theorem, the supremum
$E(\beta)$ is also continuous in $\beta \in [0,\beta_c)$.

\begin{Rem}
The conventional notation for the correlation length of order two is
$\xi_2$ (as in \eqref{eq:xiG} for $G_\beta$), whereas
the symbol $\xi$ (without subscript)
is typically used for the reciprocal of the exact exponential rate of decay
of $G_\beta(x)$ (the \emph{correlation length})
for subcritical $\beta$.  We do not use this last concept, so in order to lighten the
notation, and also to avoid notational clash with \eqref{eq:xiG},
we write $\xi$ instead of $\xi_2$ in \eqref{eq:xi2}.
\end{Rem}

When $E(\beta)$ is small, the upper and lower bounds of
\eqref{eq:dGub} and \eqref{eq:Diff inequ assumption} become closer to equalities.
It is thus natural to introduce the following quantity. For $\delta>0$, let
\begin{equation}
\label{eq:Edef}
    \beta(\delta)
    :=\sup\big\{\beta\in [0, \beta_c ) :E (\beta)<\delta\big\}.
\end{equation}
Using elementary calculus
to integrate \eqref{eq:dGub} and \eqref{eq:Diff inequ assumption}, we prove (see Proposition~\ref{prop: rough bounds chi xi}) that if $\delta <1$, then
for every $0\leq  \beta < \beta(\delta)$, we have
\begin{equation}
\label{eq:xichiintro}
	\beta(\delta) \le (1-\delta)^{-1}, \qquad \left(\frac{\chi(\beta')}{\chi(\beta)}\right)^{\frac{1+\delta}{1-\delta}}
    \leq \left(\frac{\xi(\beta')}{\xi(\beta)}\right)^2
    \leq \left(\frac{\chi(\beta')}{\chi(\beta)}\right)^{ 1-2\delta}.
\end{equation}
The second inequality in \eqref{eq:xichiintro} forms a key ingredient in our strategy: the fact that $(\xi(\beta')/\xi(\beta))^2$ remains comparable to $\chi(\beta')/\chi(\beta)$ is used to upgrade bounds that hold at $\beta'$ to bounds at a larger value $\beta$.  For the lattice Green function,
which has $E=0$, the two ratios are equal; this is a benefit of defining $\xi$ using
$F_\beta$ rather than $G_\beta$.

Since (as we have just observed) $\beta(\delta)<\infty$, it makes sense to evaluate $G_\beta$ at $\beta=\beta(\delta)$,
even if $\beta(\delta) =\beta_c$.
Our main theorem provides an upper bound for $G_\beta(x)$
for all $\beta\le \beta(\delta)$, for $\delta$ sufficiently small.
Recall that $\sigma_J$ is the variance of
$J$, defined in \eqref{eq:Jvariance}, and recall the definition of $c_0$ from \eqref{eq:R_J comparable to xi(0)}.

\begin{Thm}
\label{thm:maintheorem}
	Let $d>2$ and $\varepsilon\in(0,1)$. There exist $\bfc,\bfC>0$ depending on $(d,c_0)$,
and $\deltamain\in(0,1)$ depending on $(d,c_0,\varepsilon)$, such that, if $J$ and $G$ obey Assumption~\textup{\ref{Ass:G}}, then for every $\beta\in [0, \beta(\deltamain)]$ and every $x\in\mathbb Z^d$,
\begin{equation}
\label{eq: target estimate below beta(delta)}
	G_\beta(x)\leq \delta_0(x)
    + \frac{ \bfC}{\sigma_J^d}
    \left(\frac{\sigma_J}{\sigma_J \vee|x|} \right)^{d-2-\varepsilon}
    \exp\left(-\bfc\frac{|x|}{\xi(\beta)}\right).
\end{equation}
\end{Thm}

It is an interesting question whether
Theorem~\ref{thm:maintheorem} remains valid with $\varepsilon=0$.
We do not have an answer for the question.

Theorem~\ref{thm:maintheorem} is general and its proof is independent of
the particular choice of $J$ and $G$.  However it fails, on its own, to
cover the entire interval $\beta \in [0,\beta_c]$, since it does not
assert that $\beta(\deltamain)=\beta_c$.  The remedy for this defect
involves the following additional assumption. Its verification exploits the particular model-dependent manner
that $H_\beta$ is expressed in terms of $G_\beta$, to extract
a bound on the former from a bound on the latter.  The explicit forms for
$H_\beta$ given in \eqref{eq:HSAWperc-intro} accurately suggest that
the verification of Assumption~\ref{Ass:E} is in practice a routine task.

\begin{Ass}
\label{Ass:E}
The error term obeys $E(\beta(\deltamain)) < \deltamain$.
\end{Ass}

\begin{Thm}
\label{thm:maintheorem-betac}
Suppose that $J$ and $G$ obey both of Assumptions~\textup{\ref{Ass:G}} and \textup{\ref{Ass:E}}.
Then $\beta(\deltamain)=\beta_c$, so $\beta_c<\infty$ and
\eqref{eq: target estimate below beta(delta)}
holds for every $\beta\leq \beta_c$.
\end{Thm}

\begin{proof} By Assumption~\ref{Ass:E} we have $E(\beta(\deltamain)) < \deltamain$.
If $\beta(\deltamain)<\beta_c$, then by the continuity of $E$ on $[0,\beta_c)$, there must exist
a $\beta^*>\beta(\deltamain)$ such that $E(\beta^*) < \deltamain$.
This contradicts the definition of $\beta(\deltamain)$, so it must be the case
that $\beta(\deltamain)=\beta_c$.  This completes the proof.
\end{proof}

The following theorem is an easy consequence of our assumptions, as we will show
in Proposition~\ref{prop:gamma-nu} via integration of the assumed differential inequalities
\eqref{eq:dGub} and \eqref{eq:Diff inequ assumption}.
Its assumption that $\chi(\beta_c)=\infty$ identifies $\beta_c$ as the
\emph{critical point} in our applications to statistical mechanics.

\begin{Thm}
\label{Thm:gamma-nu}
Suppose that $J$ and $G$ obey both Assumptions~\textup{\ref{Ass:G}} and~\textup{\ref{Ass:E}}.  Suppose further that $\chi(\beta_c)=\infty$.
Then, for every $\beta < \beta_c$, and with $E=E(\beta_c)$ (which is less
than $\deltamain$), we have
 \begin{align}
 \label{eq:betacbd}
    1 &\le \beta_c \le \frac{1}{1-E},
    \\
 \label{eq:gamma}
    \frac{1}{\beta_c  -\beta}
    &\le \chi(\beta)
    \le
    \frac{1}{1-E}
    \frac{1}{\beta_c -\beta },
\\
\label{eq:nu}
    \left( \frac{1}{\beta_c -\beta} \right)^{1-2E}
    &\le
    \frac{\xi(\beta)^2}{\sigma_J^2}
    \le
    \left( \frac{1}{(1-E)(\beta_c -\beta)} \right)^{\frac{1+E}{1-E}}
    .
 \end{align}
\end{Thm}

The linear divergence of the susceptibility in \eqref{eq:gamma} is a statement
that the critical exponent $\gamma$ takes its mean-field value $\gamma=1$,
and \eqref{eq:nu} states that the critical exponent $\nu$ is within order $E(\beta_c)$
of its mean-field value $\frac 12$. In the verification of Assumption~\ref{Ass:E} for our applications in
Section~\ref{sec:applications}, for spread-out models we will find in every
case that $E(\beta_c)$ is at most of order $\sigma_J^{-d}$, and hence at most of order $R^{-d}$. This means that the powers in the upper and lower bounds on $(\xi/\sigma_J)^2$ in \eqref{eq:nu} can be brought as close to $1$ as desired by taking $R$ large enough.
With this estimate on $E(\beta_c)$, it also follows from
\eqref{eq:betacbd} that the critical value obeys $\beta_c=1+O(R^{-d})$. This $O(R^{-d})$ is optimal,  as previously observed for various models in \cite{HS05y,KS24}.

\begin{Rem}
\label{Rem:A}
It is part of Definition~\ref{Def:G} that $G_0=\delta_0$.
For some models
(e.g., the $|\varphi|^4$ model), it is instead the case that $G_0=A\delta_0$ with $A>0$ unequal to $1$. In this case, the change of variables
\begin{equation}
    \widetilde{G}_{\beta}(x) = \frac{1}{A} G_{\beta/A}(x)
\end{equation}
yields a family of functions $\widetilde G_\beta$ which do satisfy Definition~\ref{Def:G}
and Assumption~\ref{Ass:G} when $G_\beta$ satisfies them excepting $G_0=A\delta_0$.
By applying our theory to $\widetilde G$ and then undoing the change of variables,
we find that \eqref{eq: target estimate below beta(delta)}, \eqref{eq:betacbd}
and \eqref{eq:nu} are modified to:
\begin{equation}
	G_\beta(x)\leq A\delta_0(x)
    + \frac{A \bfC}{\sigma_J^d}
    \left(\frac{\sigma_J}{\sigma_J \vee|x|} \right)^{d-2-\varepsilon}
    \exp\left(-\bfc\frac{|x|}{\xi(\beta)}\right),
\end{equation}
 \begin{align}
 \label{eq:betacbdA}
    \frac{1}{A} &\le \beta_c \le \frac{1}{A(1-E)},
    \\
\label{eq:nuA}
    \left( \frac{1}{A(\beta_c -\beta)} \right)^{1-2E}
    &\le
    \frac{\xi(\beta)^2}{\sigma_J^2}
    \le
    \left( \frac{1}{(1-E)A(\beta_c -\beta)} \right)^{\frac{1+E}{1-E}}
    .
 \end{align}
Equation~\eqref{eq:gamma} remains unchanged even if $A\neq 1$.
\end{Rem}

A further consequence of Theorem~\ref{thm:maintheorem-betac} is that, under
Assumptions~\ref{Ass:G} and \ref{Ass:E}, the $p$-fold convolution of $G_{\beta_c}$
satisfies
\begin{equation}
    G_{\beta_c}^{*p}(0) < \infty \quad \text{ if } \quad
    d>p(2+\varepsilon).
\end{equation}
The condition for finiteness reduces to $d>2p$,
since $d$ is an integer and $\varepsilon$ may be chosen as small as desired.
The convolution $G_{\beta_c}^{*p}(0)$ is called the critical \emph{bubble diagram} when $p=2$,
the critical \emph{triangle diagram} when $p=3$, and the critical \emph{square diagram}
when $p=4$.
Thus, the critical bubble, triangle and square diagrams
are respectively finite when $d$ is strictly greater than $4,6,8$ (for small $\varepsilon$).
The closely related open bubble and triangle diagrams (with their extra factors $J$)
are depicted in Figure~\ref{fig:H}.

In Section~\ref{sec:applications}, we verify that each
of the six models mentioned at the end of Section~\ref{sec:motivation}  satisfies Assumptions~\ref{Ass:G} and \ref{Ass:E}.
Theorem~\ref{thm:maintheorem-betac} and its consequences therefore apply to all
six models.

\subsection{Application to statistical mechanical models}
\label{sec:statmech}

\subsubsection{Results}
\label{sec:statmechresults}

We now discuss the application of our black box to specific models.
For this, it is sufficient to verify that the model's
two-point function obeys
Definition~\ref{Def:G} (possibly with $G_0=A\delta_0$), Assumption~\ref{Ass:G}, and Assumption~\ref{Ass:E}.
For the models we consider, this is a relatively routine procedure that does
not require innovation.
Indeed, in all cases the verification of Definition~\ref{Def:G} and Assumption~\ref{Ass:E}
is based on theory dating back to the 1980s.
Assumption~\ref{Ass:G} is a bit more substantial, but also follows a well-trodden path.

We emphasise that for all but the spread-out XY, $|\varphi|^4$,
or continuous-time weakly self-avoiding walk models,
our results are not as strong as those obtained by previous
methods. We do prove new results for
these three spread-out models. Notwithstanding, our
purpose in this paper is to provide a new, unified,
relatively very easy, random walk approach to the subject.
We view Theorem~\ref{thm:maintheorem-betac} as a first
step for our methods, which upon further development may lead to more refined results for the critical regime.

We summarise our applications here.  Precise definitions of the models are
deferred to  Section~\ref{sec:applications}.
In the following, the adjective ``nearest-neighbour'' indicates that the model
is defined with the nearest-neighbour $J$ defined in \eqref{eq:Jnn}, and
``spread-out'' indicates that $J$ obeys Definition~\ref{Def:G} with a large
range $R$.

\paragraph{Self-avoiding walk.}
For precise definitions, see Sections~\ref{sec:SAW} and \ref{sec: CTWSAW}.  Let $d>4$.
Let $G_\beta$ denote the two-point function for the
nearest-neighbour weakly self-avoiding walk with repulsion strength
$\lambda\in (0,\lambda_0]$ with $\lambda_0$ sufficiently small,
or the two-point function of the spread-out strictly self-avoiding
walk with range $R$ sufficiently large.
In the latter case, $G_\beta(x)$ is a generating function for self-avoiding
walks from $0$ to $x$, whose allowed steps are from $u$ to $v$ whenever $J_{v-u}>0$,
and $\beta J_{v-u}$ is the weight of a step.
Let $\beta_c$ denote
the critical point in either case.  Then, $G_\beta(x)$ obeys the near-critical
estimate \eqref{eq: target estimate below beta(delta)} for all $\beta \in [0,\beta_c]$.
The conclusions of Theorem~\ref{Thm:gamma-nu} all hold, and the critical bubble
diagram $\sum_{x\in \mathbb Z^d}G_{\beta_c}(x)^2$ is finite.
For the weakly self-avoiding walk the error $E$ is $O(\lambda)$, and for the spread-out
strictly self-avoiding walk it is $O(R^{-d})$.
Stronger previous results obtained by other methods are discussed in
Section~\ref{sec:previous-work}.
For a continuous-time model of the weakly self-avoiding walk, our results
include both weak coupling for the nearest-neighbour model and arbitrary coupling
for spread-out models (the latter is a new result).

\paragraph{Bernoulli percolation.}
For precise definitions, see Section~\ref{sec:perc}.  Let $d>6$.
Consider spread-out Bernoulli bond percolation on $\mathbb Z^d$, with
$R$ sufficiently large. The edge set consists
of pairs $\{u,v\}$ such that $J_{v-u}>0$.   An edge $\{u,v\}$ is \emph{open}, independently of all other edges, with probability
$\beta J_{v-u}$ and otherwise is \emph{closed}.  Let
$G_\beta(x) = \mathbb P_\beta[0 \connect{} x]$ denote the probability that $0$ and
$x$ are connected by a path consisting of open bonds.
Let $\beta_c\in (0,\infty)$ denote
the critical point.  Then, $G_\beta(x)$ obeys the near-critical
estimate \eqref{eq: target estimate below beta(delta)} for all $\beta \in [0,\beta_c]$.
The conclusions of Theorem~\ref{Thm:gamma-nu} all hold
with $E=O(R^{-d})$, and the critical triangle
diagram $\sum_{x,y\in \mathbb Z^d}G_{\beta_c}(x)G_{\beta_c}(y-x) G_{\beta_c}(y)$ is finite.
In other words, the \emph{triangle condition} of \cite{AN84} holds.
This has important implications for critical exponents, e.g., it is shown
in \cite{BA91} (see also \cite{Hutc22}) that the triangle condition implies the existence and mean-field
values $\delta=2$ and $\beta=1$ for the critical exponents governing
the critical cluster-size distribution and
the percolation probability (despite the notational clash, these two exponents should not
be confused with our different usage of $\delta,\beta$).
Stronger previous results obtained by other methods are discussed in
Section~\ref{sec:previous-work}.

\paragraph{$|\varphi|^4$ model.}
For precise definitions, see Section~\ref{sec:phi4}.  Let $d>4$.
The nearest-neighbour lattice $|\varphi|^4$ model
has single-spin distribution
$\exp[-\lambda |\varphi_x|^4 - \mu|\varphi_x|^2$] with interaction
given by
$\exp[\beta\sum_{u,v}J_{v-u}(\varphi_u\cdot \varphi_v)]$.  We consider
spins with either one or two components, i.e., $\varphi_x\in \mathbb R$
or $\varphi_x\in \mathbb{R}^2$.
The two-point function is again the spin-spin correlation
$G_\beta(x)=\langle \varphi_0\cdot\varphi_x \rangle_\beta$, and the critical point
is denoted $\beta_c$.
We consider both the nearest-neighbour model with $\lambda>0$ small and any $\mu>0$\footnote{Our analysis of the weakly coupled model perturbs around $\lambda=0$,
and the restriction to positive $\mu$ ensures that the model exists when $\lambda=0$.},
and the spread-out model with any $\lambda>0$ and any $\mu\in \mathbb R$.
In both cases, we prove that
$G_\beta(x)$ obeys the near-critical
estimate \eqref{eq: target estimate below beta(delta)} for all $\beta \in [0,\beta_c]$.
The conclusions of Theorem~\ref{Thm:gamma-nu} all hold, with $E=O(\lambda)$
for the weakly-coupled nearest-neighbour model, and with $E=O(R^{-d})$ for the spread-out
model.
Some stronger previous results obtained by other methods are discussed in
Section~\ref{sec:previous-work}. Our results for the spread-out 2-component
$|\varphi|^4$ model are new.

\paragraph{Ising model.}
For precise definitions, see Section~\ref{sec:Ising}. Let $d>4$.
The spread-out Ising model is a ferromagnetic spin system with spins
taking values $\pm 1$, defined via a Gibbs measure with Hamiltonian
$H(\sigma) = -\sum_{u,v} J_{v-u} \sigma_u \sigma_v$ and inverse
temperature $\beta$.
The two-point function is the spin-spin correlation $G_\beta(x) = \langle \sigma_0
\sigma_x\rangle_\beta$.
Let $\beta_c$ denote
the critical point.  Then, $G_\beta(x)$ obeys the near-critical
estimate \eqref{eq: target estimate below beta(delta)} for all $\beta \in [0,\beta_c]$.
The conclusions of Theorem~\ref{Thm:gamma-nu} all hold with $E=O(R^{-d})$.
Stronger previous results obtained by other methods are discussed in
Section~\ref{sec:previous-work}.

\paragraph{XY model.}
For precise definitions, see Section~\ref{sec:XY}.  Let $d>4$.
The XY model is a spin model on $\mathbb Z^d$ with the single spin distribution
uniformly distributed on the unit circle in $\mathbb R^2$, with
interaction given by $\exp[\beta\sum_{u,v}J_{v-u}(\phi_u \cdot \phi_v)]$.
We take $J$ to be the spread-out kernel.
The critical point
is denoted $\beta_c$.
Then, for $R$ sufficiently large, $G_\beta(x)$ obeys the near-critical
estimate \eqref{eq: target estimate below beta(delta)} for all $\beta \in [0,\beta_c]$.
The conclusions of Theorem~\ref{Thm:gamma-nu} all hold, with $E=O(R^{-d})$.
These are new results for the XY model.

\paragraph{Lattice trees.}
For precise definitions, see Section~\ref{sec:LT}.
A lattice tree is a finite connected acyclic bond cluster.
We consider the spread-out model with $J$ given by the spread-out kernel.
Bonds are weighted by a parameter $p$, and there is a finite critical value $p_c$.
We prove a version of Theorem~\ref{thm:maintheorem} for lattice trees,
for $d>8$ and for $R$ sufficiently large.
This is then used to show that the susceptibility diverges as
$(p_c-p)^{-1/2}$ and the correlation length of order $2$ diverges as
$(p_c-p)^{-\frac 14 \pm O(R^{-d})}$.  The different critical exponents here,
compared with Theorem~\ref{Thm:gamma-nu}, are obtained via a change
of variables argument that is not needed for all the other models mentioned
above.
References to stronger previous results obtained using the lace expansion are
given in Section~\ref{sec:previous-work}.

\subsubsection{Related work}
\label{sec:previous-work}

As mentioned above, our results have been obtained earlier by other methods
in most cases, via less simple and more model-dependent methods.
We discuss this further here.

\paragraph{Reflection positivity.} In the context of the study of spin models such as
the Ising model, \emph{reflection positivity} and the \emph{infrared bound} \cite{FSS76,FILS78} provide important tools to understand the mean-field regime of the models in the regime $d\geq d_c$. These tools can be combined with various geometric representations of the models to study the near-critical or critical regime. Differential inequalities, like those in Assumption~\ref{Ass:G}, play an important role.
The substantial literature on this includes
\cite{Aize82,Froh82,AG83,AF86,AD21,FFS92,Pani26,LPS25-Ising,Pani25,DP25-Ising,EGPS26}.
The infrared bound makes it possible to study the nearest-neighbour model without
the need to introduce a small parameter.
In particular, it was recently proved in \cite{DP25-Ising} that the nearest-neighbour
Ising and $\varphi^4$ models (without any small coupling assumption) obey,
for $d>4$ and $\beta \le \beta_c$,
\begin{equation}
\label{eq:better}
    \langle \sigma_0\sigma_x \rangle_\beta
    \le
    \frac{C}{ (1\vee|x|)^{d-2}}
    \exp\big(-c(\beta_c-\beta)^{1/2}|x| \big).
\end{equation}
This improves our bound \eqref{eq: target estimate below beta(delta)}
by including the nearest-neighbour case, by not requiring a small parameter, by omitting
the $\varepsilon$ from the power law, and by obtaining the sharp power $(\beta_c
-\beta)^{1/2}$ in the exponent.
On the other hand, reliance on reflection
positivity to obtain the infrared bound is a limitation that restricts results to
nearest-neighbour interactions.

\paragraph{Lace expansion.}
The \emph{lace expansion} refers to a collection of expansion
methods which provide extremely detailed
results concerning mean-field critical phenomena above the
upper critical dimension.
The first lace expansion
was introduced by Brydges and Spencer in 1985 \cite{BS85} in the context of weakly self-avoiding walk, and was later
extended to apply to percolation, lattice trees and lattice animals, oriented percolation,
the contact process, the Ising model, and the $\varphi^4$ model; see, e.g.,
\cite{Slad06,HS90a,HS90b,Saka07,Saka14,BHH21}.
An essential distinction between the lace expansion and our method is that
the lace expansion replaces the inequalities \eqref{eq:dGub} and \eqref{eq:Diff inequ assumption}
by exact identities.  We may think of the inequalities as obtained in practice
by a single inclusion-exclusion bound.  The lace expansion obtains an identity
by performing inclusion-exclusion to \emph{all} orders.  This is necessarily a more
complicated business which---while leading to very strong results---is
much more model-dependent than our approach.

Convergence of the lace expansion requires a small parameter,
which can either be proportional to $(d-d_c)^{-1}$ for the nearest-neighbour model, or
related to $R^{-1}$ for the spread-out model.
Remarkably, with a computer-assisted proof,
the nearest-neighbour strictly self-avoiding walk has been proven to have mean-field
behaviour (including convergence to Brownian motion) in all dimensions $d \ge 5$ \cite{HS92a}.
For spread-out lattice trees in dimensions $d>8$, extensive results have
been obtained, e.g., \cite{HS90b,HS92c,DS98,CFHP23,LS25a}.
For spread-out percolation in dimensions $d>6$, and for nearest-neighbour
percolation in dimensions $d \ge 11$, mean-field behaviour has been proved   \cite{HS90a,FH17}.
A typical result is that the critical two-point function is asymptotically
equal to $c|x|^{-(d-2)}(1+o(1))$ with the amplitude $c$ given by a convergent
series.
This involves a deconvolution argument that was pioneered in \cite{HHS03,Hara08},
simplified for weakly self-avoiding walk in \cite{BHK18}, and reached its
simplest and most general form in
\cite{Slad22_lace,LS24a,LS26AIHP}.  An extension of the latter
gives the near-critical mean-field bound \eqref{eq:better} for both weakly \cite{Slad23_wsaw} and
strictly \cite{Liu25EJP} self-avoiding walk above $d_c=4$.
For near-critical percolation (either with $d \ge 11$ for the nearest-neighbour
model or with $d>6$ for the spread-out model), the sharp upper bound
\eqref{eq:better} was proven in \cite{HMS23} (see also \cite{CHS23,Vann22}).

For the continuous-time weakly self-avoiding walk, a new continuous-time lace expansion
was introduced in \cite{BHH21} to prove mean-field behaviour in dimensions $d>4$ for
sufficiently weak interactions.  We obtain complementary results here without
the lace expansion, including for spread-out models without a weak coupling
assumption.

\paragraph{Renormalisation group.}
The mean-field scaling behaviour of statistical mechanical models at the
upper critical dimension $d_c$ is typically conjectured or proved to be
modified by logarithmic factors.
The computation of these logarithmic corrections is beyond the scope of our paper.
The \emph{renormalisation group} \cite{WK74} refers
to a collection of methods which are now part of textbook theoretical
physics and which have been used, in particular,
 to compute these logarithmic corrections.  When it applies, it gives very
 detailed information.
Rigorous renormalisation group methods have been developed
for the $\varphi^4$ model in \cite{GK85,FMRS87,Hara87},
for the multi-component $|\varphi|^4$ model in \cite{BBS-phi4-log,ST-phi4,BBS-brief},
and for the continuous-time weakly self-avoiding walk in \cite{BBS-saw4-log,BBS-saw4}.
The renormalisation group approach involves a multi-scale analysis
which is entirely different from all of the other methods we are
discussing.  However, our approach is philosophically related to the renormalisation
group approach: it is multi-scale in the sense that it does involve an induction
which advances understanding of the two-point function at a value $\beta'$
to a larger value $\beta$.

 Finally, it is worth mentioning that for long-range Bernoulli percolation
 (for which $J$ has power-law decay),
 a non-perturbative analysis based on a multi-scale analysis serves
 to analyse the critical behaviour not only above and at the upper critical dimension,
 but also below the upper critical dimension \cite{Hutc26I,Hutc26II,Hutc26III}.

 \paragraph{The $\varphi_\beta$ method.}
 Recently, a novel approach to the study of the mean-field regime has been implemented to study weakly self-avoiding
walk for $d>4$ and spread-out percolation for $d>6$ \cite{DP25a,DP25b}. This method can be referred to as the ``$\varphi_\beta$ method'' due to its use of a special
quantity---usually denoted by $\varphi_\beta(S)$ for $S\subset \mathbb Z^d$---whose relevance in statistical mechanics was observed in \cite{Simo80,DT16}.
For weakly self-avoiding walk and spread-out percolation, both upper
and lower bounds as in \eqref{eq:better} are obtained.
For both models, bounds for half-space
two-point functions are obtained too.
Unlike in our black-box approach, the extension of the
$\varphi_\beta$ method from one model to another requires substantial work.

\subsection{Guide to the paper}
\label{sec:guide}

At the core of the proof of Theorem~\ref{thm:maintheorem} lies a family of random walks on $\mathbb Z^d$ called the \emph{effective random walk}, whose one-step transition
probability is $F_\beta(x)/\chi(\beta)$, with standard deviation $\xi(\beta)$.
A key feature of our analysis is the derivation of \emph{uniform} estimates for the effective random walk. These estimates are of two types: Green function estimates
derived from a \emph{regularity} property of the effective random walk and
employed to analyse the walk at and above its natural scale $\xi(\beta)$; and a \emph{stability} estimate employed to control the walk at smaller scales.
The effective random walk and its important properties are presented in Section~\ref{sec: strategy of the proof}. Proofs are deferred to Appendix~\ref{appendix: rw} (for the Green function estimates) and Section~\ref{sec:stability} (for the stability estimate).

With these preliminaries at hand,
the proof of our main result is entirely contained in Sections~\ref{sec:proof of main theorem}--\ref{sec:4props}.

In Section~\ref{sec:proof of main theorem}, we present a \emph{bootstrap} argument,
with which we prove Theorem~\ref{thm:maintheorem} conditionally on the results of Section~\ref{sec: strategy of the proof} and three additional propositions, namely
Propositions~\ref{Prop:smallx},
\ref{prop: typical scales}, and \ref{prop:intermscales}.  These three propositions
are used to complete the bootstrap at different scales.
Bootstrap arguments have played a central role in proofs of mean-field behaviour since
one was used
to prove convergence of the lace expansion in \cite{Slad87}.

In Section \ref{sec:prelim}, we first use
elementary calculus to integrate the assumed differential
inequalities \eqref{eq:dGub} and \eqref{eq:Diff inequ assumption}
(Section \ref{sec:integration}).  We then complete the proofs of the
regularity and stability properties satisfied by the effective random walk (Sections~\ref{sec:regularity-pf}--~\ref{sec:stability}). Interestingly, Section~\ref{sec:integration} is the only place where we use the
differential lower bound \eqref{eq:Diff inequ assumption}.

Section~\ref{sec:4props} provides the proof of the
Propositions~\ref{Prop:smallx},
\ref{prop: typical scales}, \ref{prop:intermscales} used in the proof of Theorem~\ref{thm:maintheorem}. This relies on elementary convolution estimates
whose proofs are deferred to Appendix~\ref{appendix:conv estimates}.

Finally, Section \ref{sec:applications} consists of a detailed verification of
Assumptions~\ref{Ass:G} and \ref{Ass:E} for the six statistical mechanical
models discussed in Section~\ref{sec:statmechresults}.

\section{Random walks: effective and regular}
\label{sec: strategy of the proof}

In this section, we first introduce the massive lattice Green function and
its important properties in Section~\ref{sec:Green}.
In Section~\ref{sec:strategy}, we define a crucial ingredient in our
analysis: the \emph{effective random walk}.
In Section~\ref{sec:regularRW}, we introduce a second key
ingredient: the concept of a \emph{regular random walk}. We state the important fact
that the effective random walk is regular, uniformly in $\beta$.
Also in Section~\ref{sec:regularRW}, we explain that regular walks obey
uniform Green function and anti-concentration estimates.
Finally, in Section~\ref{sec:stabilityintro}, we state an essential \emph{stability} estimate.

\subsection{Lattice Green function}
\label{sec:Green}

A simple prototype and point of departure
for our theory is the massive lattice Green function.
Let $(Y_n)_{n \ge 0}$ denote the random walk on $\mathbb Z^d$ started from
$Y_0=0$ and with transition probabilities $\mathbb P_J[Y_{n+1}=y\mid Y_n=x]=J_{y-x}$,
where $J$ is an admissible kernel. Recall the definition of the massive Green function $\mathbb C_\beta$ from \eqref{eq:Green-def}.

It is elementary that $\mathbb C=(\mathbb C_\beta)_{\beta\geq 0}$
satisfies the requirements of Definition~\ref{Def:G}, with $\beta_c=1$, and therefore $\mathbb C \in \mathcal G$.
Also, $\mathbb C$ satisfies Assumption~\ref{Ass:G}, as claimed in
\eqref{eq:Cub-intro}--\eqref{eq:Clb-intro} and as we verify next.
This verification is much simpler than the
model-dependent verification of Assumption~\ref{Ass:G} for statistical-mechanical
models which lack the Markov property, but it is an instructive example.

Let $\beta',\beta \in [0,1]$.
For the finite-difference upper bound \eqref{eq:SL assumption}, we use the identity
$b^n-a^n = (b-a)\sum_{j=0}^{n-1}b^ja^{n-1-j}$ and interchange summations
to see that
\begin{equation}
\label{eq:Green-SL}
    \mathbb C_\beta(x) - \mathbb C_{\beta'}(x)
    =
    (\beta-\beta')\sum_{j=0}^\infty \beta^j \sum_{m=j}^\infty (\beta')^{m-j}
    \mathbb P_J[Y_{m+1}=x].
\end{equation}
We then observe that, for $m \ge j$,
\begin{equation}
    \mathbb P_J[Y_{m+1}=x]
    =
    \sum_{y,z\in\mathbb Z^d} \mathbb P_J[Y_{j}=y] \mathbb P_J[Y_{1}=z-y] \mathbb P_J[Y_{m-j}=x -z] .
\end{equation}
After insertion of this into \eqref{eq:Green-SL}, we find that
\begin{equation}\label{eq:SLassumptionforlatticeGreen}
    \mathbb C_\beta(x) - \mathbb C_{\beta'}(x)
    =
    (\beta-\beta')(\mathbb C_{\beta'}*J*\mathbb C_\beta)(x),
\end{equation}
which is \eqref{eq:SL assumption} with equality.

For the lattice Green function, the differential inequality
\eqref{eq:Diff inequ assumption} (also \eqref{eq:dGub})
holds as an equality,
with $H_\beta =0$. Indeed, dividing \eqref{eq:SLassumptionforlatticeGreen} by $\beta-\beta'$ and taking the limit $\beta'\rightarrow \beta$ gives
\begin{equation}
	\partial_\beta \mathbb C_\beta(x)=(\mathbb C_{\beta}*J*\mathbb C_\beta)(x)
\end{equation}

Since $E(\beta)=0$ for all $\beta$, $\beta(\delta)=\beta_c$ for any $\delta>0$,
and Assumption~\ref{Ass:E} is vacuous. We could therefore be tempted to
conclude from Theorem~\ref{thm:maintheorem} that the bound
\eqref{eq: target estimate below beta(delta)} applies to the lattice Green function.
However, that would be circular, as our proof relies crucially on the
sharper bound (with no positive $\varepsilon$ in
\eqref{eq:GJxbd-intro}) stated in the following proposition.
Recall that $c_0$ is the constant occurring in assumed bounds on $J$ in
\eqref{eq:R_J comparable to xi(0)}.

\begin{Prop}
[Anti-concentration and Green function estimates for the $J$-random walk.]
\label{Prop:Green}
Let $d>2$. There exist $\cGreen,\CGreen>0$, which depend on $d$ and $c_0$ but not on $J$, such that for every $\beta\leq 1$, every $m \ge 1$, and every $x\in \mathbb Z^d$,
     \begin{align}
    \label{eq:antix-intro}
	&\mathbb P_{J}[X_m =x ]
    \leq \frac{\CGreen}{\sigma_J^d}\frac{1}{m^{d/2}}
    \exp\left(-\cGreen \frac{|x|}{\sigma_J\sqrt{m}}\right),
    \\
\label{eq:GJxbd-intro}
	&\mathbb C_\beta (x)
    \leq
    \delta_0(x)+ \frac{\CGreen}{\sigma_J^d}
    \left( \frac{\sigma_J}{\sigma_J \vee|x| } \right)^{d-2}
    \exp\left(-\cGreen\sqrt{1-\beta}\frac{|x|}{\sigma_J} \right).
\end{align}
\end{Prop}

We prove \eqref{eq:antix-intro} and \eqref{eq:GJxbd-intro}
for admissible $J$ in Section~\ref{sec:Green-pointwise} using basic random walk theory. To the best of our knowledge,
the sharp dependence of these bounds on $\sigma_J$ has not
appeared previously in the literature. It is worth mentioning the well-known fact that, for $d>2$,
at the critical value $\beta_c=1$ the Green function obeys
 $\mathbb C_1 (x) = c\sigma_J^{-2}|x|_2^{-(d-2)} + O(|x|_2^{-d})$ as $|x|_2\to\infty$,
 with $c$ depending only on $d$ \cite{LL10,Uchi98}.

\vspace{3pt}
When $\beta <1$, the bound \eqref{eq:GJxbd-intro} implies a bound on $G_\beta$
of the desired form \eqref{eq: target estimate below beta(delta)}, thanks to the finite-difference inequality in Assumption \ref{Ass:G}. Indeed, since $G_0=\delta_0$ and $F_0=J$, \eqref{eq:SL assumption} applied to $\beta'=0$ and $\beta\leq 1$ gives
\begin{equation}
	G_\beta\leq \delta_0+\beta(J*G_\beta).
\end{equation}
With $J^{*0}=\delta_0$, and $J^{*k}$ the convolution of $J$ with itself $k-1$ times
(for $k\geq 1$),
repeated iterations of the above inequality leads, for every $x\in \mathbb Z^d$, to the inequality
\begin{equation}\label{eq:GAb-bis}
	G_\beta(x)\leq \sum_{k\geq 0}\beta^kJ^{*k}(x)
    =
    \sum_{k\geq 0}\beta^k\mathbb P_J [Y_k=x] =
    \mathbb C_\beta(x).
\end{equation}
For $d>2$, the above infinite series converge when $\beta\leq 1$. We can thus obtain a bound on $G_\beta$ from \eqref{eq:GJxbd-intro}.  However, it is not quite the bound we seek in Theorem~\ref{thm:maintheorem}: the exponential rate in \eqref{eq:GJxbd-intro} involves $\sqrt{1-\beta}\sigma^{-1}=\sqrt{1-\beta}\xi(0)^{-1}$, and we instead want $\xi(\beta)^{-1}$.
The second inequality in \eqref{eq:xichiintro} allows us to make this replacement if $\beta\leq (1-\delta)\wedge \beta(\delta)$ with $\delta\in (0,\tfrac{1}{2}]$.  We prove the following proposition
in Section~\ref{sec:smallbeta}.

\begin{Prop}[Bounds on $G_\beta,F_\beta$ for $\beta<1$]
\label{prop:explicit bound on G for small beta}
Let $d>2$. Let $\cGreen,\CGreen$ be given by Proposition \textup{\ref{Prop:Green}}.
Then, for every $\delta\in (0,\tfrac{1}{2}]$, every $\beta\leq (1-\delta)\wedge \beta(\delta)$, and every $x\in \mathbb Z^d$,
\begin{align}
\label{eq:Gxbd}
	G_\beta(x)&\leq \delta_0(x)+\frac{2\CGreen}{\sigma_J^d}\left(\frac{\sigma_J}{\sigma_J\vee |x|}\right)^{d-2}\exp\left(-\frac{\cGreen}{2}\frac{|x|}{\xi(\beta)}\right),
	\\
\label{eq:Fxbd}
	F_\beta(x)&\leq \frac{2\CGreen}{\sigma_J^d}\left(\frac{\sigma_J}{\sigma_J\vee |x|}\right)^{d-2}\exp\left(-\frac{\cGreen}{2}\frac{|x|}{\xi(\beta)}\right).
\end{align}
\end{Prop}

We  think of Proposition \ref{prop:explicit bound on G for small beta} as an \emph{initialisation} step.
It explains an origin for the constants $\bfc,\bfC$ of Theorem~\ref{thm:maintheorem}: they must be ``worse'' than the constants
$\cGreen,\CGreen$ appearing in the bound \eqref{eq:GJxbd-intro} on $\mathbb C_\beta$.

The bound \eqref{eq:GAb-bis} is
useless
when $\beta>1$. To
maintain control of
$G_\beta$ in terms of random walk quantities, we replace the random walk of law $\mathbb P_J$ with a new random walk which we call the \emph{effective random walk}.

\subsection{The effective random walk}
\label{sec:strategy}

We now define the effective random walk of law $\mathbb P_\beta$,
which is a main actor in our analysis.
Throughout the section, we fix an admissible kernel $J$ together with a family of functions $G=(G_\beta)_{\beta\geq 0}\in \mathcal G$ and assume they obey Assumption \ref{Ass:G}.
Recall that $F_\beta=J*G_\beta$.

\begin{Def}[Effective random walk]\label{def: effective rw}
Let $0\leq\beta'<\beta_c$. We define a random walk $(X_k)_{k\geq 0}$ started at $X_0=0$ by the transition probability
\begin{equation}
    \mathbb P_{\beta'}[X_{1}=y]:=\frac{F_{\beta'}(y)}{\chi(\beta')}  \qquad (y\in \mathbb Z^d).
\end{equation}
We denote by $\mathbb E_{\beta'}$ the expectation with respect to $\mathbb P_{\beta'}$,
and we denote the Green function of the effective
random walk, for $Z\in [0,1]$ and $x\in \mathbb Z^d$, by
\begin{equation}
\label{eq:Green-effRW}
    \mathbb G_{Z,\beta'}(x)=\sum_{k\geq 0}Z^k \mathbb P_{\beta'}[X_k=x].
\end{equation}
\end{Def}

\begin{Rem} By definition, $\xi(\beta')^2=\mathbb E_{\beta'}[|X_1|_2^2]$.
Also, when $\beta'=0$, the measure $\mathbb P_{\beta'}$ is equal to the random walk measure $\mathbb P_J$ with step distribution $J$.
\end{Rem}

In terms of the effective random walk, the finite-difference inequality
\eqref{eq:SL assumption} can be rewritten, for every $x\in \mathbb Z^d$, as
\begin{equation}\label{eq:overview proof 0}
	G_\beta(x)\leq G_{\beta'}(x)+
    Z_{\beta',\beta}\mathbb E_{\beta'}[G_{\beta}(x-X_1)],
\end{equation}
with
\begin{equation}
	Z_{\beta',\beta}:=(\beta-\beta')\chi(\beta').
\end{equation}
Bounds on $Z_{\beta',\beta}$ can be obtained
from the differential inequalities for $G_\beta$ using elementary
calculus, as we show in Proposition~\ref{prop:Zbd}.

The following lemma presents useful iterated versions of
the finite-difference inequality \eqref{eq:SL assumption},
expressed in terms of the effective random walk with distribution $\mathbb P_{\beta'}$. It
will serve as a replacement of \eqref{eq:GAb-bis} in the regime $\beta>1$.

\begin{Lem}
\label{lem: iterated SL}
Let $0\leq \beta'\leq \beta<\beta_c$. For every $T\geq 1$ and every $x\in \mathbb Z^d$,
\begin{equation}\label{eq:SLTtimes}
	G_{\beta}(x)\leq \sum_{k=0}^{T-1}Z_{\beta',\beta}^k \mathbb E_{\beta'}[G_{\beta'}(x-X_k)]+Z_{\beta',\beta}^T\mathbb E_{\beta'}[G_{\beta}(x-X_T)].
\end{equation}
In addition, if $Z_{\beta',\beta}< 1$ then
\begin{equation}
\label{eq:SLinfty}
	G_\beta(x)\leq \sum_{k\geq 0}Z_{\beta',\beta}^k \mathbb E_{\beta'}[G_{\beta'}(x-X_k)].
\end{equation}
The same statements hold when $G$ is replaced by $F=J*G$.
\end{Lem}

\begin{proof} Let $T\geq 1$ and $x\in \mathbb Z^d$. By \eqref{eq:overview proof 0},
\begin{equation}
	G_\beta(x)\leq G_{\beta'}(x)+Z_{\beta',\beta}\mathbb E_{\beta'}[G_\beta(x-X_1)].
\end{equation}
By iteration, we obtain
\begin{equation}
	G_\beta(x)\leq \sum_{k=0}^1 Z_{\beta',\beta} \mathbb E_{\beta'}[G_{\beta'}(x-X_k)]+Z_{\beta',\beta}^2\mathbb E_{\beta'}[G_{\beta}(x-X_2)].
\end{equation}

For any $T\ge 1$,
by iterating an appropriate number of times, we obtain \eqref{eq:SLTtimes}.
Then \eqref{eq:SLinfty} follows from \eqref{eq:SLTtimes} after taking the limit $T\rightarrow \infty$. The hypothesis that $Z_{\beta',\beta}< 1$, together with the uniform boundedness of $G_\beta(\cdot)$, implies that the infinite iteration
converges.
\end{proof}

The inequality \eqref{eq:SLinfty} can be rewritten as
    \begin{equation}\label{eq:easy bound G_beta GF}
        G_\beta(x)\leq \frac{\chi(\beta')}{Z_{\beta',\beta}}\sum_{k\geq 0}Z_{\beta',\beta}^{k+1} \mathbb P_{\beta'}[X_{k+1} = x]= \frac{\chi(\beta')}{Z_{\beta',\beta}}
        \left(\mathbb{G}_{Z_{\beta',\beta},\beta'}(x)-\delta_0(x)\right).
    \end{equation}
    If an analogue of Proposition \ref{Prop:Green} could be established for the Green function of the effective random walk, we would immediately obtain a bound on $G_\beta(x)$. However, we do not  
    obtain such a result, and our bounds rather involve sums of $\mathbb{G}_{Z_{\beta',\beta},\beta'}$ over sufficiently large boxes (see Theorem \ref{thm:estimate RW} below).

Lemma~\ref{lem: iterated SL} provides a mechanism with the potential to transfer an estimate
valid at $\beta'$ to an estimate at a larger parameter $\beta\leq \beta(\delta)$.
This enables an inductive procedure that lies at the heart of the proof.
To exploit this mechanism, we will require a \emph{uniform} control over the effective random walk $\mathbb P_\beta$ for $\beta\leq \beta(\delta)$. We formulate this through two fundamental ingredients: $(i)$ Green function estimates for $\mathbb P_\beta$ which hold at scale $\xi(\beta)$; $(ii)$ a stability estimate, which controls $\mathbb P_\beta$ below the scale $\xi(\beta)$. We begin with the Green function estimates, which use  
the  concept of a \emph{regular random walk}.

\subsection{Regular random walks}
\label{sec:regularRW}

A central concept in our analysis is given in the following definition.

\begin{Def}[Regular random walks]
\label{def:regularrw}
Consider a random walk $X$ on $\mathbb R^d$ starting from 0 with law $\mathbb P$
and with variance $\sigma^2:=\mathbb E[|X_1|_2^2]<\infty$.
We assume that the distribution of
the random walk is \emph{symmetric} in the sense that
$\mathbb P[X_1 =x] = \mathbb P[X_1 =y]$ if $y$ can be obtained from $x$ by
permutation of coordinates and/or replacement of a coordinate by its negative.
Let $\cM,\CM>0$.
We say that the
random walk is $(\cM,\CM)$-{\em regular} if
it is symmetric and if the moment generating function
\begin{equation}
    M(s) = \mathbb E \left[\exp\Big(\frac{s}{\sigma}(\mathbf{e}_1 \cdot X_1)\Big)\right]
\end{equation}
satisfies
\begin{align}
    M(\cM) \le \CM.
\end{align}
We say that the random walk is \emph{regular} if it is $(\cM,\CM)$-regular for some $\cM,\CM>0$. Additionally, a family of random walks is \emph{uniformly regular} if each of the random walks in the family is regular with the same pair $(\cM,\CM)$.
\end{Def}

\begin{Rem}
The constant $\Creg$ is somewhat artificial:
what regularity requires is that the moment generating function
$M$ be analytic for $s$ in some
neighbourhood of $0$.  Once we know the analyticity,
since $M(0)=1$ we can have any $\Creg >1$ by taking $\creg$ small enough.
Nevertheless, we place stress on $\Creg$ because the constants of
Theorem~\ref{thm:estimate RW} below depend only on the dimension and on
$(\creg,\Creg)$ and on no other property of the effective random walk.
\end{Rem}

We will prove the following proposition in Section~\ref{sec:regularity-pf}.

\begin{Prop}[Regularity of the effective random walk below $\beta(\delta)$]
\label{prop: regularity below beta(delta)}
Let $\Creg =3$. There exist $\deltareg>0$ (which can be chosen equal to $2^{-9}$) and $\creg>0$ (which only depends on $d$) such that for every $\beta< \beta(\deltareg)$, the effective random walk at parameter $\beta$ is $(\creg,\Creg)$-regular.
\end{Prop}

The moment generating function
of the effective random walk
at parameter $\beta$ is denoted, for $s\in \mathbb R$, by
\begin{equation}
\label{eq:mgfX1}
	M_{\beta} (s)
    :=
    \mathbb E_\beta \left[\exp\Big(\frac{s}{\xi(\beta)}(\mathbf{e}_1 \cdot X_1)\Big)\right]
    =
    \sum_{x\in\mathbb Z^d} e^{sx_1/\xi(\beta)} \mathbb P_\beta[X_1=x].
\end{equation}
A consistency check for why
Proposition~\ref{prop: regularity below beta(delta)} could be true is the following.
If we knew that
$F_\beta(x) \le C |x|^{-(d-2)}\exp[-c |x|/\xi(\beta)]$
with $c>0$ independent of $\beta$, then we would have, for any $s<c$,
\begin{align}
\label{eq:regularity-heuristic}
    M_\beta(s) \le \frac{C}{\chi(\beta)} \sum_x e^{sx_1/\xi(\beta)}
    \frac{1}{|x|^{d-2}}e^{-c |x|/\xi(\beta)}
    \lesssim \frac{1}{\chi(\beta)} \int_0^\infty r e^{-(c-s) r/\xi(\beta)} \mathrm{d} r
    \lesssim \frac{\xi(\beta)^2}{\chi(\beta)}.
\end{align}
The upper bound on $\xi(\beta)^2/\chi(\beta)$ should in applications
 be independent of $\beta$.
However, we cannot implement this simple approach because: $(i)$
we can only prove (later) an upper bound on $G$ (or $F$)
with power $|x|^{-(d-2-\varepsilon)}$, and $(ii)$ we do not prove that
$\xi^2 \asymp \chi$ (we can only prove this up to a power $\delta$).
So the proof of Proposition~\ref{prop: regularity below beta(delta)} will be less direct.

An ingredient in the proof is the following corollary of Lemma~\ref{lem: iterated SL}.
It will allow us to transfer an estimate on $M_{\beta'}(s)$ to $M_\beta(s)$ when $\beta'<\beta$ are close together. We can then proceed inductively to show that the effective random walk is regular all the way to $\beta(\delta)$.  This induction
proof is given in Section~\ref{section: regularity}.

\begin{Coro}
\label{coro:Miteration}
Let $0\leq \beta'\leq \beta<\beta_c$ and let $s \in \mathbb R$.
If $Z_{\beta',\beta}M_{\beta'}(s\xi(\beta')/\xi(\beta))) < 1$
then
\begin{equation}
\label{eq:Mfirstbd}
    M_\beta(s)
    \le
    \frac{\chi(\beta')}{\chi(\beta)}
    \frac{M_{\beta'}(s\xi(\beta')/\xi(\beta))}{1-Z_{\beta',\beta}M_{\beta'}(s\xi(\beta')/\xi(\beta))}
    .
\end{equation}
\end{Coro}

The inequality \eqref{eq:Mfirstbd} makes apparent a need to control $\xi(\beta')/\xi(\beta)$ relative to $\chi(\beta')/\chi(\beta)$. This is achieved via the differential lower bound \eqref{eq:Diff inequ assumption}, which provides such a comparison by means of \eqref{eq:xichiintro}.

\begin{proof}[Proof of Corollary~\textup{\ref{coro:Miteration}}] By definition, by \eqref{eq:SLinfty}, and by the factorisation property
of the moment generating function of a sum of independent random variables,
\begin{align}
    M_{\beta}(s)
	&
    = \frac{1}{\chi(\beta)}
    \sum_{y\in \mathbb Z^d}F_{\beta}(y)\exp(s(\mathbf{e}_1\cdot y)/\xi(\beta) )
	\nonumber \\&
\leq \frac{\chi(\beta')}{\chi(\beta)}
\sum_{k\ge0}Z_{\beta',\beta}^k\, \mathbb E_{\beta'}[\exp(s({\bf e}_1\cdot X_{k+1})/\xi(\beta))]
    \nonumber \\
    & =\frac{\chi(\beta')}{\chi(\beta)}
    \sum_{k\ge0}Z_{\beta',\beta}^k\,
    [M_{\beta'}(s\xi(\beta')/\xi(\beta))]^{k+1}
	.
\label{eq:Miteration}
\end{align}
By assumption, the geometric series converges and the proof is complete.
\end{proof}

We use the regularity of the effective random walk
in Section~\ref{sec:proof of main theorem} to prove
our main result Theorem~\ref{thm:maintheorem}. The proof of Theorem~\ref{thm:maintheorem} relies on
the following general theorem about regular random walks, which
is of independent interest.
Its proof, which uses only classical random walk techniques,
is deferred to Appendix~\ref{appendix: rw}.
To the best of our knowledge,
Theorem~\ref{thm:estimate RW} has not appeared previously in the
literature.

We write the Green function of a generic random walk $(X_n)_{n \ge 0}$ on $\mathbb R^d$ with $X_0=0$,
for $\mu\in [0,1]$ and $A \subset \mathbb R^d$, as
\begin{equation}
    \mathbb G_\mu(A):=\sum_{m\geq 0}\mu^m \mathbb P[X_m\in A].
\end{equation}
For $x\in \mathbb R^d$, recall that
$|x|=\max_{1\leq i \leq d}|x_i|$.
For $k\geq 1$, we define the boxes
\begin{equation}\label{eq:B-box-def}
	B_k:=[-k,k]^d, \qquad B_k(x):=B_k+x,
\end{equation}
as well as their discrete counterparts,
\begin{equation}
\label{eq:Lambda-box-def}
    \Lambda_k:=[-k,k]^d\cap \mathbb Z^d , \qquad
    \Lambda_k(x)=\Lambda_k +x
\end{equation}
of radius $k$, centred at $0$ and at $x$ respectively.

\begin{Thm}[Anti-concentration and Green function estimates for regular random walks]
\label{thm:estimate RW}
Let $d>2$.
For every $\cM,\CM>0$, there exist $\mathsf{C}_{\textup{RW}}=\Crw(\cM,\CM,d)>0$ and $\crw=\crw(\cM,\CM,d)>0$ such that, for every $(\cM,\CM)$-regular random walk $X$ (started at $0$) on $\mathbb R^d$ of law $\mathbb P$,
Green function $\mathbb G$, and variance $\sigma^2$,
for every $\mu \in [0,1]$, every $m\ge 1$
and for every $x\in \mathbb Z^d$,
\begin{equation}
\label{eq:ac-estimate}
    \mathbb P[X_m\in B_\sigma(x)]
    \leq
    \frac{\Crw}{m^{d/2}}
    \exp\left(- \crw  \frac{|x|}{\sigma \sqrt{m}}\right)
    ,
\end{equation}
	\begin{equation}
\label{eq:bound massive Green function regular walk}
		\mathbb G_\mu ( B_\sigma(x))
    \leq \Crw\left(\frac{\sigma}{\sigma\vee |x|}\right)^{d-2} \exp\left(-\crw \sqrt{1-\mu} \frac{|x|}{\sigma}\right) .
	\end{equation}
\end{Thm}
\begin{Rem}

Theorem~\ref{thm:estimate RW} provides estimates for the Green function averaged
over a box whose size is the standard deviation.
Pointwise estimates do not generally hold
for regular random walks, as is illustrated in Example~\ref{ex:rw}. In particular, the regularity condition is not strong enough to obtain pointwise bounds on $G_{Z_{\beta',\beta},\beta'}$ in \eqref{eq:easy bound G_beta GF}.
\end{Rem}

By Proposition~\ref{prop: regularity below beta(delta)},
the effective random walk with $\beta<\beta(\deltareg)$ is regular.
It therefore satisfies the anti-concentration estimate \eqref{eq:ac-estimate}, and its Green function
$\mathbb G_{Z,\beta}(x)$ (defined in \eqref{eq:Green-effRW}) obeys
the averaged estimate \eqref{eq:bound massive Green function regular walk}.

\subsection{Stability below the correlation length}
\label{sec:stabilityintro}

A uniform regularity statement for the effective random walk below $\beta(\delta)$ (for $\delta$ small enough) is given in
Section~\ref{sec:regularRW}.
Proposition~\ref{prop: regularity below beta(delta)} formulates in a convenient way the fact that---as we increase $\beta$ to $\beta(\delta)$---we have uniform control over the effective random walk $\mathbb P_\beta$ for all scales larger than or equal to $\xi(\beta)$. However, we also require control of the effective random walk at
scales  below $\xi(\beta)$.
The next proposition provides this control at small scales.
Proposition~\ref{Prop:Stability section 2} is a fundamental tool in our analysis.

\begin{Prop}\label{Prop:Stability section 2} Let $d>2$. Suppose that $J$ and $G$ both obey Assumption \textup{\ref{Ass:G}}. Then there exist $\deltastab\in (0,\tfrac{1}{10}]$ and $\Cstab>0$, depending only on $d$, such that, for every $\beta<\beta(\deltastab)$ and every $\beta'\leq \beta$,
\begin{equation}\label{eq:Stability Section 2}
	\mathbb P_\beta[|X_1|\leq \xi(\beta')]\leq \Cstab \frac{\chi(\beta')}{\chi(\beta)}.
\end{equation}
\end{Prop}

To better appreciate Proposition~\ref{Prop:Stability section 2}, we rewrite its
conclusion as follows. Let $\chi_k(\beta):=\sum_{x\in \Lambda_k}F_\beta(x)$.  Then, \eqref{eq:Stability Section 2} is equivalent to
\begin{equation}
\label{eq:stability-chi}
	\chi_{\xi(\beta')}(\beta)\leq \Cstab \chi(\beta').
\end{equation}
Roughly, this inequality says that if
$\beta'\leq \beta$, then the value of $F_\beta(x)$ for $|x|\leq \xi(\beta')$ is comparable---in an averaged sense---to $F_{\beta'}(x)$.
To have some intuition for why this might be true,
suppose that
we knew that $F_\beta(x) \lesssim \sigma_J^{-2} (1\vee|x|)^{-(d-2)}$, and also
that $\chi \asymp (\xi/\sigma_J)^2$.  So armed, we could conclude that
\begin{equation}
    \chi_{\xi(\beta')}(\beta)
    \lesssim
    \sum_{|x| \le \xi(\beta')}\frac{1}{\sigma_J^2 (1\vee|x|)^{d-2}}
    \lesssim
    \left( \frac{\xi(\beta')}{\sigma_J}\right)^2
    \lesssim \chi(\beta').
\end{equation}
However, we know neither of the assumptions for the above computation,
so we will need a more circuitous inductive argument to conclude.
The proof of Proposition~\ref{Prop:Stability section 2} is given in Section~\ref{sec:stability}.

\section{Proof of Theorem~\ref{thm:maintheorem}}
\label{sec:proof of main theorem}

Our main result is Theorem~\ref{thm:maintheorem}. Theorem~\ref{Thm:gamma-nu} is then an elementary consequence of the assumptions, as
we show in Proposition \ref{prop:gamma-nu}.
The proof of Theorem~\ref{thm:maintheorem} relies on a \emph{bootstrap} argument that leverages the Green function and stability estimates for the effective random walk of Definition~\ref{def: effective rw}.

We now prove Theorem~\ref{thm:maintheorem} conditionally on the results of Section~\ref{sec: strategy of the proof} and three additional propositions, namely
Propositions~\ref{Prop:smallx}, \ref{prop: typical scales},  \ref{prop:intermscales}.
The results of Section~\ref{sec: strategy of the proof} are proved
in Section~\ref{sec:prelim} and the three propositions are
proved in Section~\ref{sec:4props}.

\subsection{The bootstrap}
\label{sec:pfmainthm}
Let $d>2$, and fix a kernel $J$ and family of functions
$G\in \mathcal G$ which satisfy Assumption \ref{Ass:G}. Our goal is to find $\bfc,\bfC,\deltamain>0$ such that the estimate \eqref{eq: target estimate below beta(delta)} holds for $\beta\leq\beta(\deltamain)$.
To prove Theorem~\textup{\ref{thm:maintheorem}}, it suffices to prove the following bound on $F_\beta=J*G_\beta$. Recall the definition of $c_0$ from \eqref{eq:R_J comparable to xi(0)}.

\begin{Thm}
\label{thm:bound on F}
Let $d>2$ and $\varepsilon\in (0,1)$.
There exist $\bfc,\bfC>0$ depending on $(d,c_0)$,
and $\deltamain\in(0,\tfrac{1}{2}]$ depending on $(d,c_0,\varepsilon)$, such that,
if $J$ and $G\in \mathcal G$ obey Assumption~\textup{\ref{Ass:G}},
then for every $\beta\in [0,\beta(\deltamain)]$ and for every $x\in \mathbb Z^d$,
\begin{equation}
\label{eq:Fmainbd}
F_\beta(x)
\leq \frac 12
\frac{\bfC}{\sigma_J^d}\left(\frac{\sigma_J}{\sigma_J\vee|x|}\right)^{d-2-\varepsilon}
\exp\left(- \bfc\frac{|x|}{\xi(\beta)}\right).
\end{equation}
\end{Thm}

\begin{proof}[Proof of Theorem~\textup{\ref{thm:maintheorem}} using Theorem~\textup{\ref{thm:bound on F}}]
We apply \eqref{eq:SL assumption} with $\beta'=0$.
This gives, for every $x\in \mathbb Z^d$,
\begin{equation}
	G_\beta(x)\leq \delta_0(x)+\beta F_\beta(x).
\end{equation}
Since $\deltamain\leq \tfrac{1}{2}$, it follows from
\eqref{eq:xichiintro}
that $\beta(\deltamain)\leq (1-\deltamain)^{-1}\leq 2$.
This factor $2$ cancels the $\frac 12$ in \eqref{eq:Fmainbd},
and results in the desired bound \eqref{eq: target estimate below beta(delta)}
on $G_\beta(x)$.
\end{proof}

We now focus on the proof of Theorem~\textup{\ref{thm:bound on F}}. Several important constants which appear throughout the analysis
are summarised in the following glossary.

\paragraph{Glossary of important constants.}
There are four pairs of important constants:
\begin{itemize}
\item
$\creg$, $\Creg$ are the constants for the moment generating function of
a regular random walk; they arise in Definition~\ref{def:regularrw}
for a generic regular random walk, and are also used to denote
the regularity constants for the effective random walk
in Proposition~\ref{prop: regularity below beta(delta)}.
The constant $\creg$ depends only on $d$, and
we set $\Creg=3$.
\item
$\crw$, $\Crw$ are the decay constants for a general regular random walk;
they arise in Theorem~\ref{thm:estimate RW}.
These constants depend only on $d$, $\creg$, and $\Creg$.
\item
$\cGreen$, $\CGreen$ are the constants in
the anti-concentration and Green function estimates for the
$J$ random walk; they arise in Proposition~\ref{Prop:Green}.
These constants depend only on $d$ and on the constant $c_0$ which controls the ratio of
$\sigma_J$ and $R_J$ as in \eqref{eq:R_J comparable to xi(0)}.
\item
$\bfc$, $\bfC$ are the decay constants for $G_\beta$ in our main result
Theorem~\ref{thm:maintheorem}, and also in Theorem~\ref{thm:bound on F}.
These constants depend only on
$d$ and $c_0$.
\end{itemize}

To prove Theorem~\ref{thm:bound on F}, we rely on a \emph{bootstrap} argument: we will show that an
{\it a priori} estimate on $F_\beta$ can be improved by using the
regularity of the effective random walk
established in Proposition~\ref{prop: regularity below beta(delta)}.

Given $x\in \mathbb Z^d$, $\beta \ge 0$, and $\bfc,\bfC>0$, we define the statements
$\mathcal H_\beta(\bfc,\bfC;x)$ and $\mathcal H_\beta(\bfc,\bfC)$
(which at this stage are unverified) by
\begin{align}
\label{eq:Hxdef}
	\mathcal H_\beta(\bfc,\bfC;x): &\quad F_{\beta}(x)\leq \frac{\bfC}{\sigma_J^d}\left(\frac{\sigma_J}{\sigma_J\vee|x|}\right)^{d-2-\varepsilon}
\exp\left(-\bfc\frac{|x|}{\xi(\beta)}\right) ,
\\
\label{eq:Hdef}
	\mathcal H_\beta(\bfc,\bfC): &\quad
    \text{$\mathcal H_\beta(\bfc,\bfC;x)$ holds for all $x\in \mathbb Z^d$}.
\end{align}

For small $\beta$, we can restate the bound on $F_\beta$ of
Proposition~\ref{prop:explicit bound on G for small beta} as follows: for $d>2$, for every $\bfC \ge 4\CGreen$, every $\bfc \le \frac 12 \cGreen$, and every $\delta \in(0, \frac 12]$,
\begin{equation}
\label{eq: bound F small values of beta in terms of C,c}
	\text{$\mathcal H_\beta(\bfc,\tfrac{1}{2}\bfC)$ holds (with $\varepsilon =0$)
    for $\beta\leq (1-\delta)\wedge\beta(\delta)$.}
\end{equation}
The fact that $\mathcal H_\beta(\bfc,\tfrac{1}{2}\bfC)$ holds for small $\beta$ serves as the \emph{initialisation step} of the bootstrap argument.  The next proposition plays the role of the \emph{contraction step}. To state it, we fix a choice of $(\bfc,\bfC)$. This choice is explained in detail in Section~\ref{sec:overview main props}. Let
\begin{equation}\label{eq:fixing bfc bfC}
    \bfc =
    \frac 12 \log 2 \wedge \frac 12 \creg \wedge \frac 12 \cGreen \wedge \frac 14 \crw ,
    \qquad
    \bfC = 16 \CGreen.
\end{equation}

\begin{Prop}[Contraction step]
\label{prop: bootstrap}
Let $d>2$ and $\varepsilon\in (0,1)$. Let $(\bfc,\bfC)$ be as in \eqref{eq:fixing bfc bfC}. There exists
$\delta_1\in (0, \delta_s\wedge \delta_{\textup{reg}}\wedge \tfrac12]$ depending on $(d,c_0,\varepsilon)$, such that the following holds. Suppose that $J$ and $G\in \mathcal G$ obey Assumption \textup{\ref{Ass:G}}.
For every $\beta<\beta(\delta_1)$,
\begin{equation}
\label{eq:bootstrap}
	\text{if $\mathcal H_{\beta'}(\bfc,\bfC)$ holds for all $\beta'\le \beta$
    then $\mathcal H_{\beta}(\bfc,\tfrac{1}{2}\bfC)$ holds.}
\end{equation}
\end{Prop}

\begin{Rem}\label{rem:bootstrap}
It follows immediately from Proposition \ref{prop: bootstrap} that, for every $\beta<\beta(\delta_1)$, if $\mathcal H_{\beta'}(\bfc,\bfC)$ holds for every $\beta'\leq \beta$, then $\mathcal H_{\beta'}(\bfc,\tfrac{1}{2}\bfC)$ holds for every $\beta'\leq \beta$.
\end{Rem}

We decompose the proof of Proposition~\ref{prop: bootstrap} into \emph{multiscale bounds}. Before introducing the multiscale bounds, we show how
\eqref{eq: bound F small values of beta in terms of C,c}
and Proposition~\ref{prop: bootstrap} can be used to prove Theorem \ref{thm:bound on F}.

\subsection{Proof of Theorem~\textup{\ref{thm:bound on F}}}

We use a \emph{forbidden region} analysis which relies on
an {\it a priori} continuity property.  We
first establish the relevant
continuity property.

As a simple initial observation, let $X$ be any
$(\creg,\Creg)$-regular random walk  with variance $\sigma^2$.
Then, by symmetry and by Markov's inequality,
\begin{align}
\label{eq:cont1}
    \mathbb P[X_1=x]
    \le
    2d\cdot
    \mathbb P[((\mathbf{e}_1 \cdot X_1)/\sigma \ge |x|/\sigma] \le
    2d \Creg \exp\left(-\creg \frac{|x|}{\sigma}\right).
\end{align}
In particular, an application of \eqref{eq:cont1} to the effective random walk shows that, for every $\beta < \beta(\delta_{\textup{reg}})$ (given by Proposition~\ref{prop: regularity below beta(delta)}) and every
$\bfc \le \frac 12 \creg$,
\begin{equation}
\label{eq:cont2}
    F_\beta(x) \le 2d \Creg \chi(\beta) \exp\left(-2\bfc \frac{|x|}{\xi(\beta)}\right).
\end{equation}
Let $\delta_1$ be given by Proposition \ref{prop: bootstrap} and $0 \le \beta<\beta(\delta_1)$.
Given $\bfc\le \frac 12 \creg$, we define
\begin{align}
        f( \beta;x) &=
        F_{\beta} (x)
        \sigma_J^d \left(\frac{\sigma_J\vee |x|}{\sigma_J}\right)^{d-2-\varepsilon}
        \exp\left(\bfc\frac{|x|}{\xi(\beta)}\right),
        \\
        f( \beta) & =\sup_{x\in \mathbb Z^d} f ( \beta;x)
        .
\label{eq:fsup}
\end{align}
By definition, $f( \beta;x)$ is continuous in $\beta\in [0,\beta(\delta_1))$.
We claim that the supremum $f(\beta)$ is
also continuous in $\beta\in [0,\beta(\delta_1))$ (in particular it is finite).
To prove the claim, it suffices to prove continuity in $\beta\in [0,\beta_1)$ for
every $\beta_1 < \beta(\delta_1)$. For $\beta\in [0,\beta_1)$, \eqref{eq:cont2} gives
\begin{align}
    f(\beta;x)
    &\le 2d \Creg \chi(\beta)\exp\left(-\bfc \frac{|x|}{\xi(\beta)}\right)\sigma_J^d \left(\frac{\sigma_J\vee |x|}{\sigma_J}\right)^{d-2-\varepsilon}
    \nonumber \\&\le
     2d \Creg \chi(\beta_1) \exp\left(-\bfc \frac{|x|}{\xi(\beta_1)}\right)
     \sigma_J^d \left(\frac{\sigma_J\vee |x|}{\sigma_J}\right)^{d-2-\varepsilon},
\label{eq:fcont}
\end{align}
where we used the facts that $\chi(\beta)\leq \chi(\beta_1)$ (by assumption on $G$) and $\xi(\beta)\leq \xi(\beta_1)$ (by \eqref{eq:xichiintro}). Note that the monotonicity of
$\xi$ is not obvious from the definition of $\xi$.
Since the right-hand side of \eqref{eq:fcont}
goes to zero as $|x|\to\infty$, uniformly in
$\beta \le \beta_1$, the supremum over $x$ in \eqref{eq:fsup}
is attained on a
$\beta$-independent (but $\beta_1$-dependent) finite set of $x$ values.
The function $f$ is therefore the supremum of
finitely many continuous functions, so it is continuous too.

\begin{proof}[Proof of Theorem~\textup{\ref{thm:bound on F}}] Let $(\bfc,\bfC)$ be as in \eqref{eq:fixing bfc bfC},
and let $\delta_1\in (0, \delta_s\wedge\delta_{\textup{reg}}\wedge\tfrac{1}{2}]$ be given by Proposition \ref{prop: bootstrap}.
We prove that $\mathcal H_\beta(\bfc,\frac 12 \bfC)$ holds, i.e., that
$f(\beta)  \le \frac 12 \bfC$,
for every $\beta\leq \beta(\delta_1)$.
We can and do assume that $\frac 12  \le \beta(\delta_1)$, since otherwise
\eqref{eq: bound F small values of beta in terms of C,c} immediately gives
the desired result.

Let $\beta_0=\frac 12 $.  By
\eqref{eq: bound F small values of beta in terms of C,c}
with $\delta=\frac 12$,
\begin{equation}
\label{eq:boot2}
    f(\beta') \le \frac 12 \bfC \quad \text{for all $\beta'\le \beta_0$.}
\end{equation}
We will prove that the combination of Proposition~\ref{prop: bootstrap}
and \eqref{eq:boot2} shows that values in $(\frac 12 \bfC, \bfC]$ are forbidden
for $f(\beta)$ for all $\beta <\beta(\delta_1)$.

The proof is by contradiction.
Suppose that there exists $\beta_* < \beta(\delta_1)$ such
that $f(\beta_*) > \frac 12 \bfC$.  By continuity of $f$
and \eqref{eq:boot2}, there must be
a $\beta_{**} \in (\beta_0,\beta_*]$ such that $f(\beta_{**})
\in (\frac 12 \bfC,\bfC]$ and also $f(\beta')\le \bfC$ for all
$\beta' \le \beta_{**}$.
But this last condition is the hypothesis in \eqref{eq:bootstrap}, and
the conclusion of \eqref{eq:bootstrap} therefore implies (see Remark \ref{rem:bootstrap}) that
$f(\beta')\le \frac 12 \bfC$ for all $\beta' \le \beta_{**}$.
This contradicts $f(\beta_{**})> \frac 12 \bfC$, so
$\beta_*$ cannot exist, and we have therefore proved that $f(\beta)\le \frac 12 \bfC$
for all $\beta < \beta(\delta_1)$.

To extend the bound $f(\beta) \le \frac 12 \bfC$ to $\beta=\beta(\delta_1)$, we argue as follows.
We know for each $x$ and each $\beta<\beta(\delta_1)$ that
\begin{equation}
\label{eq:Fgoal}
    F_\beta(x) \le \frac 12 \frac{\bfC}{\sigma_J^d}\left(\frac{\sigma_J}{\sigma_J\vee|x|}\right)^{d-2-\varepsilon} \exp\left(-\bfc \frac{|x|}{\xi(\beta)}\right).
\end{equation}
If $\beta(\delta_1)<\beta_c$, then $\xi(\beta)$ and $F_\beta(x)$ are continuous at
$\beta(\delta_1)$ by our assumptions on $G_\beta$.  Thus, \eqref{eq:Fgoal} extends by continuity to $\beta=\beta(\delta_1)$.
Suppose finally that $\beta(\delta_1)=\beta_c$ (the case of primary interest, verified under additional
hypothesis in Theorem~\ref{thm:maintheorem-betac}).
By Definition~\ref{Def:G}$(iii)$, $G_{\beta_c}(x)$
is the supremum over $\beta<\beta_c$ of $G_{\beta}(x)$, and hence
the same is true for $F_{\beta_c}(x)$. Consequently,
by the monotone convergence theorem, $\lim_{\beta\nearrow \beta_c}\chi(\beta)=\chi(\beta_c)$ and $\lim_{\beta\nearrow \beta_c}\Vert |x|_2^2F_\beta\Vert_1=\Vert |x|_2^2F_{\beta_c}\Vert_1$. If $\chi(\beta_c)<\infty$, we additionally obtain that $\lim_{\beta\nearrow \beta_c}\xi(\beta)=\xi(\beta_c)$.
On the other hand, if $\chi(\beta_c)=\infty$,
then \eqref{eq:xichiintro} implies that $1/\xi(\beta) \to
0$ as $\beta\nearrow \beta_c$. In either case, \eqref{eq:Fgoal}
extends to $\beta(\delta_1)=\beta_c$ by continuity. Thus, we conclude that $f(\beta) \le \frac 12 \bfC$ for all $\beta \le \beta(\delta_1)$.
This completes the proof, by setting $\deltamain:=\delta_1$.
\end{proof}

\begin{Rem} In the above proof, we did not really need \eqref{eq: bound F small values of beta in terms of C,c}. Indeed,  it is sufficient to know that $f$ is continuous and that $f(0)=0$.
We have highlighted \eqref{eq: bound F small values of beta in terms of C,c} because it plays a role in the proof of Proposition~\ref{prop: bootstrap}.
\end{Rem}

\subsection{Reduction of Proposition \ref{prop: bootstrap} to multiscale bounds}
\label{sec:overview main props}

For $x$ with $|x|\leq 2\xi(0)=2\sigma_J$, the exponential
factor in \eqref{eq:Fmainbd} is unimportant.
The following proposition gives a version of \eqref{eq:Fmainbd}
for these small $x$.
Proposition~\ref{Prop:smallx} is proved in Section~\ref{sec:smallscales},
as an elementary consequence of our assumptions and the Green function estimate of Proposition \ref{Prop:Green}.

\begin{Prop}[Bounds at small scales]
\label{Prop:smallx} There exists $\delta_{\textup{s}} = \delta_{\textup{s}}(d)\in (0,\tfrac12]$ such that, for every $\beta < \beta(\delta_{\textup{s}})$, and every $|x|\le 2 \xi(0)$,
    \begin{equation}
        F_{\beta}(x)
        \leq
        \frac{ 4 \CGreen}{\sigma_J^d}  \left(\frac{\sigma_J}{\sigma_J \vee |x|}\right)^{d-2}
        .
    \label{eq:smallscalesbound}
    \end{equation}
\end{Prop}

By Proposition~\ref{Prop:smallx}, if $\bfc\leq\tfrac{1}{2}\log 2$, $\bfC\geq 16\mathsf{C}$, and $\beta<\beta(\delta_s)$, then $\mathcal H_{\beta}(\bfc,\tfrac{1}{2}\bfC;x)$ holds for every $|x|\leq 2\xi(0)$. In particular, we do not need to assume the bootstrap hypothesis
that $\mathcal H_{\beta'}(\bfc,\bfC)$ holds for every $\beta'\leq \beta$ to obtain this result for small $x$. However, the bootstrap assumption becomes essential in the treatment of the scales $|x|>2\xi(0)$.

We use two different arguments depending on whether $|x|\gtrsim \xi(\beta)$ or not. This threshold corresponds to the values of $|x|$ for which the exponential term becomes relevant in \eqref{eq:Fmainbd}.

\begin{Prop}[Improvement at typical length scales]
\label{prop: typical scales} Let $d>2$ and $\varepsilon>0$.
Let $\eta >0$.
Assume that $\bfC\ge 16\CGreen$ and
$\bfc \le \frac 12 \cGreen \wedge \frac{1}{4}\crw \wedge 1$.
There exists $\delta_2=\delta_2(\eta,\varepsilon,d)\leq \delta_s\wedge \delta_{\textup{reg}}\wedge \tfrac{1}{2}$ such that the following
is true for every $\beta<\beta(\delta_2)$:
if $\mathcal H_{\beta'}(\bfc,\bfC)$ holds for every $\beta' \le \beta$,
then, for every $|x|\geq \eta \xi(\beta)\vee 2\xi(0)$,
\begin{equation}
	F_\beta(x)
    \leq
    \frac{1}{8} \frac{\bfC}{\sigma^d_J}
    \left(\frac{\sigma_J}{\sigma_J\vee |x|}\right)^{d-2-\varepsilon}
    \exp\left(-\bfc \frac{|x|}{\xi(\beta)}\right).
\end{equation}
\end{Prop}

\begin{Prop}[Improvement at intermediate scales]
\label{prop:intermscales} Let $d>2$ and $\varepsilon>0$. Assume that $\bfC\geq 16\CGreen$ and $\bfc\leq \tfrac{1}{2}\bfc\wedge \tfrac{1}{4}\crw\wedge 1$.
There exists $\eta\in (0,1)$
such that the following holds.
Let $\delta_2=\delta_2(\eta,\varepsilon,d)$ be given by
Proposition~\textup{\ref{prop: typical scales}}. The following is true for every $\beta<\beta(\delta_2)$: if $\mathcal H_{\beta'}(\bfc,\bfC)$ holds for every $\beta' \le \beta$
then for every $2\xi(0)\leq |x|\leq \eta \xi(\beta)$,
\begin{equation}
	F_\beta(x)
    \leq
    \frac{1}{6} \frac{\bfC}{\sigma_J^d}
    \left(\frac{\sigma_J}{\sigma_J\vee |x|}\right)^{d-2-\varepsilon}
    .
\end{equation}
\end{Prop}

It is not hard to deduce Proposition \ref{prop: bootstrap} from Propositions \ref{Prop:smallx}, \ref{prop: typical scales}, and \ref{prop:intermscales}.

\begin{proof}[Proof of Proposition \textup{\ref{prop: bootstrap}}]
First,
observe that the choice of $(\bfc,\bfC)$ in \eqref{eq:fixing bfc bfC} meets all the requirements of Propositions~\ref{Prop:smallx}, \ref{prop: typical scales}, and \ref{prop:intermscales}. Let $\eta,\delta_2$ be given by Proposition~\ref{prop:intermscales}. Assume that $\beta<\beta(\delta_2)$ and that $\mathcal H_{\beta'}(\bfc,\bfC)$ holds for every $\beta'\leq \beta$.
To prove that $\mathcal H_{\beta'}(\bfc,\frac 12 \bfC)$ holds,
we proceed scale by scale using
Propositions~\ref{Prop:smallx}--\ref{prop:intermscales}.

We have already observed  that $\mathcal H_{\beta}(\bfc,\tfrac{1}{2}\bfC;x)$ holds for every $|x|\leq 2\xi(0)$,
thanks to Proposition~\ref{Prop:smallx}.
For typical scales, Proposition~\ref{prop: typical scales} gives more than we need, as it says that $\mathcal H_{\beta}(\bfc,\tfrac{1}{8}\bfC;x)$ holds for every
$|x|\geq \eta \xi(\beta)\vee 2\xi(0)$.
Finally, by Proposition~\ref{prop:intermscales}, for the remaining intermediate scale $2\xi(0)\leq |x|\leq \eta\xi(\beta)$, we have
\begin{equation}
	F_\beta(x)
    \leq \frac{\exp(\bfc \eta)}{3} \cdot \frac{1}{2}\frac{\bfC}{\sigma^d_J}
    \left(\frac{\sigma_J}{\sigma_J\vee |x|}\right)^{d-2-\varepsilon}
    \exp\left(-\bfc \frac{|x|}{\xi(\beta)}\right).
\end{equation}
Since $\eta\leq 1$ and $\bfc\leq \tfrac{1}{2}\log 2$ by assumption,
this proves that $\mathcal H_{\beta}(\bfc,\tfrac{1}{3}\bfC;x)$ holds
for the intermediate scale.
This concludes the proof, by setting $\delta_1:=\delta_2$.
\end{proof}

\section{Preliminaries: integration, regularity, stability}\label{sec:prelim}

In this section, we develop three preliminary ingredients
for the proof of Theorem~\ref{thm:bound on F}.
In Section~\ref{sec:integration}, we use elementary calculus to derive consequences of the upper and lower differential inequalities \eqref{eq:dGub} and \eqref{eq:Diff inequ assumption} provided by Assumption~\ref{Ass:G}. This is the \emph{only} place in the paper where \eqref{eq:Diff inequ assumption} is used. In Section~\ref{sec:smallbeta}, we give a first application of the results of Section~\ref{sec:integration} and prove Proposition \ref{prop:explicit bound on G for small beta}. In
Section~\ref{sec:regularity-pf}, we combine
Section~\ref{sec:integration} and Corollary~\ref{coro:Miteration} to prove the regularity of the effective random walk
stated in Proposition~\ref{prop: regularity below beta(delta)}. Finally, in Section~\ref{sec:stability}, we prove the stability result
stated in Proposition~\ref{Prop:Stability section 2}. It plays an important role in the proofs of Propositions~\ref{prop: typical scales} and \ref{prop:intermscales}.

\subsection{Integration of differential inequalities}
\label{sec:integration}

Throughout Section~\ref{sec:integration}, we fix $d>2$ and
 we assume that $J$ and $G$ obey Assumption~\ref{Ass:G}.
Recall that $F_\beta = J*G_\beta$ and

\begin{equation}
	\chi(\beta)=\sum_{x\in \mathbb Z^d}G_\beta(x)=\sum_{x\in \mathbb Z^d}F_\beta(x) = \Vert F_\beta\Vert_1,
    \quad
    E(\beta)=\max_{0\leq t \leq \beta}
    \left(\Vert H_t\Vert_1 + \frac{\Vert |x|_2^2\cdot H_t\Vert_1}{\xi(t)^2}\right).
\end{equation}

It is only in this section that we use the definition of $E(\beta)$. Subsequently
we will only use the fact that $E(\beta) <\delta$ for every $\beta <\beta(\delta)$, by the definition of $\beta(\delta)$ in \eqref{eq:Edef}.

\begin{Prop}
\label{prop: rough bounds chi xi}
For every $0\leq \beta'\leq \beta < \beta_c$,
\begin{equation}\label{eq:bounds chi}
	(\beta-\beta')(1-E(\beta))\leq \frac{1}{\chi(\beta')}-\frac{1}{\chi(\beta)}\leq \beta-\beta',
\end{equation}
and, if $E(\beta)<1$,
\begin{equation}\label{eq:bounds chi xi} \left(\frac{\chi(\beta')}{\chi(\beta)}\right)^{\frac{1+E(\beta)}{1-E(\beta)}}
    \leq \left(\frac{\xi(\beta')}{\xi(\beta)}\right)^2
    \leq \left(\frac{\chi(\beta')}{\chi(\beta)}\right)^{ 1-2E(\beta)}.
\end{equation}
In particular, if $E(\beta)\le \frac 12$ then
$\xi(\beta') \le \xi(\beta)$.
\end{Prop}

By taking $\delta\in (0,1)$ and $\beta <\beta(\delta)$, the
second inequality of \eqref{eq:xichiintro} is seen to follow
immediately from \eqref{eq:bounds chi xi}.  By choosing
$\beta'=0$ and $\beta<\beta(\delta)$ in \eqref{eq:bounds chi}, we find that
$\beta (1-\delta)\leq 1/\chi(0)=1$,
which implies that
\begin{equation}\label{eq:bound on beta(delta)}
	\beta(\delta)\leq (1-\delta)^{-1},
\end{equation}
as stated in the first inequality of \eqref{eq:xichiintro}.
Also, with $\delta\in (0,1)$, $\beta'=0$, and $\beta<\beta(\delta)=\beta(\delta)\wedge (1-\delta)^{-1}$,
\eqref{eq:bounds chi} rearranges to
\begin{equation}\label{eq:chi0}
	\frac{1}{1-\beta(1-\delta)}\leq \chi(\beta)\leq \frac{1}{1-\beta}.
\end{equation}

\begin{proof}[Proof of Proposition~\textup{\ref{prop: rough bounds chi xi}}]
We begin with \eqref{eq:bounds chi}.
It follows by summation of the differential upper bound \eqref{eq:dGub} that
\begin{equation}
\label{eq:dchi-ub}
	\partial_\beta \chi(\beta)\leq \sum_{x\in \mathbb Z^d}(F_\beta*G_\beta)(x)=\chi(\beta)^2.
\end{equation}
Similarly, it follows from summation of the differential lower bound \eqref{eq:Diff inequ assumption} that
\begin{equation}
\label{eq:dchi-lb}
	\partial_\beta \chi(\beta)
    \geq \sum_{x\in \mathbb Z^d}(F_\beta*[J-H_\beta]*G_\beta)(x)
    =\left(1- \Vert H_\beta \Vert_1 \right)\chi(\beta)^2
    \geq (1-E(\beta))\chi(\beta)^2.
\end{equation}
The combination of \eqref{eq:dchi-ub}--\eqref{eq:dchi-lb} gives
\begin{equation}\label{eq:proof basic prop0}
	(1-E(\beta))\leq -\partial_\beta\chi(\beta)^{-1} \leq 1,
\end{equation}
and then integration over the interval $[\beta',\beta]$ gives \eqref{eq:bounds chi}.
For later use, we also observe that \eqref{eq:dchi-ub}--\eqref{eq:dchi-lb} yield
\begin{equation}
\label{eq:bounds log der}
	\partial_\beta\log\chi(\beta)\leq \chi(\beta)\leq \frac{1}{1-E(\beta)}\partial_\beta\log\chi(\beta).
\end{equation}

To begin the proof of \eqref{eq:bounds chi xi}, we note that
by definition and by \eqref{eq:dFub},
\begin{equation}
	\partial_\beta(\chi(\beta)\xi(\beta)^2)
    =\sum_{x\in \mathbb Z^d}|x|_2^2 \partial_\beta F_\beta(x)
    \le
    \sum_{x,y\in \mathbb Z^d}|x|_2^2 F_\beta(y)F_\beta(x-y).
\end{equation}
We insert $|x|_2^2 = |y|_2^2 + 2x\cdot(x-y) + |x-y|_2^2$ in the right-hand side,
and observe that the cross term vanishes by symmetry of $F_\beta$.
Therefore,
\begin{equation}\label{eq:proof basic prop2}
	\partial_\beta(\chi(\beta)\xi(\beta)^2)\leq 2\chi(\beta)^2\xi(\beta)^2.
\end{equation}
We apply the lower bound of
\eqref{eq:Diff inequ assumption} in a similar manner.
With the inequality
\begin{equation}\label{eq:ineq l2 G F}
    \| |x|^2 G_\beta \|_1 = \| |x|^2 F_\beta \|_1 - \sigma_J^2\chi(\beta)
    \le \| |x|^2 F_\beta \|_1,
\end{equation}
and with our assumption that $H_\beta(x)=H_\beta(-x)$, this leads to
\begin{align}
	&\partial_\beta(\chi(\beta)\xi(\beta)^2) \nonumber
    \\&\quad \geq \sum_{x\in \mathbb Z^d}|x|_2^2
    \left[(F_\beta * F_\beta)(x) - (F_\beta * H_\beta *G_\beta)(x)\right]
        \nonumber
    \\&\quad =2\chi(\beta)^2\xi(\beta)^2-\chi(\beta)^2\Vert |x|_2^2\cdot H_\beta\Vert_1-\chi(\beta)^2\Vert H_\beta\Vert_1 \xi(\beta)^2-\chi(\beta)\Vert H_\beta\Vert_1\cdot\Vert |x|_2^2\cdot G_\beta\Vert_1
    \nonumber
    \\
    &\quad \ge 2\chi(\beta)^2\xi(\beta)^2-\chi(\beta)^2\Vert |x|_2^2 \cdot H_\beta\Vert_1
    - 2 \Vert H_\beta\Vert_1 \chi(\beta)^2\xi(\beta)^2
    \nonumber
    \\ &\quad
    \ge
    2(1-E(\beta))\chi(\beta)^2\xi(\beta)^2.
\label{eq:proof basic prop3}
\end{align}
Together, \eqref{eq:proof basic prop2} and \eqref{eq:proof basic prop3} give
\begin{equation}\label{eq:proof basic prop4}
	2(1-E(\beta))\chi(\beta)\leq \partial_\beta \log(\chi(\beta)\xi(\beta)^2)\leq 2\chi(\beta).
\end{equation}
With \eqref{eq:bounds log der}, \eqref{eq:proof basic prop4} gives
\begin{equation}
	2(1-E(\beta))\partial_\beta\log \chi(\beta)\leq \partial_\beta\log(\chi(\beta)\xi(\beta)^2)\leq \frac{2}{1-E(\beta)}\partial_\beta \log \chi(\beta).
\end{equation}
Then, \eqref{eq:bounds chi xi} follows after integration over $[\beta',\beta$].
This completes the proof.
\end{proof}

The following corollary of Proposition~\ref{prop: rough bounds chi xi}
shows that the ratios $\frac{\chi(\beta')}{\chi(\beta)}$ and $(\frac{\xi(\beta')}{\xi(\beta)})^2$ are comparable up to constants,
as long as $\beta-\beta'$ is sufficiently small that the
ratios $\frac 14 \frac{\chi(\beta')}{\chi(\beta)}$ and
$\frac 14 (\frac{\xi(\beta')}{\xi(\beta)})^2$ are not smaller than $E(\beta)$.
\begin{Coro}
\label{cor: comparison chi xi when E small}	
Let $0\leq \beta'\leq \beta< \beta_c$.
If $E(\beta)\leq \frac{1}{4}\frac{\chi(\beta')}{\chi(\beta)} \vee \frac{1}{4}(\frac{\xi(\beta')}{\xi(\beta)})^2$ and also $E(\beta)\leq \tfrac{1}{2}$, then
\begin{equation}
\label{eq: comparison chi xi when E small}
	\frac{1}{2}\frac{\chi(\beta')}{\chi(\beta)}
    \leq
    \left(\frac{\xi(\beta')}{\xi(\beta)}\right)^2\leq 2\frac{\chi(\beta')}{\chi(\beta)}.
\end{equation}
\end{Coro}

\begin{proof}
The claim follows from \eqref{eq:bounds chi xi} and basic algebra.
Let $a=(\frac{\xi(\beta')}{\xi(\beta)})^2$, $b=\frac{\chi(\beta')}{\chi(\beta)}$,
and $E=E(\beta)$.  Then $0 < b \leq 1$, and since $E\leq \tfrac{1}{2}$
it follows from Proposition~\ref{prop: rough bounds chi xi} that also $0<a\leq 1$.  We will use the facts that
\begin{equation}
\label{eq:bounds-for-ab}
    t^t \ge \frac 12 \;\;\text{and}\;\; t^{t/2} \ge \frac 12 \;\;\text{for $t >0$,}
    \qquad
    \frac{1}{1-t} \le 1 +2t \;\;\text{for $t \in [0, \tfrac 12]$.}
\end{equation}
The inequalities in \eqref{eq:bounds chi xi} can be restated as
\begin{equation}\label{eq:ab0}
    b\cdot b^{2(\frac{1}{1-E}-1)} \le a \le b\cdot \frac{1}{b^{2E}}.
\end{equation}
Suppose first that $E \le \frac 14 b$.  Then, a combination of \eqref{eq:bounds-for-ab} and \eqref{eq:ab0} gives
the required bounds
\begin{equation}
 a \ge b \cdot b^{2(\tfrac{1}{1-E}-1)} \ge
b\cdot b^{4E}\ge b \cdot b^b \ge
    b\cdot \frac 12
\end{equation}
 and
   \begin{equation}
   	a\leq b\cdot \frac{1}{b^{2E}}\leq b\cdot \frac{1}{b^{b/2}}\leq 2b.
   \end{equation}
Now suppose that $E \le \frac 14 a$.  In this case we restate \eqref{eq:bounds chi xi} as
\begin{equation}
\label{eq:ab1}
    a\cdot a^{\frac{1}{1-2E}-1} \le b \le a\cdot \frac{1}{a^{\frac{2E}{1+E}}}.
\end{equation}
In the lower bound on $b$ in
\eqref{eq:ab1},
the exponent of the second factor is at most $4E \le a$, so the left-hand side
is bounded below by $\frac 12 a$.  The right-hand side of \eqref{eq:ab1} is
increased if we increase the power to $2E \le \frac 12 a$, so the right-hand side
is bounded above by $2a$.  Together, this gives the desired bounds $\frac 12 b \le a \le 2b$ in
this case, and completes the proof.
\end{proof}

Recall that
\begin{equation}
\label{eq:Zdef1}
Z_{\beta',\beta}=
(\beta-\beta')\chi(\beta').
\end{equation}
In particular,
\begin{equation}
\label{eq:Z0}
        Z_{0,\beta} = \beta .
\end{equation}
The following proposition shows that $Z_{\beta',\beta}$ is close to and less than $1$
as long as $\beta-\beta'$ is sufficiently small that the
ratios $\frac 14 \frac{\chi(\beta')}{\chi(\beta)}$ and
$\frac 18 (\frac{\xi(\beta')}{\xi(\beta)})^2$ are not smaller than $E(\beta)$. The upper bound on $Z_{\beta',\beta}$ is the useful one in our analysis.

\begin{Prop}[Bounds on $Z_{\beta',\beta}$]
\label{prop:Zbd}
Let $0\leq \beta'\leq \beta<\beta_c$. Then,
\begin{equation}
\label{eq: bounds on Z1}
	1-\frac{\chi(\beta')}{\chi(\beta)}\leq Z_{\beta',\beta}\leq \Big(1-\frac{\chi(\beta')}{\chi(\beta)}\Big)\frac{1}{1-1\wedge E(\beta)}.
\end{equation}
In particular, if $E(\beta)\leq \tfrac{1}{4}\tfrac{\chi(\beta')}{\chi(\beta)}$, then
\begin{equation}\label{eq: explicit bounds on Z1}
	1-\frac{\chi(\beta')}{\chi(\beta)}\leq Z_{\beta',\beta}\leq 1-\frac{1}{2}\frac{\chi(\beta')}{\chi(\beta)}.
\end{equation}
If we assume instead
that $E(\beta)\leq \tfrac{1}{8}(\tfrac{\xi(\beta')}{\xi(\beta)})^2 \wedge \tfrac{1}{2}$, then
\begin{equation}
\label{eq: bounds on Z in terms of n/N}
	1-2\left(\frac{\xi(\beta')}{\xi(\beta)}\right)^2\leq Z_{\beta',\beta}\leq 1-\frac{1}{4}\left(\frac{\xi(\beta')}{\xi(\beta)}\right)^2.
\end{equation}
\end{Prop}

 \begin{proof} The upper bound in \eqref{eq:bounds chi} implies that
 \begin{equation}
 	1-\frac{\chi(\beta')}{\chi(\beta)}\leq (\beta-\beta')\chi(\beta')=Z_{\beta',\beta},
 \end{equation}
 which is the lower bound in \eqref{eq: bounds on Z1}.
 The upper bound in \eqref{eq: bounds on Z1} is vacuous if $E(\beta)\geq 1$,
  while for $E(\beta)<1$ the lower bound in \eqref{eq:bounds chi} yields
 \begin{equation}
 	(1-E(\beta))Z_{\beta',\beta}
    \leq 1-\frac{\chi(\beta')}{\chi(\beta)},
 \end{equation}
which is the upper bound in \eqref{eq: bounds on Z1}.

For the upper bound of \eqref{eq: explicit bounds on Z1},
we again use $(1-t)^{-1}\leq 1+2t$ for $t\leq \tfrac{1}{2}$,
so that \eqref{eq: bounds on Z1} and the assumption on $E(\beta)$ give
 \begin{equation}
 	Z_{\beta',\beta}\leq \left(1-\frac{\chi(\beta')}{\chi(\beta)}\right)\left(1+\frac{1}{2}\frac{\chi(\beta')}{\chi(\beta)}\right)\leq 1-\frac{1}{2}\frac{\chi(\beta')}{\chi(\beta)}.
 \end{equation}
Finally, for \eqref{eq: bounds on Z in terms of n/N}, we
apply Corollary~\ref{cor: comparison chi xi when E small} to see that
$E(\beta) \le \frac 18 (\frac{\xi(\beta')}{\xi(\beta)})^2 \leq \frac 14 \frac{\chi(\beta')}{\chi(\beta)}$, and then apply
\eqref{eq: explicit bounds on Z1} and again use the bounds
of Corollary~\ref{cor: comparison chi xi when E small}. This completes the proof.
\end{proof}

Theorem~\ref{Thm:gamma-nu} is contained in the following proposition.

\begin{Prop}
\label{prop:gamma-nu}
For $\delta \in(0,1)$, $\beta \le \beta(\delta)$, and   $E=E(\beta(\delta))$, we have
\begin{equation}
\label{eq:chibetadelta}
    \frac{1}{\chi(\beta(\delta))^{-1} + (\beta(\delta) -\beta)}
    \le \chi(\beta)
    \le
    \frac{1}{\chi(\beta(\delta))^{-1} + (\beta(\delta) -\beta)(1-E)},
\end{equation}
\begin{equation}
\label{eq:xibetadelta}
    \chi(\beta)^{1-2E}
    \le
    \frac{\xi(\beta)^2}{\sigma_J^2}
    \le
     \chi(\beta)^{\frac{1+E}{1-E}}.
\end{equation}
Assume that $\beta(\delta)=\beta_c$ and $\chi(\beta_c)=\infty$. Then,
 \begin{equation}
 \label{eq:betac-asy}
    1 \le \beta_c \le \frac{1}{1-E},
 \end{equation}
\begin{equation}
\label{eq:chibetac}
    \frac{1}{\beta_c  -\beta}
    \le \chi(\beta)
    \le
    \frac{1}{1-E}
    \frac{1}{\beta_c -\beta },
\end{equation}
\begin{equation}
\label{eq:xibetac}
    \left( \frac{1}{\beta_c -\beta} \right)^{1-2E}
    \le
    \frac{\xi(\beta)^2}{\sigma_J^2}
    \le
    \left( \frac{1}{(1-E)(\beta_c -\beta)} \right)^{\frac{1+E}{1-E}}.
\end{equation}
\end{Prop}

\begin{proof}
The inequality \eqref{eq:chibetadelta}
follows by replacing $(\beta',\beta)$ by $(\beta,\beta(\delta))$
in \eqref{eq:bounds chi},
and \eqref{eq:xibetadelta} follows from \eqref{eq:bounds chi xi} with $\beta'=0$.
To prove \eqref{eq:betac-asy}, we set $\beta'=0$ and
take the limit $\beta \uparrow \beta(\delta)=\beta_c$ in \eqref{eq:bounds chi}.
The inequalities \eqref{eq:chibetac}--\eqref{eq:xibetac} follow
by setting $\chi(\beta_c)=\infty$ in \eqref{eq:chibetadelta}--\eqref{eq:xibetadelta}.
\end{proof}

\subsection{Proof of Proposition~\ref{prop:explicit bound on G for small beta}: bounds on $G_\beta$ and $F_\beta$ for $\beta<1$}
\label{sec:smallbeta}

We now prove the bounds on $G_\beta$ and $F_\beta$ for $\beta<1$
stated in Proposition~\ref{prop:explicit bound on G for small beta}.

\begin{proof}[Proof of Proposition~\textup{\ref{prop:explicit bound on G for small beta}}] Let $\delta \in (0,\tfrac{1}{2}]$
and $\beta\leq (1-\delta)\wedge \beta(\delta)$. We start with the bound on $G_\beta(x)$. It follows from \eqref{eq:GAb-bis} and Proposition~\ref{Prop:Green} that, for every $x\in \mathbb Z^d$,
\begin{equation}
\label{eq:CAbeta}
    G_\beta(x) \le  \mathbb C_{\beta}(x)
    \leq \delta_0(x)+ \frac{\CGreen}{\sigma^d}
    \left( \frac{\sigma}{\sigma \vee|x| } \right)^{d-2}
    \exp\left(-\cGreen\sqrt{1-\beta}\frac{|x|}{\xi(0)} \right).\end{equation}
The second inequality in
\eqref{eq:bounds chi xi} gives that
\begin{equation}
\label{eq:xiratio}
	\frac{\xi(\beta)}{\xi(0)}\geq \left(\frac{\chi(\beta)}{\chi(0)}\right)^{\tfrac{1}{2}-\delta}.
\end{equation}
Combining \eqref{eq:chi0} with the inequality $1-\delta \ge \beta$ yields
\begin{equation}
\label{eq:chiratio}
	\frac{\chi(\beta)}{\chi(0)}\geq \frac{1}{1-(1-\delta)\beta}
    \geq \frac{1}{1-\beta^2}.
\end{equation}
Insertion of \eqref{eq:chiratio} into \eqref{eq:xiratio} gives
\begin{align}
	\sqrt{1-\beta}\cdot\frac{\xi(\beta)}{\xi(0)}
    &\geq (1-\beta)^{\delta}\left(\frac{1-\beta}{1-\beta^2}\right)^{\tfrac{1}{2}-\delta}
	\nonumber\\&\geq\delta^{\delta}\left(\frac{1}{1+\beta}\right)^{\tfrac{1}{2}-\delta}
	 \geq \delta^{\delta}\left(\frac{1}{2-\delta}\right)^{\tfrac{1}{2}-\delta}
	 \geq
   \frac 12
    ,
\label{eq:proof explicit bound G 2}
\end{align}
where last inequality follows from elementary calculus to
minimise the penultimate expression. Combining \eqref{eq:CAbeta}, and \eqref{eq:proof explicit bound G 2},
we obtain, for every $x\in \mathbb Z^d$,
\begin{equation}
\label{eq:proof explicit bound G 3}
	G_\beta(x)
    \le \mathbb C_\beta(x)\leq
    \delta_0(x)+\frac{\CGreen}{\sigma^d}
    \left( \frac{\sigma}{\sigma \vee|x| } \right)^{d-2}
    \exp\left(-\frac{\cGreen}{2} \frac{|x|}{\xi(\beta)} \right).
\end{equation}
This completes the proof of \eqref{eq:Gxbd}.

Next we consider $F_\beta = J*G_\beta$.  By \eqref{eq:GAb-bis},
\begin{equation}
\label{eq:proof explicit bound G 3.5}
	F_\beta(x)
    \leq
    (J* \mathbb C_{\beta})(x)
    =
    \frac{1}{\beta}(\mathbb C_{\beta}(x)-\delta_0(x)).
\end{equation}
For the case $\beta \ge \frac 12$, we immediately obtain \eqref{eq:Fxbd}
from the bound on $\mathbb C_{\beta}(x)$ in \eqref{eq:proof explicit bound G 3},
with $\bfCo=2\CGreen$.
For the remaining case $\beta < \frac 12$, it follows from
the monotonicity of $F_\beta$, \eqref{eq:proof explicit bound G 3.5}, and
the bound on $\mathbb C_{1/2}(x)$ from
\eqref{eq:GJxbd-intro} that
\begin{align}
\notag
	F_{\beta}(x)\leq F_{1/2}(x)&\leq 2 (\mathbb C_{1/2}(x)-\delta_0(x))
	\\&\leq 2\frac{\CGreen}{\sigma^d}\left( \frac{\sigma}{\sigma \vee|x| } \right)^{d-2}
\exp\Big(-\cGreen\frac{1}{\sqrt{2}}\frac{|x|}{\sigma} \Big).
\label{eq:proof explicit bound G 4}
\end{align}
Since $\xi(\beta) \ge \xi(0)=\sigma$ by \eqref{eq:xiratio},
we then obtain \eqref{eq:Fxbd} with $\bfCo=2\CGreen$ and
$\bfco=\frac 12 \cGreen$, from the observation that
\begin{equation}
    \frac{1}{\sqrt{2}}\frac{1}{\sigma}
    =
    \frac{1}{\sqrt{2}}\frac{1}{\xi(0)} \ge \frac 12 \frac{1}{\xi(\beta)}.
\end{equation}
This concludes the proof.
\end{proof}

\subsection{Regularity of the effective random walk}\label{section: regularity}
\label{sec:regularity-pf}

We now prove Proposition~\ref{prop: regularity below beta(delta)}, which states
that the effective random walk introduced in
Definition~\ref{def: effective rw} is uniformly regular for $\beta < \beta(\deltareg)$ with $\deltareg$ sufficiently small. Consequently, for every $ \beta< \beta(\deltareg)$,
the effective random walk
obeys the anti-concentration and Green function estimates of Theorem~\ref{thm:estimate RW}, with uniform constants $(\crw,\Crw)$ that depend only on $d$.
The fact that $(\creg,\Creg)$ and
$(\crw,\Crw)$ do not depend on  $J$ or $\beta$ is essential for our proof.

\begin{proof}[Proof of Proposition~\textup{\ref{prop: regularity below beta(delta)}}]
The effective random walk $X$ with law $\mathbb P_\beta$ is symmetric by definition.
Let
\begin{equation}
    L:= 2^7 = 128.
    \end{equation}
Let $\delta>0$, to be chosen small enough. In particular, we assume that
$\delta\leq \frac{1}{4L}$.
Fix $\beta< \beta(\delta)$.  Then $\chi(\beta)<\infty$.
Recall that
\begin{equation}
    M_\beta(s)
    =
    \mathbb E_\beta[\exp(s( \mathbf{e}_1\cdot X_1)/\xi(\beta))].
\end{equation}
It suffices to prove that we can choose $\delta,c_1>0$ such
that $M_\beta(c_1)\leq 3$ for every $\beta< \beta(\delta)$.

We define a finite integer $K$ by
\begin{equation}
    K := \min\Big\{k \ge 0 : \text{there exists $\beta_0 < (1-\textstyle\frac{1}{2L})$
    such that $\chi(\beta_0)=L^{-k}\chi(\beta)$}\Big\}.
\end{equation}
With $\beta_0$ so defined, for $0 \le k \le K$ we introduce $\beta_k$ by setting
\begin{equation}
    \chi(\beta_k) = L^k \chi(\beta_0).
\end{equation}
In particular, $\beta_K=\beta$.
We will prove by induction on $k$ that $M_{\beta_k}(c_1) \le 3$,
for appropriate $\delta,c_1>0$.

To start the induction, we prove that $M_{\beta_0}(c_1) \le 3$,
with $c_1>0$ to be chosen.
By Proposition~\ref{prop: rough bounds chi xi} and the fact that
(by definition of $\beta(\delta)$)
\begin{equation}
\label{eq:Edelta}
    E(\beta)\leq \delta\leq \frac{1}{4L}\leq \frac{1}{2},
\end{equation}
we have $\xi(\beta_0)\geq \xi(0)=\sigma_J$.
Let $\lambda_0 =\exp[c_1/c_0]$ where $c_0$ is given by \eqref{eq:R_J comparable to xi(0)}.
Then $M_0(c_1\sigma_J/\xi(\beta_0))\leq \exp(c_1\tfrac{\sigma_J}{\xi(\beta_0)}\tfrac{R_J}{\sigma_J})\leq \lambda_0$.
We require that $c_1=c_1(c_0,L,d)>0$ be sufficiently small that
$\lambda_0 \leq 1+\tfrac{1}{4L}$.
Then $\beta_0\lambda_0< (1-\tfrac{1}{2L})(1+\tfrac{1}{4L}) < 1$
and hence, by Corollary~\ref{coro:Miteration},
\begin{equation}
	M_{\beta_0}(c_1)\leq \frac{1}{\chi(\beta_0)}\frac{M_0(c_1\sigma_J/\xi(\beta_0))}{1-\beta_0M_0(c_1\sigma_J/\xi(\beta_0))}
\leq \frac{\lambda_0}{\chi(\beta_0)}\frac{1}{1-\beta_0 \lambda_0}.
\end{equation}
Also, since $\delta\leq \frac{1}{4L}$, \eqref{eq:chi0} implies that
\begin{equation}
	\frac{1}{\chi(\beta_0)}
    =
    \frac{\chi(0)}{\chi(\beta_0)}\leq 1-\beta_0(1-\delta)
    \leq 1-\beta_0\left(1-\frac{1}{4L}\right).
\end{equation}
As a result, we have
\begin{align}
\label{eq:proof reg1}
    M_{\beta_0}(c_1) & \le
    \left(1+\frac{1}{4L}\right)\max_{t\in [0,1-\tfrac{1}{2L}]}
    \frac{1-(1-\tfrac{1}{4L})t}{1-(1+\tfrac{1}{4L})t}
    \nonumber \\
    & =
    \left(1+\frac{1}{4L}\right)
    \frac{1-(1-\tfrac{1}{4L})(1-\tfrac{1}{2L})}{1-(1+\tfrac{1}{4L})(1-\tfrac{1}{2L})}\leq 3.
\end{align}

To advance the induction, we assume that $M_{\beta_k}(c_1) \le C_1$
and prove that the same bound holds when $k$ is replaced by $k+1$.
To abbreviate the notation, we write
\begin{equation}
    Z_k = Z_{\beta_k,\beta_{k+1}}.
\end{equation}
Let
\begin{equation}
    r_k := \frac{\xi(\beta_k)}{\xi(\beta_{k+1})}.
\end{equation}
By Corollary~\ref{coro:Miteration}, if $Z_k M_{\beta_k}(c_1r_k) < 1$ then
\begin{equation}
\label{eq:Msum}
    M_{\beta_{k+1}}(c_1 )
    \le
    \frac{\chi(\beta_k)}{\chi(\beta_{k+1})}
    \frac{M_{\beta_k}(c_1r_k)}{1-Z_k M_{\beta_k}(c_1r_k)}
    .
\end{equation}
To apply this, we must show that $Z_{k} M_{\beta_k}(c_1r_k)<1$.
For this, we will make use of the fact that \eqref{eq:Edelta},
Corollary~\ref{cor: comparison chi xi when E small} and Proposition~\ref{prop:Zbd}
imply the upper bounds
 \begin{align}\label{eq:relation susceptibility-bis}
    r_k^2 \le 2\frac{\chi(\beta_k)}{\chi(\beta_{k+1})}= \frac{2}{L},
\qquad
   Z_{k} &\le  1-\frac{1}{2}\frac{\chi(\beta_k)}{\chi(\beta_{k+1})}=1-\frac{1}{2L}.
    \end{align}
It is elementary that for any $0 \le |u|\le |v|$, we have
\begin{equation}
    \cosh u \le 1 + \frac{u^2}{2} + \frac{u^4}{v^4}\cosh v,
\end{equation}
and hence for any real and symmetric random variable $U$ and
any $0 \le s\le t$,
\begin{equation}\label{eq: general bound laplace transform-bis}
\mathbb E[\exp(sU)] = \mathbb E[\cosh sU]
\le 1+\frac{s^2}{2}\mathbb E[U^2]+\left(\frac st\right)^4\mathbb E[\exp(tU)].
\end{equation}
We apply \eqref{eq: general bound laplace transform-bis}
to $\mathbb E_{\beta_k}$ with $s=c_1r_k$,
$t=c_1$, and $U= (\mathbf{e}_1\cdot X_1)/\xi(\beta_{k})$.
By symmetry, $\mathbb E_{\beta_k}[U^2] =\frac{1}{d}$.
Therefore,  by \eqref{eq:relation susceptibility-bis}
and the induction hypothesis,
\begin{align}
M_{\beta_k}(c_1r_k)
&\le
1+\frac{c_1^2}{2d} r_k^2
+3 r_k^4
\le
1+\frac{c_1^2}{2d}\frac{2}{L}+ 3\left(\frac{2}{L}\right)^{2}
\le 1+\left(c_1^2 + \frac{12}{L}\right)\frac{1}{L} .
\label{eq: proof regularity 2-bis}
\end{align}
We choose $c_1$ small depending on $L$, so that the above gives
\begin{align}\label{eq:regM}
M_{\beta_k}(c_1r_k)
\le 1+  \frac{15}{L} \frac{1}{L}.
\end{align}
Then, by \eqref{eq:relation susceptibility-bis}, and since $\tfrac{15}{L}=\tfrac{15}{128}<1/2$,
\begin{equation}
\label{eq:ZMbd}
    Z_{k}\, M_{\beta_k}(c_1r_k)
    \le
    \left(1-\frac{1}{2L} \right) \left(1 + \frac{15}{L} \frac{1}{L} \right)
    \le
    1 - \left( \frac 12 - \frac{15}{L}\right) \frac 1L
    <1
    .
\end{equation}
This allows us to apply \eqref{eq:Msum}, to obtain
\begin{align}
    M_{\beta_{k+1}}(c_1)
	&
    \le
    \frac 1L \frac{1 + \frac{15}{L} \frac{1}{L}}{( \frac 12 - \frac{15}{L}) \frac 1L}
    =
    \frac{2 + \frac{30}{L^2}}{1 - \frac{30}{L}}
	.
\end{align}
By choice of $L$, the right-hand side is less than
$3$.
This concludes the induction and the proof.

Something important just happened:
we gain by $Z \le 1 -\frac{1}{2L}$ and we lose by $1+\frac{15}{L^2}$ in \eqref{eq: proof regularity 2-bis}
when $c_1$ is small
enough.
When we take the product of these two effects, we still gain.
The potentially bad effect of the constant $15$ is overcome by a choice of large $L$ .

To summarise the two cases, we have obtained that for $L=128$ and $\delta\leq \tfrac{1}{4L}$,  for every $\beta< \beta(\delta)$ we have
\begin{equation}
    M_\beta(c_1)
    \leq 3.
\end{equation}
We now set $\deltareg:=\frac{1}{4L} = 2^{-9}$ and $(\creg,\Creg)=(c_1,3)$.
This completes the proof.
\end{proof}

\subsection{Stability of the finite-volume susceptibility}\label{sec:stability}

We now prove the stability estimate stated in Proposition \ref{Prop:Stability section 2} and reformulated in \eqref{eq:stability-chi}. For $k\geq 0$, the finite-volume susceptibility is defined by
\begin{equation}
        \chi_k(\beta)=\sum_{x\in \Lambda_k}F_\beta(x).
\end{equation}

\begin{Prop}[Stability estimate]
\label{Prop:stab}
Let $d>2$. Suppose that $J$ and $G$ both obey Assumption \textup{\ref{Ass:G}}.
There exist $\delta_{\textup{stab}}\in (0,\tfrac{1}{10}]$ and $\Cstab>0$, depending only on $d$, such that, for every $\beta<\beta(\delta_{\textup{stab}})$,
\begin{equation}\label{eq:propstab}
	\chi_{\xi(\beta')}(\beta)\leq \Cstab\chi(\beta'), \qquad \forall \beta'\leq \beta.
\end{equation}
\end{Prop}

\begin{Rem}
Let $(\creg,\Creg=3)$ be the regularity constants of the effective random walk from Proposition~\ref{prop: regularity below beta(delta)}, and let $\Crw=\Crw(\creg,\Creg)>0$ be the constants
of Theorem~\ref{thm:estimate RW}.  The proof of
Proposition~\ref{Prop:stab} shows that the constant $\Cstab$ can
be chosen as $\Cstab=32\Crw$.
\end{Rem}

The proof of Proposition~\ref{Prop:stab} builds on the following
three
elementary lemmas. The first lemma is a straightforward extension of Lemma~\ref{lem: iterated SL}. Its bound on $F_\beta(x)$ serves as the starting point for
bounding $\chi_{\xi(\beta')}(\beta)$ in the proof of Proposition~\ref{Prop:stab}.

\begin{Lem}
\label{lem: iterated SL-F}
Let $0\leq \beta'\leq \beta<\beta_c$. For every $T\geq 1$ and every $x\in \mathbb Z^d$,
\begin{equation}\label{eq:FGreenbdfinite}
	F_\beta(x)
    \leq
    (Z_{\beta',\beta}\vee 1)^T
    \Big(
    \chi(\beta')\big(\mathbb G_{1,\beta'}(x)-\delta_{0}(x)\big)
    +
    \sum_{z\in \mathbb Z^d}F_{\beta}(z)\mathbb P_{\beta'}[X_T=x-z]
    \Big).
\end{equation}
\end{Lem}

\begin{proof} Let $T\geq 1$ and $x\in \mathbb Z^d$.
Recall from \eqref{eq:SLTtimes} (with $G$ replaced by $F$) that
\begin{equation}\label{eq:SLTtimes-F}
	F_{\beta}(x)\leq \sum_{k=0}^{T-1}Z_{\beta',\beta}^k \mathbb E_{\beta'}[F_{\beta'}(x-X_k)]+Z_{\beta',\beta}^T\mathbb E_{\beta'}[F_{\beta}(x-X_T)].
\end{equation}
By \eqref{eq:SLTtimes-F}, and since $Z_{\beta',\beta}\leq (Z_{\beta',\beta}\vee 1)$, we obtain
\begin{align}
	F_{\beta}(x)&\leq (Z_{\beta',\beta}\vee 1)^{T-1}\chi(\beta')\sum_{k=0}^{T-1}\mathbb P_{\beta'}[X_{k+1}=x]+(Z_{\beta',\beta}\vee 1)^T\sum_{z\in \mathbb Z^d}\mathbb P_{\beta'}[X_T=z]F_{\beta}(x-z)\nonumber
	\\&\leq (Z_{\beta',\beta}\vee 1)\Big(\chi(\beta')
\big( \mathbb G_{1,\beta'}(x) - \delta_{0}(x) \big)
+\sum_{z\in \mathbb Z^d}F_\beta(z)\mathbb P_{\beta'}[X_T=x-z]\Big).
\end{align}
This completes the proof.
\end{proof}

The second lemma provides a means, given a bound on
a finite-volume susceptibility for one volume, to extract a bound for a larger volume. As is clear from
its proof,
the exponent $\frac 52$ in \eqref{eq:52} could be reduced to any exponent
$a>2$ by choosing $\delta$ appropriately small.
We use $\frac 52$ as a concrete choice; later we only need that $d>a$,
which is satisfied for every $d > 2$.

\begin{Lem}
\label{Lem:am2} Let $\delta\leq \frac{1}{10}$ and $0<\beta'<\beta<\beta(\delta)$. Assume that there exists $C_1\geq 1$ such that, for every $\beta''<\beta$ with $\frac{\xi(\beta'')}{\xi(\beta')}\geq2$, the inequality
    \begin{equation}\label{eq: assumption intermediate lemma for stab}
        \chi_{\xi(\beta'')}(\beta)\leq C_1 \chi(\beta'')
    \end{equation}
    holds. Then, for every $k\geq 2$,
    \begin{equation}
    \label{eq:52}
        \chi_{k \xi(\beta')}(\beta) \leq C_1 k^{5/2} \chi(\beta').
    \end{equation}
\end{Lem}

\begin{proof}
    Fix $\delta\leq\frac{1}{10}$ and $0<\beta'<\beta<\beta(\delta)$. Let $k\geq 2$. We consider two cases. First, suppose that $k\xi(\beta')\geq\xi(\beta)$. By \eqref{eq:bounds chi xi} and the fact that $\tfrac{2}{1-2\delta}\leq \tfrac{5}{2}$,
    \begin{equation}
        \chi_{k \xi(\beta')}(\beta) \leq \chi(\beta) = \frac{\chi(\beta)}{\chi(\beta')} \chi(\beta') \leq \left(\frac{\xi(\beta)}{\xi(\beta')}\right)^{5/2} \chi(\beta') \leq k^{5/2}\chi(\beta').
    \end{equation}
If instead $k\xi(\beta')<\xi(\beta)$, then we define $\beta''<\beta$ by the
requirement that $\xi(\beta'') = k \xi(\beta')$.
Since $k\geq 2$, we have $\tfrac{\xi(\beta'')}{\xi(\beta')}\geq 2$.
It then follows from \eqref{eq: assumption intermediate lemma for stab} that
\begin{equation}
        \chi_{k \xi(\beta')}(\beta) = \chi_{\xi(\beta'')}(\beta) \leq C_1 \chi(\beta'') \leq C_1\left(\frac{\xi(\beta'')}{\xi(\beta')}\right)^{5/2} \chi(\beta') = C_1 k^{5/2}\chi(\beta').
\end{equation}
    This concludes the proof.
\end{proof}

The third and final lemma is an elementary estimate.
It will be used in conjunction with Lemma~\ref{Lem:am2}
to handle the convolution appearing in the
last term of \eqref{eq:FGreenbdfinite}.

\begin{Lem}
\label{lem:conv stab}
Let $\alpha,K> 0$ and $\xi\geq 1$.  Suppose
that $f:\mathbb Z^d \to [0,\infty)$ satisfies, for every $m \ge 2$,
\begin{equation}
         \sum_{x \in \Lambda_{m \xi } } f(y) \leq  K m^{\alpha}  .
\end{equation}
Then, for every $\kappa >0$,
there exists $C_0=C_0(\kappa,\alpha)>0$ such that, for every
$T \ge 4$,
\begin{equation}
    \sum_{x\in\mathbb Z^d}f (x)
    \exp \left( -\kappa \frac{|x|}{\xi\sqrt{T}} \right)
    \le C_0 K T^{\alpha/2}.
\end{equation}
\end{Lem}

\begin{proof}
We decompose the sum into annuli centred at the origin.
For this, we define radii $u_{0}:=0$, $u_1:=\xi\sqrt{T}$, and $u_{k+1}:=2u_k$
for $k\geq 1$. We also let $\varphi_T(x)=\exp(-\kappa \frac{|x|}{\xi\sqrt{T}})$. Then, as $\sqrt{T}\geq 2$, we can apply our hypothesis on $f$
to obtain
\begin{align}
	\sum_{x\in \mathbb Z^d}f(x)\varphi_T(x)
    &=
    \sum_{k\geq 0}\sum_{u_k\leq |x|<u_{k+1}}
    f(x)\varphi_T(x)
	\notag\\
    &\leq
    \sum_{x\in \Lambda_{\xi\sqrt{T}}} f(x)
    + \sum_{m\geq 1} e^{-\kappa 2^{m-1}}\sum_{x\in \Lambda_{2^m\xi\sqrt{T}}} f(x)
    \nonumber \\
    & \le KT^{\alpha/2}
    + \sum_{m\geq 1} e^{-\kappa 2^{m-1}}K(2^m \sqrt{T})^{\alpha}
    .
\label{eq:newlem1a}
\end{align}
The right-hand side is bounded by a constant multiple of $KT^{\alpha/2}$,
with a constant depending only on $\kappa$ and $\alpha$.
\end{proof}

With Lemmas~\ref{lem: iterated SL-F}--\ref{lem:conv stab}, we are now in a position to prove Proposition \ref{Prop:stab}.

\begin{proof}[Proof of Proposition~\textup{\ref{Prop:stab}}]
Let $\delta\in (0,\tfrac{1}{10}]$ to be chosen small enough, and $\Cstab$ to be chosen large enough. Fix $\beta<\beta(\delta)$. We recursively define a decreasing sequence $(\beta_\ell)_{\ell \geq 0}$ as follows:
\begin{enumerate}
	\item[$\bullet$] Set $\beta_0:=\beta$.
	\item[$\bullet$] Assume that $\beta_0,\ldots,\beta_\ell$ have been constructed. If $\xi(\beta_\ell)\in [\xi(0),2\xi(0))$, we stop the construction. Otherwise, we define $\beta_{\ell+1}\in [0,\beta_\ell)$ by $\xi(\beta_{\ell+1})=\tfrac{1}{2}\xi(\beta_\ell)$.
	\item[$\bullet$] Let $M\geq 0$ be the largest $\ell$ such that $\beta_\ell$ has been constructed, and let $\beta_{M+1}=0$.
\end{enumerate}
We now assume that $\delta\le \deltareg$, with $\deltareg$   given by Proposition~\ref{prop: regularity below beta(delta)}. For such $\delta$, the effective random walk is uniformly $(\creg,3)$-regular for some $\creg=\creg(d)>0$. The bounds for regular random walks in Theorem~\ref{thm:estimate RW} therefore apply.
Let $\Crw\geq 1$ be given by Theorem~\ref{thm:estimate RW} for the pair $(\creg,3)$.

We will prove that if $\delta$ is small enough, then for every $0\leq \ell\leq M$, we have
\begin{equation}\label{eq:pstab2}
	\chi_{\xi(\beta_\ell)}(\beta)\leq 4\Crw\chi(\beta_\ell).
\end{equation}
We claim that the above is sufficient to conclude the proof, with $\Cstab:=32\Crw$. Indeed, let $\beta'<\beta$ (there is nothing
to prove if $\beta'=\beta$).
There exists a unique $0\leq \ell \leq M$ such that $\beta'\in [\beta_{\ell+1},\beta_{\ell})$.
Therefore,
\begin{equation}
	\chi_{\xi(\beta')}(\beta)
\le
\chi_{\xi(\beta_{\ell})}(\beta) \le
4\Crw\chi(\beta_{\ell})=
4\Crw\frac{\chi(\beta_{\ell})}{\chi(\beta')}\chi(\beta')
\leq 32\Crw \chi(\beta'),\label{eq:pstab3}
\end{equation}
where in the first inequality we used $\delta\leq \tfrac{1}{2}$ and
Proposition~\ref{prop: rough bounds chi xi}
to see that $\xi(\beta')\leq \xi(\beta_\ell)$,
in the second we used \eqref{eq:pstab2},
and in the last inequality we used \eqref{eq:bounds chi xi} and the assumption that $\delta\leq \tfrac{1}{10}$ to obtain
\begin{equation}
	\frac{\chi(\beta_{\ell})}{\chi(\beta')}\leq \left(\frac{\xi(\beta_\ell)}{\xi(\beta')}\right)^{\tfrac{2}{1-2\delta}}\leq 2^3=8.
\end{equation}

We prove \eqref{eq:pstab2} by induction on $\ell$. First, observe that \eqref{eq:pstab2} for $\ell=0$ follows from the
fact that $4\Crw \ge 4 \ge 1$.

To advance the induction, we fix $1\leq \ell\leq M$ and assume that \eqref{eq:pstab2} holds for all values $m\leq \ell-1$. Our goal is to prove \eqref{eq:pstab2} at $\ell$. Let $\beta'=\beta_\ell$, $n:=\xi(\beta')$, $Z:=Z_{\beta',\beta}$, and $\mathbb P=\mathbb P_{\beta'}$. By \eqref{eq:FGreenbdfinite}, for every $T\geq 1$, and every $x\in \mathbb Z^d$,
\begin{equation}\label{eq:pstab4}
	F_\beta(x)\leq (Z\vee 1)^T\Big(\chi(\beta')\mathbb G(x)+\sum_{z\in \mathbb Z^d}F_\beta(z)\mathbb P[X_T=x-z]\Big),
\end{equation}
where $\mathbb G$ is the Green function associated with the effective random walk at $\beta'$ at $Z=1$.
Summation of \eqref{eq:pstab4} over $x\in \Lambda_n$ gives
\begin{equation}\label{eq:pstab7}
	\chi_n(\beta)\leq (Z\vee 1)^T\bigg(\mathbb G(\Lambda_n)\chi(\beta')+\sum_{z\in \mathbb Z^d}F_\beta(z)\mathbb P[X_T\in \Lambda_n(z)]\bigg).
\end{equation}
Since it follows from Theorem~\ref{thm:estimate RW} that
$\mathbb G(\Lambda_n)\leq \Crw$, the first term on the right-hand side
of \eqref{eq:pstab7} obeys
\begin{equation}\label{eq:pstab7.5}
	(Z\vee 1)^T \mathbb G(\Lambda_n)\chi(\beta')
    \le
   (Z\vee 1)^T  \Crw \chi(\beta').
\end{equation}
For the last term in \eqref{eq:pstab7}, the anti-concentration estimate \eqref{eq:ac-estimate} gives, for every $z\in \mathbb Z^d$
\begin{equation}\label{eq:pstab8}
	\mathbb P[X_T\in \Lambda_n(z)]\leq \frac{\Crw}{T^{d/2}}\exp\left(-\crw\frac{|z|}{n\sqrt{T}}\right).
\end{equation}
In order to apply Lemma~\ref{lem:conv stab},
we need to show that there exists $C_2=C_2(d)>0$ such that, for every $k\geq 2$,
\begin{equation}\label{eq:pstab9 new}
	\sum_{x\in \Lambda_{kn}}F_\beta(x)=\chi_{kn}(\beta)\leq C_2 \chi(\beta') k^{5/2}.
\end{equation}
By Lemma~\ref{Lem:am2}, \eqref{eq:pstab9 new} will follow if we show the existence of $C_3=C_3(d)>0$ such that for any $\beta''<\beta$ with $\tfrac{\xi(\beta'')}{\xi(\beta')}\geq 2$,
\begin{equation}\label{eq:pstab10 new}
	\chi_{\xi(\beta'')}(\beta)\leq C_3\chi(\beta'').
\end{equation}
However, if $\beta''$ is as above, then $\beta''\geq \beta_{\ell-1}$. Therefore, by the induction hypothesis and the same computation as in \eqref{eq:pstab3}, \eqref{eq:pstab10 new} is satisfied with $C_3=32\Crw$. Hence, the conclusion of Lemma \ref{Lem:am2} holds and so does \eqref{eq:pstab9 new} for some $C_2=C_2(d)>0$. We can now take $T \geq 4$ and apply Lemma \ref{lem:conv stab} with $\alpha=\frac 52$ and $K=C_2\chi(\beta')$ to find $C_4=C_4(d)>0$ such that
\begin{equation}\label{eq:pstab11 new}
	\sum_{x\in \mathbb Z^d}F_\beta(x)\varphi_T(x)\leq C_4\chi(\beta')T^{5/4}.
\end{equation}
The combination of \eqref{eq:pstab8}, \eqref{eq:pstab10 new} and \eqref{eq:pstab11 new} gives the existence of $C_5=C_5(d)>0$ such that
\begin{equation}
\label{eq:pstab11.1}
	\sum_{z\in \mathbb Z^d}F_\beta(z)\mathbb P[X_T\in \Lambda_n(z)]
    \leq \frac{C_5}{T^{\frac 12 (d-\frac 52)}}\chi(\beta').
\end{equation}
Since $d \ge 3$,
we may choose $T$ large enough (depending only on $d$) that
\begin{equation}
	\frac{C_5}{T^{\frac 12 (d-\frac 52)}} \leq \Crw.
\end{equation}
This shows that
\begin{equation}
\label{eq:pstab11.5}
	(Z\vee 1)^T \sum_{z\in \mathbb Z^d}F_\beta(z)\mathbb P[X_T\in \Lambda_n(z)]
    \leq
    (Z\vee 1)^T \Crw\chi(\beta').
\end{equation}

We now insert the bounds \eqref{eq:pstab7.5} and \eqref{eq:pstab11.5}
into \eqref{eq:pstab7}, and obtain
\begin{equation}\label{eq:pstab12}
	\chi_n(\beta)\leq 2(Z\vee 1)^T\Crw\chi(\beta').
\end{equation}
Finally, it follows from \eqref{eq: bounds on Z1} that $(Z\vee 1)\leq (1-\delta)^{-1}$. We choose $\delta$
small enough (depending on $d$) such that $(1-\delta)^{-T}\leq 2$. For this choice, \eqref{eq:pstab12} gives $\chi_n(\beta)\leq 4\Crw\chi(\beta')$, and the induction step is complete.  This concludes the proof.
\end{proof}

\section{Proof of the multiscale bounds}\label{sec:4props}

We now turn to the proofs of the three remaining propositions,
Propositions~\ref{Prop:smallx}--\ref{prop:intermscales}. Throughout the section, we fix $d>2$ and $\varepsilon>0$. We also suppose that $J,G$ satisfy Assumption \ref{Ass:G}. For convenience, we drop the subscript $J$ in the notations and write
simply $\sigma$ instead of $\sigma_J$. We always assume that $\delta\leq \deltareg\wedge \deltastab$ so that Propositions \ref{prop: regularity below beta(delta)} and \ref{Prop:stab} hold. As a result, by Proposition \ref{prop: regularity below beta(delta)}, we know that the effective random walk is uniformly $(\creg,3)$-regular for every $\beta<\beta(\delta)$. This allows us to
apply
Theorem~\ref{thm:estimate RW} with constants $(\crw,\Crw)$ which depend on $(\creg,\Creg=3,d)$.

\subsection{Proof of Proposition~\ref{Prop:smallx}: bound on $F_\beta(x)$ for $|x|\leq 2\sigma$}
\label{sec:smallscales}

Recall that the constant $\CGreen$, which appears in \eqref{eq:smallscalesbound},
is the constant arising in the estimate for the Green function $\mathbb C$ of the random walk $\mathbb P_J$,
in Proposition~\ref{Prop:Green}.

\begin{proof}[Proof of Proposition~\textup{\ref{Prop:smallx}}] We will prove that there exists $\delta_s\in (0,\tfrac{1}{2}]$ such that, for every $\beta<\beta(\delta_s)$ and every $|x| \leq 2\xi(0)$,
    \begin{equation}
        F_{\beta}(x)
        \leq
        2  \big(\mathbb{C}_1(x) -\delta_{0}(x)\big) + 2 \CGreen \left(\frac{1}{2\sigma}\right)^{d}.
        \label{eq: smallscalesbetterbound}
    \end{equation}
    We can bound the right-hand side using
    Proposition~\ref{Prop:Green}, which together with $|x|\le 2\sigma$ gives
    \begin{equation}
    \big(\mathbb{C}_1(x) -\delta_{0}(x)\big)  +   \CGreen \left(\frac{1}{2\sigma}\right)^{d}
    \le
    \frac{\CGreen}{\sigma^d}  \left(\frac{\sigma}{\sigma \vee |x|}\right)^{d-2}
    +
    \frac{\CGreen}{\sigma^d}  \left(\frac{\sigma}{\sigma \vee |x|}\right)^{d-2}.
    \end{equation}
    With \eqref{eq: smallscalesbetterbound}, this proves our goal
    \eqref{eq:smallscalesbound}.

    It remains to prove \eqref{eq: smallscalesbetterbound}. Let $\delta\in(0,\tfrac12]$ to be taken small enough, let $\beta<\beta(\delta)$, and set $\beta' = 0$.
    Then $\chi(\beta') = 1$, $\mathbb P_{\beta' } = \mathbb P_J $, and $\mathbb G_{1,\beta'} = \mathbb{C}_{1}$.
    We write $Z:=Z_{\beta',\beta}$. By \eqref{eq:FGreenbdfinite}, for every $T\geq 1$ and every $x\in \mathbb Z^d$, we have
    \begin{equation}
    \label{eq:smallx-1}
	F_\beta(x)
    \leq
    (Z\vee 1)^T\Big(
    \big(\mathbb{C}_1(x) -\delta_{0}(x)\big)
    +\sum_{z\in \mathbb Z^d}F_\beta(z)\mathbb P_{J}[X_T=x-z]
    \Big).
    \end{equation}
    Our main effort is to bound the second term on the right-hand side of \eqref{eq:smallx-1}. We argue as in the proof of Proposition~\ref{Prop:stab}.
    The anti-concentration estimate \eqref{eq:antix-intro} asserts that there exists $(\cGreen,\CGreen)$ (depending on $d$ and $c_0$ from \eqref{eq:R_J comparable to xi(0)})
    such that, for every $z\in \mathbb Z^d$,
    \begin{equation}
    \label{eq:antix}
	\mathbb P_{J}[X_T = z ]\leq \frac{\CGreen}{\sigma^d T^{d/2}}\exp\left(-\cGreen\frac{|z|}{\sigma\sqrt{T}}\right).
    \end{equation}
Since $|x|\le 2\xi(0)$, the choice $T\ge 4$ in \eqref{eq:antix} gives, for every $z\in \mathbb Z^d$,
\begin{align}
    \mathbb P_{J}[X_T = x-z ]&\leq \frac{\CGreen}{\sigma^d T^{d/2}}\exp\left(-\cGreen\frac{|z|}{\sigma\sqrt{T}}\right) \exp\left(\cGreen\frac{2}{\sqrt{T}}\right)
    \nonumber \\
    &\leq \frac{\CGreen e^{\cGreen}}{\sigma^d T^{d/2}}
    \exp\left(-\cGreen\frac{|z|}{\sigma\sqrt{T}}\right).
\label{eq: smalleq1}
\end{align}
Proposition~\ref{Prop:stab} and Lemma~\ref{Lem:am2} then
imply that, if $\delta \leq \delta_{\textup{stab}}$, then for every $k\geq 2$,
\begin{equation}
    \sum_{z \in \Lambda_{k\sigma}}F_{\beta}(z) \leq C_{\textup{stab}} k^{5/2} \chi(\beta') = C_{\textup{stab}} k^{5/2}.
\end{equation}
We apply Lemma~\ref{lem:conv stab}
with $\alpha = \frac 52$, $f=F_\beta$, and $K = C_{\textup{stab}} $, to obtain $C_1 = C_1(\cGreen,d)>0$ such that
\begin{equation}
    \sum_{z \in \mathbb{Z}^d}F_{\beta}(z)
    \exp\left(-\cGreen\frac{|z|}{\sigma\sqrt{T}}\right)
     \leq C_1  T^{5/4}\label{eq: smalleq2}.
\end{equation}
The combination of \eqref{eq: smalleq1} and \eqref{eq: smalleq2} gives $C_2 = C_2(\cGreen,d)>0$ such that, for every $T\geq 4$ and every $|x|\leq 2\xi(0)$,
\begin{equation}
\label{eq:T2}
    \sum_{z\in \mathbb Z^d}F_\beta(z)\mathbb P_{J}[X_T=x-z] \leq \frac{C_2}{\sigma^d T^{\frac 12 (d-\frac 52)}}.
\end{equation}
Since $d>2$, we may now choose $T$ large enough so that
\begin{equation}
	\frac{C_2}{T^{\frac 12 (d-\frac 52)}}\leq \frac{\CGreen}{2^d}.
\end{equation}

Finally, we pick $\delta$ sufficiently small so that the prefactor on the right-hand side of \eqref{eq:smallx-1} satisfies $(Z \vee 1)^T \leq (1-\delta)^{-T} \leq 2$. With \eqref{eq: smalleq2}--\eqref{eq:T2}, we see that the second term on the right-hand side of \eqref{eq:smallx-1}
is at most $2\CGreen (2\sigma)^{-d}$.  This proves \eqref{eq: smallscalesbetterbound} and completes the proof.
\end{proof}

\subsection{Proof of Proposition~\ref{prop: typical scales}: contraction step for
typical scales}
\label{sec:contraction-typical}

Sections~\ref{sec:contraction-typical} and \ref{sec:contraction-intermediate} are dedicated to the proof of the
contraction step of the bootstrap argument, namely
Propositions~\ref{prop: typical scales} and \ref{prop:intermscales}.
We now prove Proposition~\ref{prop: typical scales}.

Recall that, for every $x\in \mathbb Z^d$, for every $\beta \ge 0$, and for $\bfc,\bfC>0$,
$\mathcal H_\beta(\bfc,\bfC;x)$ and $\mathcal H_\beta(\bfc,\bfC)$
are the statements:
\begin{align}
	\mathcal H_\beta(\bfc,\bfC;x): &\quad F_{\beta}(x)\leq \frac{\bfC}{\sigma^d}\left(\frac{\sigma}{\sigma \vee|x|}\right)^{d-2-\varepsilon}
\exp\left(-\bfc\frac{|x|}{\xi(\beta)}\right) ,
\\
	\mathcal H_\beta(\bfc,\bfC): &\quad
    \text{$\mathcal H_\beta(\bfc,\bfC;x)$ holds for all $x\in \mathbb Z^d$}.
\end{align}
Proposition~\ref{prop: typical scales} can be reformulated as follows: if $\eta>0$, $\bfC\geq 16\CGreen$, $\bfc\leq \tfrac{1}{2}\cGreen\wedge \tfrac{1}{4}\crw\wedge 1$, then there exists $\delta_2$ small enough such that for every $\beta<\beta(\delta_2)$, if $\mathcal H_{\beta'}(\bfc,\bfC)$ holds for $\beta'\leq \beta$, then $\mathcal H_{\beta}(\bfc,\tfrac{\bfC}{8};x)$ holds for every $|x|\geq \eta\xi(\beta)\vee 2\xi(0)$.
The improved bound will be obtained by relying on the finite-difference inequality \eqref{eq:SL assumption}.

We stress that the proof of Proposition \ref{prop: typical scales} is the only place in the entire paper where the assumption that $\varepsilon>0$ is used.
We begin by stating two simple lemmas. The first one, an elementary consequence of \eqref{eq:SL assumption}, is the starting point of the proof.  Recall from \eqref{eq:Green-effRW} that the Green function of the effective random
walk is defined, for $Z\in [0,1]$ and $x\in \mathbb Z^d$, by
\begin{equation}
    \mathbb G_{Z,\beta'}(x)=\sum_{k\geq 0}Z^k \mathbb P_{\beta'}[X_k=x].
\end{equation}

\begin{Lem}
\label{lemma:typ-decomp}
Let $0 \le\beta' \le \beta<\beta_c$ be such that $Z_{\beta',\beta}< 1$. For every  $m >0$ and every $x\in \mathbb Z^d$ with $|x| > m$, we have
\begin{align}
\label{eq:typ1}
    F_\beta(x)
    \le
    \frac{\chi_m(\beta')}{\chi(\beta')} \sup_{y \in \Lambda_m(0)} F_\beta(x-y)
    +
    \sum_{y \notin \Lambda_{m}(0)} F_{\beta'}(y)\mathbb G_{Z_{\beta',\beta},\beta'}(x-y)
    .
\end{align}
\end{Lem}

\begin{proof}
By the finite-difference upper bound in \eqref{eq:SL assumption},
\begin{align}
\label{eq:fdbd-F}
    F_{\beta}(x) &\leq F_{\beta'}(x) + (\beta-\beta')
    \sum_{y\in \mathbb Z^d}F_{\beta'}(y) F_{\beta}(x-y).
\end{align}
The contribution to the convolution sum due to $y \in \Lambda_m(0)$
can be rewritten and bounded (since $Z_{\beta',\beta}< 1$) by\begin{equation}
    \frac{Z_{\beta',\beta}}{\chi(\beta')}
    \sum_{y \in \Lambda_{m}(0) }F_{\beta'}(y)F_{\beta}(x-y)
    \le
    \frac{\chi_m(\beta')}{\chi(\beta')} \sup_{y \in \Lambda_m(0)} F_\beta(x-y).
\end{equation}
This gives the first term on the right-hand side
of \eqref{eq:typ1}.

From \eqref{eq:SLinfty} with $G$ replaced by $F$, we see that
\begin{align}
\label{eq:FGreenbd}
    (\beta-\beta') F_\beta(x) &
    \le
    Z_{\beta',\beta}
    \sum_{k\geq 0}Z_{\beta',\beta}^k \mathbb P_{\beta'}[X_{k+1}=x]
    =
    \mathbb G_{Z_{\beta',\beta},\beta'}(x) - \delta_{0}(x).
\end{align}
By \eqref{eq:FGreenbd}, the remaining part of the right-hand side
of \eqref{eq:fdbd-F} is therefore at most
\begin{align}
    F_{\beta'}(x) +
    \sum_{y\notin \Lambda_{m}(0)}F_{\beta'}(y)
    \big[ \mathbb G_{Z_{\beta',\beta},\beta'}(x-y) - \delta_{0}(x-y) \big].
\end{align}
By our assumption that $|x|>m$, the first term is cancelled by the Kronecker delta.
This completes the proof.
\end{proof}

To bound the last term on the right-hand side of
\eqref{eq:typ1}, we will use the following convolution lemma. We defer the elementary proof to Appendix~\ref{appendix:conv estimates}.

\begin{Lem}
\label{lemma:fstarg}
Let $a,b,c_1,c_2,\sigma,\xi >0$,
$\mu>0$ and $\varepsilon \in [ 0,1]$.
Suppose that the functions $f,g: \mathbb Z^d \to [0,\infty)$ satisfy
\begin{align}
    f(x)
    &\le
    c_1\frac{1}{\sigma^d} \left( \frac{\sigma}{\sigma\vee |x|}\right)^{d-2-\varepsilon}
    e^{-a|x|/\xi},
    \\
    g(\Lambda_\xi(x))
    &\le
    c_2\left( \frac{\xi}{\xi\vee |x|}\right)^{d-2}
    e^{-b|x|/\xi}.
\end{align}
Then, there exists $C_{a,\mu}>0$ (depending on $a,\mu,d$) such that  for every $|x| \ge 2 (\sigma \vee \xi)$,
\begin{equation}
\label{eq:fstarg}
    \sum_{y \notin \Lambda_{ \mu \xi}(0)} f(y)g(x-y)
    \le \frac{c_1}{\sigma^2|x|^{d-2}} \left(\frac{|x|}{\sigma}\right)^\varepsilon
    \left(
    2^{d} \|g\|_1  e^{-a|x|/2\xi}
    +
    c_2C_{a,\mu}\left(\frac{\xi}{|x|}\right)^\varepsilon e^{-b|x|/2\xi}
    \right).
\end{equation}
\end{Lem}

We are now in a position to prove Proposition \ref{prop: typical scales}. Let $\crw$ and $\Crw$ be given by Theorem~\ref{thm:estimate RW}, for the pair $(\creg,\Creg)=(\creg,3)$ provided by Proposition \ref{prop: regularity below beta(delta)}.

\begin{proof}[Proof of Proposition~\textup{\ref{prop: typical scales}}]
Let $d>2$, $\varepsilon>0$ and $\eta >0$.
Let $\bfC\ge 16\CGreen$ and
$\bfc \le \frac 12 \cGreen \wedge \frac{1}{4}\crw \wedge 1$.
Our goal is to prove that there exists $\delta_2=\delta_2(\eta,\varepsilon,d)\leq \delta_s\wedge \delta_{\textup{reg}}\wedge \tfrac{1}{2}$ such that, for every $\beta<\beta(\delta_2)$,
if $\mathcal H_{\beta'}(\bfc,\bfC)$ holds for every $\beta' \le \beta$,
then, for every $|x|\geq \eta \xi(\beta)\vee 2\xi(0)$,
\begin{equation}
\label{eq:largex}
	F_\beta(x)
    \leq
    \frac{\bfC}{8}
    \left(\frac{|x|}{\sigma}\right)^\varepsilon
    \frac{1}{\sigma^2 |x|^{d-2}}
    \exp\left(-\bfc \frac{|x|}{\xi(\beta)}\right).
\end{equation}
On the right-hand side, we used our assumption that
$|x|\ge2\xi(0)=2\sigma$.

Let $\delta \leq \tfrac14\wedge \deltareg\wedge\deltastab$, where $\deltareg$ and $\deltastab$ are respectively given by
Propositions~\ref{prop: regularity below beta(delta)} and \ref{Prop:stab}.
If $\beta\leq \beta_0:=(1-\delta)\wedge \beta(\delta)$, the
desired bound \eqref{eq:largex} is already known from \eqref{eq:Fxbd}.
We can therefore assume that $\beta\in(\beta_0,\beta(\delta))$.
Fix such a $\beta$.
We wish to apply Lemma~\ref{lemma:typ-decomp}, which states that
if $0 \le\beta' \le \beta$ are such that $Z_{\beta',\beta}<1$,
and if $|x| > m > 0$, then
\begin{align}
\label{eq:typ1-bis}
    F_\beta(x)
    \le
    \frac{\chi_m(\beta')}{\chi(\beta')} \sup_{y \in \Lambda_m(0)} F_\beta(x-y)
    +
    \sum_{y \notin \Lambda_{m}(0)} F_{\beta'}(y)\mathbb G_{Z_{\beta',\beta},\beta'}(x-y)
    .
\end{align}
Later in the proof, we write the two terms on the right-hand side of \eqref{eq:typ1-bis}
as $({\rm I})$ and $({\rm II})$.

To apply \eqref{eq:typ1-bis}, we need to pick $\beta'$ and $m$.
These choices are made in terms of $\xi(\beta)$, as follows.
First, given $0 < \beta'' < \beta' < \beta$, we set
\begin{equation}
    m = \xi(\beta''), \qquad n = \xi(\beta'), \qquad N = \xi(\beta).
\end{equation}
Given $\beta''$ and $\beta'$, we use $\mu,\nu\in (0,1]$ to denote the ratios
\begin{equation}
    \mu = \frac mn =\frac{\xi(\beta'')}{\xi(\beta')}, \qquad
    \nu = \frac nN = \frac{\xi(\beta')}{\xi(\beta)}.
\end{equation}
We will want to take $\mu,\nu$ to be small, but we must have
$m = \mu\nu N \ge \xi(0)$ in order to have the existence of $\beta''$
with $m=\xi(\beta'')$.
This can be accomplished as follows.
First, if $\delta\leq \tfrac{1}{4}$ and $\beta< \beta(\delta)$, \eqref{eq:bounds chi xi} and \eqref{eq:chi0} yield
\begin{equation}
\label{eq:typ8}
    \frac{\xi(\beta)}{\xi(0)} \ge \left(\frac{\chi(\beta)}{\chi(0)}\right)^{\tfrac{1}{2}-\delta}  \ge\left( \frac{\chi(\beta)}{\chi(0)} \right)^{1/4}
    \ge \frac{1}{(1-\beta(1-\delta) )^{1/4}},
\end{equation}
so,
\begin{equation}
    \frac{\xi(\beta_0)}{\xi(0)}
    \ge \frac{1}{(1-(1-\delta)^2 )^{1/4}}
        \ge
    \frac{1}{(2\delta)^{1/4}}
    .
\end{equation}
Therefore, given any $\mu,\nu \le 1$,
we can choose $\delta(\mu,\nu)$ to be small enough that
$m=\mu\nu \xi(\beta) \ge \mu\nu \xi(\beta_0)  \ge (2\delta)^{1/4}\xi(\beta_0)\geq \xi(0)$.

For \eqref{eq:typ1-bis}, we also need $|x| >m = \mu\nu \xi(\beta)$.
With later needs in mind, we require that $\nu \le \frac 12 \eta$.
Then, since we assume that $|x|\ge \eta \xi(\beta)$,
we (more than) ensure that
$|x|>m$ since $m=\mu\nu N \le \nu N \le \frac 12 \eta N$.  In the following,
we will first fix $\mu$, and then choose a $\nu$ satisfying $\nu \le \frac 12 \eta$
(with further restriction below).
This adds $\eta$-dependence to $\delta$. We also require that $\delta\leq \tfrac{1}{8}(\mu\wedge \nu)^{2}$ so that the conclusions of Corollary~\ref{cor: comparison chi xi when E small} and Proposition~\ref{prop:Zbd}
hold with the pairs $(\beta'',\beta')$ and $(\beta',\beta)$.

By Proposition~\ref{Prop:stab} and \eqref{eq: comparison chi xi when E small}, the ratio in the first term on
the right-hand side of \eqref{eq:typ1-bis} obeys
\begin{equation}
    \frac{\chi_m(\beta')}{\chi(\beta')}
    =
    \frac{\chi_{\xi(\beta'')}(\beta')}{\chi(\beta'')}
    \frac{\chi(\beta'')}{\chi(\beta')}
    \le
    \Cstab\cdot 2\left(\frac{\xi(\beta'')}{\xi(\beta')}\right)^{2}
    =
    \Cstab \cdot 2\mu^2.
\end{equation}
By the assumption that $\mathcal H_\beta(\bfc,\bfC)$ holds,
and by using $\mu  n/N =\mu\nu \le 1$ and our assumption
 $\bfc \le 1$ in the exponent, we see that
\begin{equation}
    \sup_{y \in \Lambda_m(0)} F_\beta(x-y)
    \le
    e \cdot \frac{\bfC}{\sigma^d}
    \left(\frac{\sigma}{\sigma\vee(|x|-\mu n)}\right)^{d-2-\varepsilon}
    \exp\left(-\bfc\frac{|x| }{\xi(\beta)}\right).
\end{equation}
Since $|x|\ge\eta N$, and since $\mu\nu \le \nu \le \frac 12 \eta$, we have
\begin{equation}
    |x| -\mu n \ge |x|\left( 1- \frac{\mu n}{\eta N} \right)
    =
    |x|\left( 1- \frac{\mu\nu }{\eta} \right) \ge \frac 12 |x|.
\end{equation}
Therefore, with the choice
\begin{equation}
\label{eq:lambda-choice}
    \mu^2 = \frac{1}{16} \frac{1}{e 2^{d-2-\varepsilon}\Cstab 2},
\end{equation}
the first term of \eqref{eq:typ1-bis} obeys
\begin{equation}
\label{eq:typIbd}
    ({\rm I})
    \le
    \frac{1}{16}
    \frac{\bfC}{\sigma^d}
    \left(\frac{\sigma}{\sigma\vee|x| }\right)^{d-2-\varepsilon}
    \exp\left(-\bfc\frac{|x| }{\xi(\beta)}\right).
\end{equation}

To handle the second term on the right-hand side of \eqref{eq:typ1-bis}, we apply
Lemma~\ref{lemma:fstarg} with $f=F_{\beta'}$, $g=\mathbb G_{Z_{\beta',\beta},\beta'}$, $\xi=\xi(\beta')$, and with the $\mu$ we chose in \eqref{eq:lambda-choice}.
The hypothesis on $f$ is verified by
$\mathcal H_{\beta'}(\bfc,\bfC)$ with $c_1=\bfC$, $a=\bfc$, $\xi=n$, and the hypothesis
on $g$ is verified by Theorem~\ref{thm:estimate RW} with
$c_2=\Crw$, $b=\crw \sqrt{1-Z_{\beta',\beta}}$, $\xi=n$.
It therefore follows from Lemma~\ref{lemma:fstarg}, since $|x| \ge \eta N \ge 2n$
(as $\nu \le \frac 12 \eta$)
and $\|\mathbb G_{Z_{\beta',\beta},\beta'}\|_1 = (1-Z_{\beta',\beta})^{-1}$, that
\begin{align}
    ({\rm II}) \le
    \frac{\bfC}{\sigma^2|x|^{d-2}} \left(\frac{|x|}{\sigma}\right)^\varepsilon
    \left(
    \frac{2^{d}}{1-Z}   e^{-\bfc|x|/2n}
    +
    \Crw C_{\bfc,\mu}\left(\frac{n}{|x|}\right)^\varepsilon
    e^{-\crw\sqrt{1-Z}|x|/2n}
    \right),
\end{align}
where $Z=Z_{\beta',\beta}$. By  \eqref{eq: bounds on Z in terms of n/N} applied to $(\beta',\beta)$, we
have
 \begin{align}
 \label{eq: bound Z final proof}
   Z \le  1-\frac{1}{4}\Big( \frac nN \Big)^2.
 \end{align}
Since $\bfc\leq\frac{1}{4}\crw$, we also find
\begin{equation}
\label{eq:exponent-rescale}
    \crw \frac{\sqrt{1-Z}}{2n} \ge  4 \bfc \frac{1}{4N} = \frac{\bfc}{N}.
\end{equation}
We use the above, together with $|x| \ge \eta N$, to conclude that
there is a constant $C_1=C_1(d,\eta,\bfc,\mu)>0$ such that
\begin{align}
\label{eq:new2}
    ({\rm II}) \le
    C_1 \Big( \frac nN \Big)^\varepsilon
    \frac{\bfC}{\sigma^2|x|^{d-2}} \left(\frac{|x|}{\sigma}\right)^\varepsilon
    \left(
    \left( \frac Nn \right)^{2+\varepsilon}  e^{-\bfc|x|/2n}
    +
    e^{-\bfc|x|/N}
    \right).
\end{align}
Again using $|x| \ge \eta N$, we see that
\begin{equation}
    \frac{|x|}{2n} = \frac{|x|}{N} + \Big( \frac{N}{2n} - 1\Big) \frac{|x|}{N}
    \ge   \frac{|x|}{N} + \Big( \frac{N}{2n} - 1 \Big)\eta,
\end{equation}
and hence
\begin{equation}
\label{eq:new3}
    \left(\frac{N}{n}\right)^{2+\varepsilon}  e^{-  \bfc |x|/2n}
    \le e^{- \bfc |x|/N}
    e^{\bfc  \eta}  \left(\frac{N}{n}\right)^{2+\varepsilon} e^{-\frac 12 \bfc \eta(N/n)}
    \le
    C_2(\bfc,\eta,\varepsilon)e^{- \bfc |x|/N}.
\end{equation}
Therefore, by \eqref{eq:new2}
and \eqref{eq:new3},
there is a constant $C_3=C_3(d,\eta,\bfc,\mu,\varepsilon)>0$ such that
\begin{equation}
\label{eq:typ-IIbd}
    ({\rm II}) \le C_3 \nu^\varepsilon
    \frac{\bfC}{\sigma^2|x|^{d-2}}
    \left(\frac{|x|}{\sigma}\right)^\varepsilon
    \exp\left(-\bfc \frac{|x|}{N}\right)
    .
\end{equation}
The inequality \eqref{eq:typ-IIbd} holds for any $\nu$ such that  $\nu \le \frac 12 \eta$.
In terms of our fixed choice of $\mu$ in \eqref{eq:lambda-choice},
we can now choose $\nu$ small enough (which entails taking $\delta$ small) so that $C_3\nu^\varepsilon \le \frac{1}{16}$.

From the bound on $({\rm I})$ in \eqref{eq:typIbd}, together with the bound on $({\rm II})$
in \eqref{eq:typ-IIbd} with our choice of $\nu$, we obtain the desired result
that
\begin{equation}
    F_\beta(x) \le \left( \frac{1}{16} + \frac{1}{16} \right)
    \frac{\bfC}{\sigma^2|x|^{d-2}}
    \left(\frac{|x|}{\sigma}\right)^\varepsilon
    \exp\left(-\bfc \frac{|x|}{N}\right).
\end{equation}
This completes the proof.
\end{proof}

\begin{Rem}
The need for $\varepsilon >0$ occurs only near the end of
the previous proof, to achieve $C_3\nu^\varepsilon \le \frac 1{16}$.
This is the essential element of the contraction proof for typical scales.
There is no additional need for $\varepsilon$ elsewhere in the entire paper.
\end{Rem}

\subsection{Proof of Proposition~\ref{prop:intermscales}: contraction step for intermediate scales}
\label{sec:contraction-intermediate}

The proof of Proposition~\ref{prop:intermscales} uses the following elementary convolution lemma,
whose proof is deferred to Appendix~\ref{appendix:conv estimates}.

\begin{Lem}
\label{lemma:fstarf}
Let $p,a>0$.
For $i=1,2$, suppose that $f_i \in \ell^1(\mathbb Z^d)$
satisfy $0 \le f_i(x) \le a(1\vee |x|)^{-p}$ for every $x\in\mathbb Z^d$.
Let $k\ge 1$.
Then, for every $x\in \mathbb Z^d$,
\begin{equation}
    (f_1*f_2)(x)
    \le
    \frac{a}{(1\vee |x|)^p}
    \left(  \frac{1}{k^{p}}  \|f_1\|_1
    + 2^{p} \sum_{y \in \Lambda_{k|x|}(0)}(f_1(y)+f_2(y)) \right).
\end{equation}
\end{Lem}

Proposition~\ref{prop:intermscales} establishes an improvement on the bound on $F_\beta$ for the \emph{intermediate} scales
$2\xi(0)\leq |x|\leq \eta \xi(\beta)$.
We do not use $\varepsilon>0$ in the proof of Proposition~\ref{prop:intermscales}---the parameter
$\varepsilon$ is merely a spectator.

\begin{proof}[Proof of Proposition~\textup{\ref{prop:intermscales}}]
 Let $d>2$, $\varepsilon>0$, $\bfC\geq 16\CGreen$,  $\bfc\leq \tfrac{1}{2}\bfc\wedge \tfrac{1}{4}\crw\wedge 1$ and let
$\delta_2=\delta_2(\eta,\varepsilon,d)$ be given by
Proposition~\textup{\ref{prop: typical scales}}.
Our goal is to prove that there exists $\eta\in (0,1)$
such that the following is true for every $\beta<\beta(\delta_2)$: If $\mathcal H_{\beta'}(\bfc,\bfC)$ holds for every $\beta' \le \beta$
then, for every $2\xi(0)\leq |x|\leq \eta \xi(\beta)$,
\begin{equation}
	F_\beta(x)
    \leq
    \frac{1}{6} \frac{\bfC}{\sigma^d}
    \left(\frac{\sigma}{\sigma\vee |x|}\right)^{d-2-\varepsilon}
    .
\end{equation}
Since $|x| \ge 2\xi(0)=2\sigma$, the above is equivalent to
\begin{equation}
\label{eq:intx-want}
	F_\beta(x)
    \leq
    \frac{1}{6} \bfC \left(\frac{|x|}{\sigma}\right)^\varepsilon
    \frac{1}{\sigma^2 |x|^{d-2}}.
\end{equation}

Let $\eta\in(0,1)$ to be chosen small enough.
Let $\beta<\beta(\delta_2)$.
Fix $x$ with $2\xi(0)\leq |x|\leq \eta \xi(\beta)$.
We may assume that $2\xi(0) \le \eta \xi(\beta)$, since otherwise there is nothing to prove. Then, since $2/\eta \ge 1$,
\begin{equation}
    \xi(0)  \le \frac{|x|}\eta \le \xi(\beta).
\end{equation}
Since $\xi(\beta)$ is monotone increasing in
$\beta <\beta(\delta_2)$ (by Proposition~\ref{prop: rough bounds chi xi}), we can choose
$0< \beta' \le \beta$ such that
$|x|=\eta\xi(\beta')$.  With this choice, we apply
\eqref{eq:SL assumption} to obtain
\begin{equation}
\label{eq:FF}
	F_\beta(x)\leq F_{\beta'}(x)+(\beta-\beta')(F_{\beta'}*F_\beta)(x).
\end{equation}
For the first term on the right-hand side, we
apply Proposition~\ref{prop: typical scales} at the parameter $\beta'$ to get
\begin{equation}
	F_{\beta'}(x)\leq
    \frac{1}{8} \bfC \left(\frac{|x|}{\sigma}\right)^\varepsilon
    \frac{1}{\sigma^2 |x|^{d-2}}.
\end{equation}
The second term on the right-hand side of \eqref{eq:FF} is
\begin{equation}
    \frac{Z_{\beta',\beta}}{\chi(\beta')}
    (F_{\beta'}*F_\beta)(x)
    \le
    \frac{2}{\chi(\beta')}
    (F_{\beta'}*F_\beta)(x)
\end{equation}
because, by \eqref{eq: bounds on Z1}, $Z_{\beta',\beta}\leq \tfrac{1}{1-\delta_2}\leq 2$
since $\delta_2 \le \frac 12$.
As a result, to prove \eqref{eq:intx-want} it suffices to show that, if $\eta$ is small enough, then
\begin{equation}\label{eq: interm improv 1}
    \frac{1}{\chi(\beta')}
    (F_{\beta'}*F_\beta)(x)
    \leq
    \frac{1}{48} \bfC \left(\frac{|x|}{\sigma}\right)^\varepsilon
    \frac{1}{\sigma^2 |x|^{d-2}}.
\end{equation}

Given that $\mathcal H_{\beta'}(\bfc,\bfC)$ and $\mathcal H_{\beta}(\bfc,\bfC)$ hold, the functions
$f_1=F_{\beta'}$ and $f_2=F_{\beta}$ obey the hypothesis of
Lemma~\ref{lemma:fstarf} with $p=d-2-\varepsilon$
and $a= \bfC\sigma^{-2-\varepsilon}$.
Since
$f_1 \le f_2$, this gives, for every $k \ge 1$,
\begin{equation}
\label{eq:interm5}
    \frac{1}{\chi(\beta')}
    (F_{\beta'}*F_\beta)(x)
    \leq
    \bfC \left(\frac{|x|}{\sigma}\right)^\varepsilon
    \frac{1}{\sigma^2 |x|^{d-2}}
    \left( \frac{1}{k^{d-2-\varepsilon}}  + 2^{d-2-\varepsilon} \frac{\chi_{k|x|}(\beta)}{\chi(\beta')} \right)
    .
\end{equation}
We choose $k=\eta^{-1/2}\geq 1$,
so that $k|x|=k\eta\xi(\beta')=\eta^{1/2}\xi(\beta')$.  Since
\begin{equation}
\xi(0)  \le \eta^{-1/2}|x| =
\eta^{1/2}\xi(\beta')    \le \xi(\beta'),
\end{equation}
there is a
$\beta''\le \beta'$ for which $\eta^{1/2}\xi(\beta') = \xi(\beta'')$.
Then, the ratio in the last term on the right-hand side of \eqref{eq:interm5} satisfies
\begin{equation}
    \frac{\chi_{k|x|}(\beta)}{\chi(\beta')}
    =
    \frac{\chi_{\xi(\beta'')}(\beta)}{\chi(\beta'')}
     \frac{\chi(\beta'')}{\chi(\beta')}
     \le
     \Cstab
     \left(\frac{\xi(\beta'')}{\xi(\beta')} \right)^{2\frac{1-\delta_2}{1+\delta_2}}
     \le
     \Cstab  \eta^{1/3},
\end{equation}
where we successively used Proposition~\ref{Prop:stab},
\eqref{eq:bounds chi xi}, and the fact that the exponent in the third member is at least $\frac 23$ because $\delta_2 \le \frac 12$.
We now choose $\eta\in(0,1)$ small enough that
\begin{equation}
    2^{d-2-\varepsilon}
    (\eta^{\frac 12(d-2-\varepsilon)} + \eta^{1/3})
    \le \frac{1}{48}.
\end{equation}
With \eqref{eq:interm5}, this completes the proof of \eqref{eq: interm improv 1}, and concludes the proof of the proposition.
\end{proof}

\begin{Rem} In the proof of Proposition \ref{prop:intermscales}, we showed a weak form of pointwise stability: the second term in the right-hand side of \eqref{eq:FF} is much smaller than \emph{any} input bound on $F_\beta(x)$. More precisely, we proved that, if $\kappa>0$, for
$2\xi(0) \le |x|\le \eta \xi(\beta')$, if $\eta$ (and hence $\delta_2$) is sufficiently small, we can ensure that
\begin{equation}
	F_\beta(x)\leq F_{\beta'}(x)
    +\frac{\kappa}{\sigma^2|x|^{d-2}}
    \left(\frac{|x|}{\sigma}\right)^{\varepsilon}.
\end{equation}
This is the essential element in the contraction proof for intermediate scales.
\end{Rem}

\section{Applications}
\label{sec:applications}

In this section,
we show that our results apply to several statistical mechanical
models above their upper critical dimensions.
The applications are to: self-avoiding walk
for $d>4$, percolation for $d>6$, spin models (Ising, XY, $1$- and $2$-component $|\varphi|^4$)
for $d>4$, and lattice trees for $d>8$.
For each of these we: $(1)$~define the model,
$(2)$~verify that its two-point function satisfies Definition~\ref{Def:G},
$(3)$~verify Assumption~\ref{Ass:G}, and $(4)$~verify Assumption~\ref{Ass:E}.
Throughout Section~\ref{sec:applications}, the kernel $J$ is any admissible kernel as in Definition~\ref{Def:J}$(i)$--$(ii)$.
Definition~\ref{Def:J}$(iii)$ does not play a role in Section~\ref{sec:applications}, except for lattice trees in Section~\ref{sec:LT} (it is also used indirectly in Section~\ref{sec:AssE}
via application of Theorem~\ref{thm:maintheorem}).

Most of the analysis in Section~\ref{sec:applications} is well-known from literature going back to the
1980s.  We include it for clarity and completeness, and to illustrate
the parallels between the various models.
Although we work on $\mathbb Z^d$ because that is the setting of our
results, the verification of the bounds of Assumptions~\ref{Ass:G} and \ref{Ass:E} presented here apply to general transitive graphs. If the model has a small parameter as part of its definition, then
the nearest-neighbour $J$ of \eqref{eq:Jnn} can be used.
If there is no small parameter in the definition of the model, then the verification
of Assumption~\ref{Ass:E} requires us to use spread-out models with sufficiently large $\sigma_J$.

The verification of Assumption~\ref{Ass:E} is deferred to Section~\ref{sec:AssE}
for the sake of efficiency, since it is similar for all models under
consideration.

\subsection{Self-avoiding walk}
\label{sec:SAW}

\subsubsection{The model}
\label{sec:SAWmodel}

Let $d \ge 1$.
Detailed introductions to the self-avoiding walk on $\mathbb Z^d$ can be found in
\cite{MS93,BDGS12}.
To define the model, we introduce the  \emph{repulsion parameter}
$\lambda\in [0,1]$. Since $\lambda$ is fixed, we omit it from the notation.
For $x,y\in\mathbb{Z}^d$, we write $x \sim y$ if $J_{y-x} >0$. An $n$-\emph{step}
\emph{walk} from $x$ to $y$ is a sequence $(\gamma(i))_{0\leq i \leq n}$ with $\gamma(0) = x$, $\gamma(n) =y$, and $\gamma(i) \sim \gamma(i+1)$ for all $0\leq i \leq n-1$. Let $\mathcal{W}_0$ be the set of all walks in $\mathbb Z^d$ with $x_0=0$. We denote the length of an $n$-step walk $\gamma$ by $|\gamma|=n$.

For $\gamma\in \mathcal{W}_0$, $0\leq s<t\leq |\gamma|$, and $\beta \ge 0$, we define
\begin{align}
    U_{s,t}(\gamma)&:=-\lambda\mathds{1}_{\gamma(s)=\gamma(t)},
\\
\label{eq: defrho}
    \rho(\gamma) &:=\prod_{0\leq s<t\leq |\gamma|}\big(1+U_{s,t}(\gamma)\big),
    \\
    (\beta J)^{\gamma} &:=\prod_{0\leq i\leq |\gamma|-1}\beta J_{\gamma(i+1)-\gamma(i)}.
\end{align}
The \emph{two-point function} is defined, for every $\beta\geq0$ and every
$x\in \mathbb Z^d$ by
\begin{equation}\label{eq:def 2pt wsaw}
G_\beta(x):=\sum_{\gamma: 0\rightarrow x}(\beta J)^{\gamma}\rho(\gamma).
\end{equation}
When $\lambda=0$, we recover the Green function of the random walk with step distribution $J$.
The case $\lambda=1$ is the \emph{strictly} self-avoiding walk model.
When $\lambda \in (0,1)$, we have the \emph{weakly} self-avoiding walk
(also known as the \emph{Domb--Joyce model}).

We define
$c_n := \sum_{\gamma \in \mathcal W_0: |\gamma|=n} J^{\gamma}\rho(\gamma)$.
This sequence satisfies the inequality $c_{m+n} \le c_mc_n$, from which we conclude
by Fekete's Lemma (see \cite[Lemma~1.2.2]{MS93}) that $c_n^{1/n}$ approaches a
limit
\begin{equation}
\label{eq:Fekete}
    \lim_{n\to\infty}c_n^{1/n} = \inf_{n \ge 1} c_n^{1/n}.
\end{equation}
This limit can be shown to lie in $(0,\infty)$.
Its reciprocal is defined to
be the \emph{critical point} $\beta_c$.
The \emph{susceptibility} is defined by
\begin{equation}
    \chi(\beta) = \sum_{x\in\mathbb Z^d}G_\beta(x) = \sum_{n=0}^\infty c_n \beta^n.
\end{equation}
By \eqref{eq:Fekete}, $\beta_c$ is the radius of convergence of $\chi$. In particular, $G_\beta(x)$ is finite for $\beta\in [0,\beta_c)$.
Also, since $c_n \ge \beta_c^{-n}$, we have $\chi(\beta) \ge (1-\beta/\beta_c)^{-1}$ for $\beta \le \beta_c$.

\subsubsection{Verification of Definition~\ref{Def:G}}
\label{sec:SAW-DefG}

We verify the conditions imposed by Definition~\ref{Def:G}:

\begin{enumerate}
	\item[$(i)$] (Initial condition.) It is immediate from the definition that $G_0=\delta_{0}$.
	\item[$(ii)$] (Regularity.) For every $x\in \mathbb Z^d$, the function
    $\beta\in [0,\beta_c)\mapsto G_\beta(x)$ is a power series in $\beta$ with positive
    coefficients, so it is monotone and differentiable.
	\item[$(iii)$] (Symmetry.) Since $J$ is $\mathbb Z^d$-symmetric, so is  $G_\beta$.
	\item[$(iv)$] (Exponential decay.)
Fix $\beta\in [0,\beta_c)$ and let $\beta_1=\frac 12 (\beta+\beta_c)$.
A walk from $0$ to $x$ must take at least $|x|/R_j$ steps, so
\begin{equation}
    G_\beta(x) \le \sum_{n=|x|/R_j}c_n\beta^n
    \le
    \Big(\frac{\beta}{\beta_1}\Big)^{|x|/R_J} \chi(\beta_1).
\end{equation}
Since $\beta<\beta_1$, the right-hand side decays exponentially in $|x|$.
    \item[$(v)$] (Limit as $\beta\nearrow\beta_c$ when $\beta_c<\infty$.)
    By the Monotone Convergence Theorem,
     $\lim_{\beta \uparrow \beta_c}G_{\beta}(x) = G_{\beta_c}(x)$.
	\end{enumerate}

\subsubsection{Verification of Assumption~\ref{Ass:G}}
\label{sec:SAW-AssG}

\begin{proof}[Proof of \eqref{eq:SL assumption}.]
For $a,b\in \mathbb R$, we have $b^{n}-a^{n} =(b-a) \sum_{k = 0}^{n-1}b^{k}a^{n-1-k}$.
As a result,
		\begin{align}
		G_{\beta}(x)-G_{\beta'}(x) &= \sum_{\gamma:0 \rightarrow x}\big((\beta J)^{\gamma}-(\beta' J)^{\gamma}\big)\rho(\gamma)
\nonumber\\
		&= \sum_{\gamma:0 \rightarrow x} \left( (\b-\b') \sum_{k = 0}^{|\gamma|-1}(\b' J )^{\gamma[0:k]}J_{\gamma(k+1)-\gamma(k)}(\b J )^{\gamma[k+1:|\gamma|]} \right)\rho(\gamma)
\nonumber\\
		&=(\b-\b')\sum_{u,v \in \mathbb{Z}^d} \sum_{\substack{\gamma_1 : 0 \rightarrow u \\\gamma_2: v \rightarrow x}}(\beta' J)^{\gamma_1}
J_{v-u}(\beta J)^{\gamma_2} \rho(\gamma_1 \circ (uv)\circ \gamma_2).
\label{eq: SL step 1}
		\end{align}
By definition, $\rho(\gamma_1 \circ (uv) \circ \gamma_2) \leq \rho(\gamma_1)
\rho(\gamma_2)$, so
\begin{equation}
\label{eq:SAW-rho-ub}
			\sum_{\gamma_2 : v\rightarrow x}(\beta J)^{\gamma_2} \rho(\gamma_1 \circ (uv)\circ \gamma_2) \leq  \rho(\gamma_1) G_{\beta}(x-v),
\end{equation}
and therefore
		\begin{align}
		G_{\beta}(x)-G_{\b'}(x)& \leq
    (\beta-\beta')
    \sum_{\substack{u,v \in \mathbb{Z}^d \\ \gamma:0 \rightarrow u } } (\beta' J)^{\gamma} J_{v-u} \rho(\gamma)G_{\beta}(x-v)
\nonumber\\
		&= (\beta-\beta') (G_{\beta'}*J*G_\beta)(x).
\label{eq:I1SAW}
		\end{align}
This proves \eqref{eq:SL assumption}.
Figure~\ref{fig:wsaw} illustrates the upper bound in \eqref{eq:I1SAW}.
\end{proof}

\begin{figure}
	\begin{center}
		\includegraphics{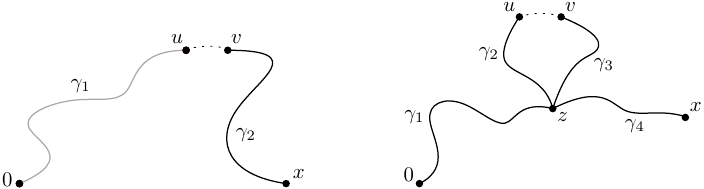}
		\caption{Diagrammatic representations for \eqref{eq:I1SAW} (left)
and \eqref{eq: proof diff ineq wsaw 1} (right). On the left,
the grey path has parameter $\beta'$ and the black has $\beta$. The dotted lines represent a single-step walk.
}
		\label{fig:wsaw}
	\end{center}
\end{figure}

We define the \emph{open bubble diagram}
 \begin{equation}
 B^{\rm o}(\beta):=
 (G_\beta*J*G_\beta)(0).
 \end{equation}
and verify the lower differential inequality \eqref{eq:Diff inequ assumption} with
\begin{equation}
\label{eq:HSAW}
	H_\beta(x) =\lambda  B^{\rm o}(\beta)  \delta_{0}(x) .
\end{equation}
By definition, $H_0=0$, $H_\beta(x)=H_\beta(-x)$, and $H_\beta$ is increasing
and continuous on $[0,\beta_c)$, as required by Assumption~\ref{Ass:G}.
Also, since $H_\beta$ is supported on $\{0\}$, it certainly decays exponentially.

Since we have verified \eqref{eq:SL assumption},
it follows from \eqref{eq:dGub} that
 \begin{align}
 	 \partial_\beta G_{\beta}&\leq G_\beta*J*G_\beta .
 \label{eq:diff ineq 1 WSAW intro}
  \end{align}
We will use this in the proof of the complementary lower bound
\begin{equation}
\label{eq:SAWlb}
    \partial_\beta G_{\beta} \ge G_\beta*J*G_\beta - G_\beta*H_\beta*G_\beta,
\end{equation}
which is \eqref{eq:Diff inequ assumption}.

\begin{proof}[Proof of \eqref{eq:Diff inequ assumption}.]
By writing $|\gamma|=
\sum_{z\in \mathbb Z^d}\sum_{i=1}^{|\gamma|}\mathds{1}_{\gamma(i)=z}$,
we see that
\begin{align}
\label{eq:dGlbSAW}
    \partial_\beta G_{\beta}(x)
    &=
    \sum_{v\in \mathbb Z^d}
    \sum_{\substack{\gamma:0\to x \\|\gamma|\ge 1}}\beta^{|\gamma|-1}J^\gamma
    \rho(\gamma)
    \sum_{i=1}^{|\gamma|}\mathds{1}_{\gamma(i)=v}
    \nonumber \\
    & =
    \sum_{u,v\in \mathbb Z^d}
    \sum_{\substack{\gamma_1 : 0 \rightarrow u \\\gamma_2 : v \rightarrow x}}
    (\beta J)^{\gamma_1} J_{v-u} (\beta J)^{\gamma_2}
    \rho(\gamma_1 \circ  (uv) \circ \gamma_2).
\end{align}
We apply the inequality
\begin{equation}
\label{eq : lower bound path decomp wsaw}
\rho(\gamma_1\circ (uv)\circ \gamma_2)
\geq \rho(\gamma_1)\rho(\gamma_2)
-\lambda \rho(\gamma_1)\rho(\gamma_2)\sum_{\substack{0\leq i \leq |\gamma_1|\\0\leq j\leq |\gamma_2|}}\mathds{1}_{\gamma_1(i)=\gamma_2(j)}.
\end{equation}
Insertion of the first term on the right-hand side
into \eqref{eq:dGlbSAW} gives the first term
$(G_\beta*J*G_\beta)(x)$ in our desired bound
\eqref{eq:SAWlb}.

To arrive at the remaining term in \eqref{eq:SAWlb},
we decompose according to the common value $z$ of $\gamma_1(i)$ and $\gamma_2(j)$. The
subtracted term in \eqref{eq : lower bound path decomp wsaw} becomes
		\begin{align}
			&\lambda \sum_{u,v\in \mathbb{Z}^d}\sum_{\substack{\gamma_1 : 0 \rightarrow u\\\gamma_2 : v \rightarrow x}}(\beta J)^{\gamma_1}J_{v-u}(\beta J)^{\gamma_2}\rho(\gamma_1)\rho(\gamma_2)\sum_{\substack{0\leq i \leq |\gamma_1|\\0\leq j\leq |\gamma_2|}}\mathds{1}_{\gamma_1(i)=\gamma_2(j)}
\nonumber\\
			&\leq  \lambda\sum_{u,v,z \in \mathbb{Z}^d}\sum_{\substack{\gamma_1 : 0 \rightarrow z\\\gamma_2 : z \rightarrow u\\\gamma_3 : v \rightarrow z\\\gamma_4 : z \rightarrow x}}(\beta J)^{\gamma_1  \circ \gamma_2}J_{v-u}(\beta J)^{\gamma_3 \circ\gamma_4}\rho(\gamma_1  \circ \gamma_2)\rho(\gamma_3  \circ\gamma_4)
\nonumber\\
			&\leq  \lambda\sum_{u,v,z \in \mathbb{Z}^d} G_{\b}(z)G_{\b}(u-z)J_{v-u} G_{\b}(z-v)G_{\b}(x-z)
\nonumber\\
						&= \lambda B^{\rm o}(\beta)  \sum_{z \in \mathbb{Z}^d}G_{\b}(z)  G_{\b}(x-z)
    = \lambda B^{\rm o}(\beta) (G_\beta*G_\beta)(x).
\label{eq: proof diff ineq wsaw 1}
		\end{align}
The right-hand side is exactly $(G_\beta*H_\beta*G_\beta)(x)$ with $H_\beta$
given by \eqref{eq:HSAW}.
\end{proof}

\subsection{Continuous-time weakly self-avoiding walk}\label{sec: CTWSAW}

We now verify Assumption~\ref{Ass:G} for the continuous-time weakly
self-avoiding walk. Since the argument closely follows that of the
(discrete-time) weakly self-avoiding walk, we only sketch the details.

Let $d \ge 1$.
The study of the continuous-time weakly self-avoiding walk on $\mathbb Z^d$
(also known as the \emph{discrete Edwards model})
goes back to \cite{BFF84}.  The model is interesting even for $d=1$, e.g., \cite{Liu24}.

To define the model, we first consider the continuous-time random walk $X$ on $\mathbb Z^d$ which takes steps
according to an admissible kernel $J$, with ${\rm Exp}(1)$ holding times.
In other words, the random time spent at a vertex before making a next step
has density $e^{-t} \D t$, and the holding times at each step are independent
of each other and of the choice of next vertex.
For a walk with trace $\gamma$ taking $|\gamma|$
steps, let $T_0,\ldots,T_{|\gamma|}$ denote the independent ${\rm Exp}(1)$ holding
times, and let $T=\sum_{i=0}^{|\gamma|}T_i$ denote the total time of $\gamma$.

The \emph{local time} at a vertex $y$ visited by $\gamma$ is the random
variable
\begin{equation}
    L_{y,\gamma} = \sum_{i=0}^{|\gamma|} T_i \1_{\gamma(i)=y}
    =
    \int_0^T \1_{X(t)=y} \D t.
\end{equation}
The \emph{intersection local time}
\begin{equation}
    I(\gamma) :=
    \sum_{y\in \mathbb Z^d} L_{y,\gamma}^2 = \int_0^T \int_0^T \1_{X(s)=X(t)} \D s \D t
\end{equation}
gives an indication of the total time that
the walk $\gamma$ spends intersecting itself.
Given $\lambda >0$, we define a weight $\rho(\gamma)$ by
\begin{equation}
    \rho(\gamma) := \mathbb E \, (e^{-\lambda I(\gamma)}),
\end{equation}
where $\mathbb E$ denotes the expectation over the holding times.

When $y$ is visited $n_y$ times by the walk $\gamma$, the local time $L_{y,\gamma}$ has a ${\rm Gamma}(n_y,1)$ distribution.
Therefore, in terms of the
measure on $(\mathbb R^+)^{\mathbb Z^d}$ defined by
\begin{align}
\label{eq:dnu-CTWSAW}
    \mathrm{d}\nu_{\gamma}(\mathbf{t}) &= \prod_{v \in \mathbb Z^d } \left(\mathrm{1}_{n_v(\gamma)=0}\delta_0(t_v)\mathrm{d}t_v+ \mathrm{1}_{n_v(\gamma)\geq 1} \frac{t_v^{n_v(\gamma)-1}}{(n_v(\gamma)-1)!}\mathrm{d}t_v\right),
\end{align}
the weight of $\gamma$ is equal to
\begin{equation}
    \rho(\gamma)
    = \mathbb E \, (e^{-\lambda I(\gamma)})
    =
    \int
    e^{-\sum_y  t_y} e^{-\lambda\sum_y  t_y^2}
    \D \nu_{\gamma}(\mathbf{t})
.
\end{equation}
The \emph{two-point function} is defined by
\begin{equation}\label{eq:2pt CTWSAW}
    G_\beta(x) := \sum_{\gamma: 0 \to x} (\beta J)^\gamma \rho (\gamma).
\end{equation}

\smallskip\noindent
\emph{Verification of Definition~\textup{\ref{Def:G}}.}
This is identical to the verification for the self-avoiding walk model in Section~\ref{sec:SAW-DefG}. Indeed, in Section~\ref{sec:SAW-DefG}, we did not use the precise form of $\rho (\gamma)$, but only the fact that it does not depend on $\beta$. This remains true in
\eqref{eq:2pt CTWSAW}.

\begin{proof}[Proof of \eqref{eq:SL assumption}.]
Suppose that $\gamma = \gamma_1 \circ (uv) \circ \gamma_2$.
The intersection local time obeys the inequality
\begin{equation}
    I(\gamma) \ge I(\gamma_1)+I(\gamma_2).
\end{equation}
Since the holding times of $\gamma_1$ and $\gamma_2$ are independent, so are the random variables $I(\gamma_1)$ and $I(\gamma_2)$.
This leads to the inequality $\rho(\gamma_1 \circ (uv) \circ \gamma_2) \leq \rho(\gamma_1)
\rho(\gamma_2)$ that was used to verify Assumption~\ref{Ass:G} in
Section~\ref{sec:SAW-AssG}.
Then \eqref{eq:SL assumption} follows exactly as in Section~\ref{sec:SAW-AssG}.
\end{proof}

\begin{proof}[Proof of \eqref{eq:Diff inequ assumption}.]
We start from \eqref{eq:dGlbSAW}, with the new interpretation of the weight $\rho$. Let $\gamma = \gamma_1 \circ (uv) \circ \gamma_2$.
The error in the upper bound $\rho(\gamma) \leq \rho(\gamma_1) \rho(\gamma_2)$
used to prove \eqref{eq:SL assumption} arises from
\begin{align}
    e^{-\lambda I(\gamma_1)}e^{-\lambda I(\gamma_2)} -
    e^{-\lambda I(\gamma)}
    & =
    e^{-\lambda I(\gamma_1)}e^{-\lambda I(\gamma_2)}
    \Big[1-e^{ -\lambda [I(\gamma)-(I(\gamma_1)+I(\gamma_2))] } \Big]
    \nonumber \\ & \le
    e^{-\lambda I(\gamma_1)}e^{-\lambda I(\gamma_2)}
     \lambda [I(\gamma)-(I(\gamma_1)+I(\gamma_2))]   .
\end{align}
By definition, $L_{z,\gamma} = L_{z,\gamma_1}+L_{z,\gamma_2}$. Therefore,
\begin{equation}
    I(\gamma)-(I(\gamma_1)+I(\gamma_2)) = \sum_{z \in \mathbb Z^d} (L_{z,\gamma_1}+L_{z,\gamma_2})^2-\left(L_{z,\gamma_1}^2+L_{z,\gamma_2}^2\right) =  2 \sum_{z\in \mathbb Z^d} L_{z,\gamma_1}L_{z,\gamma_2} .
\end{equation}
As a consequence,
\begin{equation}
    \rho(\gamma_1) \rho(\gamma_2) - \rho(\gamma)
    \le
    2\lambda\sum_{z\in \mathbb Z^d}
    \Big( \mathbb E \big(e^{-\lambda I(\gamma_1)} L_{z,\gamma_1}\big) \Big)
    \Big( \mathbb E \big(e^{-\lambda I(\gamma_2)} L_{z,\gamma_2}\big) \Big).
\end{equation}
Now, we sum over $\gamma$ as in \eqref{eq:dGlbSAW}.  The result is
\begin{align}
    &(G_\beta * J * G_\beta)(x) - \partial_\beta G_\beta(x)
    \nonumber \\ & \quad \le
    2\lambda     \sum_{u,v,z\in \mathbb{Z}^d}
    \sum_{\gamma_1 : 0 \rightarrow u}
    (\beta J)^{\gamma_1}
    \Big( \mathbb E \big(e^{-\lambda I(\gamma_1)} L_{z,\gamma_1}\big) \Big)
    J_{v-u}
    \sum_{\gamma_2 : v \rightarrow x}
    (\beta J)^{\gamma_2}
    \Big( \mathbb E \big(e^{-\lambda I(\gamma_2)} L_{z,\gamma_2}\big) \Big).
\end{align}
We appeal to \cite[Lemma~2.1]{BFS82} to see that
\begin{equation}
    \sum_{\gamma_1 : 0 \rightarrow u}
    (\beta J)^{\gamma_1}
    \Big( \mathbb E \big(e^{-\lambda I(\gamma_1)} L_{z,\gamma_1}\big) \Big)
        \le
    G_\beta(z) G_\beta(u-z),
\end{equation}
and similarly for the sum over $\gamma_2$.  We therefore obtain
\begin{align}
    &(G_\beta * J * G_\beta)(x) - \partial_\beta G_\beta(x)
    \nonumber \\ & \quad \le
    2\lambda     \sum_{z\in \mathbb{Z}^d}
    G_\beta(z)G_\beta(x-z)\sum_{u,v\in \mathbb Z^d}G_\beta(u-z)J_{v-u}G_\beta(z-v).
\end{align}
By replacing $u,v$ by $u-x,v-z$, this leads to
\begin{align}
    &(G_\beta * J * G_\beta)(x) - \partial_\beta G_\beta(x)
    \le
    2\lambda    ( G_\beta*H_\beta*G_\beta)(x)
\end{align}
with
\begin{equation}
\label{eq:HCTWSAW}
    H_\beta(x) = 2\lambda \delta_{0}(x) (G_\beta*J*G_\beta)(0).
\end{equation}
This completes the proof.
\end{proof}

\subsection{Bernoulli percolation}
\label{sec:perc}

\subsubsection{The model}

For an introduction to percolation theory, see \cite{Grim99}.
Let $d \ge 2$.
We consider Bernoulli bond percolation on the infinite graph whose vertex
set is $\mathbb Z^d$ and whose edge set $\mathcal E=\mathcal E_J$ consists of pairs $\{x,y\}$ with
$J_{y-x}>0$.  Let $\beta\in [0,(\max_{x\in \mathbb Z^d} J_x)^{-1}]$.
Edges are independently open with probability $\beta J_{y-x}$
and otherwise are closed.  We write $\{x\connect{}y\}$
for the event that $x$ and $y$ are connected by a path consisting of open bonds,
and define the \emph{two-point function}
 \begin{equation}\label{eq:def 2pt perco}
 	G_\beta(x,y)=G_\beta(y-x):=\mathbb P_\beta[x\connect{}y].
 \end{equation}

\subsubsection{Verification of Definition~\ref{Def:G}}

For $d \ge 2$,
the properties listed in Definition~\ref{Def:G} are standard facts about
percolation \cite{Grim99}.  The critical value $\beta_c$ separates the subcritical regime,
where $G_\beta$ decays exponentially, from the supercritical regime, where
there is a positive probability for the existence of an infinite connected
cluster \cite{Mens86,AB87}. We assume familiarity with two basic techniques: the BK inequality (van den Berg--Kesten) and Russo's formula \cite[Chapter~2]{Grim99}.

 \subsubsection{Verification of Assumption~\ref{Ass:G}}

\begin{proof}[Proof of \eqref{eq:SL assumption}]
A version of \eqref{eq:SL assumption} appeared in \cite[Lemma~2.4]{Hutc19_hyperbolic}.
We use the standard increasing coupling $\mathbb P$, as follows.
First, we assign independent uniform random variables $\eta_{u,v}$ in $[0,1]$ to each edge $\{u,v\}$.
Given a realisation of these random variables, and given $\beta\in [0,(\max_{x\in \mathbb Z^d} J_x)^{-1}]$, we define a percolation configuration $\omega_\beta\in \{0,1\}^{\mathcal E}$ as follows: $\omega_\beta(\{u,v\})=1$  if
$\eta_{u,v}<  \beta J_{v-u}$ (the bond $\{u,v\}$ is \emph{open}), and otherwise $\omega_\beta(\{u,v\})=0$. Below, we also view $\omega_\beta$ as a subgraph of $(\mathbb Z^d,\mathcal E)$ of vertex set $\mathbb Z^d$ and edge set $\{\{u,v\}\in \mathcal E: \omega_\beta(\{u,v\})=1\}$. We also define $\{x \connect{A\:} y\}$
 to be the event that $x$ is connected to $y$
by a path which does not pass through any vertex in $A^c$.
By definition of the coupling, if $\beta<\beta'$ then $\omega_{\beta'}$ is a subgraph of $\omega_\beta$, and, for every $x\in \mathbb Z^d$,
		\begin{equation}\label{eq:proof sl perco 1}
			G_\beta(x)-G_{\beta'}(x) = \mathbb{P}[ \{0\connect{\omega_\beta\:}x\}\setminus\{0\connect{\omega_{\beta'}\:} x\}].
		\end{equation}

For $\beta \in [0,(\max_{x\in \mathbb Z^d} J_x)^{-1}]$ and $z\in \mathbb Z^d$, we define $\mathcal C_{\beta}(z)$ to be the cluster of $z$ in $\omega_\beta$. For $\{u,v\}\in \mathcal E$, we define $\omega^{\{u,v\}}_\beta$ to be the percolation configuration obtained from $\omega_\beta$ by setting $\omega_{\beta}(\{u,v\})=0$,
and we define $\mathcal C_\beta^{\{u,v\}}(z)$ to be the cluster of $z$ in $\omega_\beta^{\{u,v\}}$.
For fixed $\beta'<\beta$ and $x\in \mathbb Z^d$, we claim that
the event on the right-hand side of \eqref{eq:proof sl perco 1} satisfies
\begin{multline}
\label{eq:proof perco sl 2}
 \{0\connect{\omega_\beta\:} x\}\setminus\{0\connect{\omega_{\beta'}\:} x\}
\\\subset \bigcup_{\{u,v\}\in \mathcal E} \Big\{u \in \mathcal{C}_{\beta'}(0) \Big\} \cap\{\omega_{\beta'}(\{u,v\})=0, \: \omega_{\beta}(\{u,v\})=1\} \cap \Big\{v\connect{\omega_{\beta}\setminus \mathcal{C}_{\beta'}(0) \:} x \Big\}.
 \end{multline}
The claim is justified as follows.
For every configuration in $\{0\connect{\omega_\beta\:} x\}\setminus\{0\connect{\omega_{\beta'}\:} x\}$, there must exist $u\in \mathcal C_{\beta'}(0)$ and $v \notin \mathcal C_{\beta'}(0)$ such that $\omega_\beta(\{u,v\})=1$ and  $v$ is connected to $x$ in $\omega_{\beta}$ without using the vertices in $\mathcal{C}_{\beta'}(0)$. This can be seen by exploring an open self-avoiding path from $0$ to $x$ in $\omega_\beta$ and marking its edge $\{u,v\}$ for which $u$ is the last vertex of $\mathcal C_{\beta'}(0)$ visited by this path.

Since $\omega_{\beta'}(\{u,v\})=0$, we can replace $\mathcal C_{\beta'}(0)$ by $\mathcal C_{\beta'}^{\{u,v\}}(0)$ in \eqref{eq:proof perco sl 2}. A union bound then gives
\begin{multline}\label{eq:proof perco sl 3}
	G_\beta(x)-G_{\beta'}(x) \\\leq \sum_{\{u,v\}\in \mathcal E} \mathbb P\Big[\Big\{u \in \mathcal{C}_{\beta'}^{\{u,v\}}(0) \Big\} \cap\{\omega_{\beta'}(\{u,v\})=0, \: \omega_{\beta}(\{u,v\})=1\} \cap \Big\{v\connect{\omega_{\beta}\setminus \mathcal{C}_{\beta'}^{\{u,v\}}(0) \:} x \Big\}\Big].
\end{multline}
Given $\{u,v\}\in \mathcal E$, we condition on the cluster
$\mathcal{C}^{^{\{u,v\}}}_{\beta'}(0)$.  This gives
\begin{multline}\label{eq:proof perco sl 4}
    \mathbb P\Big[\Big\{u \in \mathcal{C}_{\beta'}^{\{u,v\}}(0) \Big\} \cap\{\omega_{\beta'}(\{u,v\})=0, \: \omega_{\beta}(\{u,v\})=1\} \cap \Big\{v\connect{\omega_{\beta}\setminus \mathcal{C}_{\beta'}^{\{u,v\}}(0) \:} x \Big\}\Big]  \\= \sum_{C \ni 0,u} \mathbb P[\mathcal{C}^{^{\{u,v\}}}_{\beta'}(0) = C] \mathbb{P}[\omega_{\beta'}(\{u,v\})=0, \: \omega_{\beta}(\{u,v\})=1] \mathbb P [v\connect{\omega_{\beta}\setminus C\:} x ] ,
\end{multline}
because the three events on the right hand side of \eqref{eq:proof perco sl 4} depend on disjoint sets of edges and are therefore independent.
By definition of the coupling, $\mathbb{P}[\omega_{\beta'}(\{u,v\})=0, \: \omega_{\beta}(\{u,v\})=1] = \mathbb P [\eta_{u,v} \in [\beta' J_{v-u} ,\beta J_{v-u}) ]  = \beta-\beta'$,
and by inclusion of events,
 \begin{align}
 \mathbb P [v\connect{\omega_{\beta}\setminus C \:} x ]  &\leq \mathbb P[v\connect{\omega_{\beta}\:}x]=G_{\beta}(x-v),
  \\ \sum_{C \ni 0,u} \mathbb P[\mathcal C^{\{u,v\}}_{\beta'}(0) &= C]=\mathbb P[u\in \mathcal C_{\beta'}^{\{u,v\}}(0)] \leq G_{\beta'}(u).
 \end{align}
The last three observations, combined with \eqref{eq:proof perco sl 3}--\eqref{eq:proof perco sl 4},
complete the proof.
\end{proof}

\begin{figure}
\begin{center}
\includegraphics{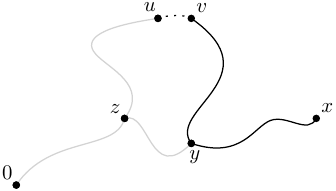}
\end{center}
\caption{Depiction of $(G_\beta*H_\beta*G_\beta)(x)$ for percolation.
Lines represent two-point functions and the gap between $u,v$ represents
a factor $J_{v-u}$.
The vertices $u,v,y,z$ are summed over $\mathbb Z^d$. }
\label{fig:perco H}
\end{figure}

We will prove the lower differential inequality \eqref{eq:Diff inequ assumption} with
 \begin{equation}
 \label{eq:Hperc}
H_\beta:= (G_\beta*J*G_\beta)\cdot G_\beta .
 \end{equation}
By definition, $H_0(x)=J_x\delta_0(x)=0$, and $H_\beta(-x)=H_\beta(x)$ for every $x \in \mathbb Z^d$
and all $\beta\in [0,\beta_c)$.
Since $G_\beta$ decays exponentially for $\beta\in (0,\beta_c)$,
the same is true for $H_\beta$.

\begin{proof}[Proof of \eqref{eq:Diff inequ assumption}]
The lower differential inequality \eqref{eq:Diff inequ assumption}, with
$H_\beta$ given by \eqref{eq:Hperc},
is a minor modification of a differential inequality for the susceptibility
that was first proved in \cite{AN84} before the advent of the BK inequality.
For a more modern proof, see \cite[Proposition~9.11]{Slad06}.
Although we have no new insights for the proof, we provide the argument for the sake of completeness. Let $0\leq \beta<\beta_c$.
By Russo's formula,
\begin{align}
    \partial_\beta  G_\beta(x) &= \sum_{u,v\in \mathbb Z^d}J_{v-u} \mathbb P_{\beta} [\{u,v\} \text{ closed and pivotal for } 0 \connect{} x] \nonumber \\
    &= \sum_{u,v\in \mathbb Z^d}J_{v-u} \mathbb P_{\beta} [ \{0 \connect{} u\}\cap  \{v \connect{} x\}\cap  \{u \connect{} v\}^c]
\end{align}
(to focus on the key ideas, we do not address the subtlety that Russo's formula initially applies only to events depending on
finitely many bonds).
It suffices to prove that
    \begin{multline}
    \label{eq:perclb}
        \mathbb P_{\beta} [ \{0 \connect{} u\}\cap  \{v \connect{} x\}\cap  \{u \connect{} v\}^c] \geq G_\beta(u) G_\beta(x-v)
        \\- \sum_{y,z \in \mathbb Z^d} G_\beta(z)G_\beta(u-z) G_\beta(y-z) G_\beta(y-v) G_\beta(x-y),
    \end{multline}
since this yields
\begin{equation}
    \partial_\beta G_\beta(x)
    \ge
    (G_\beta*(J-H_\beta)* G_\beta)(x) .
\end{equation}
Connections arising in  the proof of \eqref{eq:perclb}
are illustrated in Figure~\ref{fig:perco H}.

Given a set $A$ of vertices, we define $\{x\connect{}y \textup{ using }A\}$ to be the event that $x$ is
connected to $y$ and that every path that realises the connection must contain a vertex
in $A$.

To prove \eqref{eq:perclb},
we condition on the cluster $\mathcal C(u)$ of $u$.
This gives,
    \begin{align}\label{eq:lemperco1}
              \mathbb P_{\beta} [ \{0 \connect{} u\}\cap  \{v \connect{} x\}\cap  \{u \connect{} v\}^c] = \sum_{C\ni 0,u} \mathbb P_{\beta}[ \mathcal C(u)=C]\mathbb P_\beta[v\connect{C^c\:}x],
              \end{align}
where we used the facts that $\mathbb P_\beta[v\connect{C^c\:}x]=0$ if $v\in C$, and that the events $\{\mathcal C(u)=C\}$ and $\{v\connect{C^c\:}x\}$ are independent.
It follows by definition of the events that
\begin{equation}
	\mathbb P_\beta[v\connect{C^c\:}x]=G_\beta(x-v)-\mathbb P[v\connect{}x \textup{ using }C].
\end{equation}
If $\{v\connect{}x \textup{ using }C\}$ occurs, there must exist $y\in C$ such that $\{v\connect{} y\}\circ \{y\connect{} x\}$ occurs.
Therefore, by the BK inequality,
    \begin{align}\label{eq:lemperco2}
        \mathbb P_\beta[v\connect{C^c\:}x] \geq G_\beta(x-v)- \sum_{y \in \mathbb Z^d}\mathds{1}_{y\in C} G_\beta(y-v) G_\beta(x-y).
    \end{align}
    We insert \eqref{eq:lemperco2} into \eqref{eq:lemperco1} and obtain
    \begin{multline}\label{eq:lemperco3}
    	 \mathbb P_{\beta} [ \{0 \connect{} u\}\cap  \{v \connect{} x\}\cap  \{u \connect{} v\}^c]\geq G_\beta(u)G_\beta(x-v)
    	 -\\\sum_{y\in \mathbb Z^d}G_\beta(y-v)G_\beta(x-y)\sum_{C\ni 0,u,y}\mathbb P_\beta[\mathcal C(u)=C].
    \end{multline}
    Finally, we observe that $\sum_{C\ni 0,u,y}\mathbb P_\beta[\mathcal C(u)=C]=\mathbb P_\beta[0,u,y\textup{ lie in the same cluster}]$. If this event occurs, then there must be $z\in \mathbb Z^d$ such that $\{ 0 \connect{} z\}\circ\{ z \connect{} u\} \circ\{ z\connect{}y\}$. Combining this observation with the BK inequality gives
    \begin{equation}\label{eq:lemperco4}
    	\mathbb P_\beta[0,u,y\textup{ lie in the same cluster}]\leq \sum_{z\in \mathbb Z^d}G_\beta(z)G_\beta(u-z)G_\beta(y-z).
    \end{equation}
    The combination of \eqref{eq:lemperco4} and \eqref{eq:lemperco3} completes the proof.
\end{proof}

\subsection{$1$- and $2$-component $|\varphi|^4$ models}
\label{sec:phi4}

\subsubsection{The model}

Let $d\geq 2$. We define the $1$- and $2$-component $|\varphi|^4$ models on $\mathbb Z^d$,
as follows. Let $n\in \{1,2\}$ and let $x\cdot y$ denote the dot product of $x,y\in\mathbb R^n$. We
denote the Euclidean norm on $\mathbb R^n$ by $|\cdot|_2$; context will distinguish this from the norm on $\mathbb Z^d$ which is denoted in the same way. Let $\Lambda\subset \mathbb Z^d$ be finite, $\beta\geq 0$, $J$ an admissible interaction, and $F:(\mathbb R^n)^\Lambda\rightarrow \mathbb R$. The $n$-component $\varphi^4$ model on $\Lambda$ is the measure $\langle \cdot\rangle_{\Lambda,\beta}$ on $(\mathbb R^n)^\Lambda$ given by:
\begin{equation}\label{eq:def phi4}
	\langle F(\varphi)\rangle_{\Lambda,\beta}:=\frac{1}{Z_{\Lambda,\beta}}\int_{(\mathbb R^n)^\Lambda}F(\varphi)\exp(-\beta H_{\Lambda}(\varphi))\mathrm{d}s_{\Lambda}(\varphi),
\end{equation}
where
\begin{equation}
	H_{\Lambda}(\varphi):= - \frac 12 \sum_{x,y\in \Lambda}
J_{y-x}(\varphi_x\cdot \varphi_y), \qquad Z_{\Lambda,\beta}:=\int_{(\mathbb R^n)^\Lambda}\exp(-\beta H_{\Lambda}(\varphi))\mathrm{d}s_{\Lambda}(\varphi),
\end{equation}
and for $\lambda>0$ and $\mu\in \mathbb R$,
\begin{equation}
	\mathrm{d}s_{\Lambda}(\varphi):=
    \prod_{x\in \Lambda}g(|\varphi_x|^2_2)\mathrm{d}\varphi_x
    \qquad\text{ with }
    \qquad g(t):=\exp\left(-  \frac 14   \lambda t^2-  \frac 12  \mu t\right).
\end{equation}
Additionally, we denote by $\mathrm{d}m_{\lambda,\mu}(\varphi_x)$ the probability distribution on $\mathbb R^n$ with density proportional to $g(|\varphi_x|_2^2)\mathrm{d}\varphi_x$.
It is a classical consequence of Griffiths' \cite{Grif67} (for $n=1$) or Ginibre's \cite{Gini70} (for $n=2$) inequalities that the sequence of measures $\langle \cdot\rangle_{\Lambda,\beta}$ admits a weak limit as $\Lambda\nearrow \mathbb Z^d$. We denote the limiting measure by $\langle \cdot\rangle_{\beta}$.

The \emph{two-point function} is defined for $\beta\geq 0$ and $x\in \mathbb Z^d$ by
\begin{equation}
	G_\beta(x):=\langle \varphi_0^1\varphi_x^1\rangle_{\beta},
\end{equation}
where $\varphi^1=\varphi$ for $n=1$ and $\varphi=(\varphi^1,\varphi^2)$ for $n=2$.
By definition, $G_0=A\delta_0$ with
\begin{equation}
\label{eq:Aphi4}
	A=A(\lambda,\mu)=\int_{\varphi_0\in \mathbb R^n}(\varphi_0^1)^2\mathrm{d}m_{\lambda,\mu}(\varphi_0). \end{equation}
The fact that $A$ is not necessarily equal to $1$ violates Definition~\ref{Def:G}$(i)$ but this is not a problem, as explained in
Remark~\ref{Rem:A}.
As usual, we set $F_\beta=J*G_\beta$.

\subsubsection{The Brydges--Fr\"ohlich--Spencer random walk representation}

To prepare for the verification of the finite-difference upper
bound \eqref{eq:SL assumption} of Assumption~\ref{Ass:G}, we recall
the Brydges--Fr\"ohlich--Spencer (BFS) random walk expansion \cite{BFS82} (see also \cite{Lamm23} for a recent alternative perspective).
To state the expansion, we introduce the following definitions.

First, given $\mathbf{t}=(t_x)_{x\in \Lambda}\in \mathbb R_+^\Lambda$,
we define a measure on $(\mathbb R^n)^\Lambda$ and its associated
partition function by
\begin{align}
    \mathrm{d}s_{\Lambda,\mathbf{t}}(\varphi) &:= \prod_{x\in \Lambda}  g(|\varphi_x|^2_2+2t_x)\mathrm{d}\varphi_x ,\\
    Z_{\Lambda,\beta}(\mathbf{t}) &:= \int_{(\mathbb R^n)^\Lambda}\exp (-\beta H_{\Lambda}(\varphi) ) \mathrm{d}s_{\Lambda,\mathbf{t}}(\varphi).
\end{align}
We also define a normalised version of $Z_{\Lambda,\beta,J}(\mathbf{t})$ by
\begin{align}
    z_{\Lambda,\beta}(\mathbf{t}) &:= \frac{Z_{\Lambda,\beta}(\mathbf{t})}{Z_{\Lambda,\beta}}.
\end{align}
As in Section~\ref{sec:SAWmodel}, a
\emph{walk} $\gamma=(\gamma(0),\ldots,\gamma(|\gamma|))$ is a sequence of points in $\mathbb Z^d$ satisfying
$J_{\gamma(i+1)-\gamma(i)}>0$.
The \emph{local time} of $\gamma$ at $v\in \Lambda$ is defined by
$\ell_v(\gamma) = \sum_{i = 0}^{|\gamma|}\1_{\gamma(i) = v}$.
As in \eqref{eq:dnu-CTWSAW}, we define a measure on $(\mathbb R^+)^\Lambda$ by
\begin{align}
    \mathrm{d}\nu_{\Lambda,\gamma}(\mathbf{t}) &:= \prod_{v \in \Lambda}\left(\mathrm{1}_{\ell_v(\gamma)=0}\delta_0(t_v)\mathrm{d}t_v+ \mathrm{1}_{\ell_v(\gamma)\geq 1} \frac{t_v^{\ell_v(\gamma)-1}}{(\ell_v(\gamma)-1)!}\mathrm{d}t_v\right).
\end{align}

\begin{Prop}[The BFS expansion \cite{BFS82}]
\label{prop:BFS}
Let $\Lambda \subset \mathbb Z^d$ be finite and $\beta\geq 0$. Then
  \begin{equation}
    \langle \varphi_x^1 \varphi_y^1 \rangle_{\Lambda,\beta} = \sum_{\gamma: x \rightarrow y} (\beta J)^{\gamma} \int z_{\Lambda,\beta}(\mathbf{t}) \mathrm{d}\nu_{\Lambda,\gamma}(\mathbf{t}),
\end{equation}
where the sum runs overs paths $\gamma=(\gamma_1,\ldots,\gamma_{|\gamma|})$, and, as
usual, $(\beta J)^\gamma=\prod_{i=1}^{|\gamma|-1}\beta J_{\gamma_{i+1}-\gamma_i}$.
\end{Prop}

The BFS expansion matches the formula \eqref{eq:2pt CTWSAW} for the two-point function of the continuous-time weakly self-avoiding walk if we define the weight
\begin{equation}
	\rho(\gamma)=\int z_{\Lambda,\beta}(\mathbf{t}) \mathrm{d}\nu_{\Lambda,\gamma}(\mathbf{t}).
\end{equation}
However, the fact that this $\rho$ depends on $\beta$ makes the arguments of Section~\ref{sec:SAW} inapplicable, and we must proceed more delicately.

Our verification of Assumption~\ref{Ass:G} combines the BFS expansion
with the following two  monotonicity properties.
They use the measure $\langle \cdot\rangle_{\Lambda,\beta,\mathbf{t}}$ on $(\mathbb R^n)^\Lambda$ defined by replacing $\mathrm{d}s_{\Lambda}$ by $\mathrm{d}s_{\Lambda,\mathbf{t}}$ in \eqref{eq:def phi4}.

\begin{Lem}\label{lem:phi4 monot BFS}  Let $\Lambda\subset \mathbb Z^d$
be finite and $0\leq \beta'\leq \beta$. Then, for every $\mathbf{t}\in (\mathbb R^+)^\Lambda$,
\begin{equation}
	z_{\Lambda,\beta}(\mathbf{t})\leq z_{\Lambda,\beta'}(\mathbf{t}).
\end{equation}
\end{Lem}

\begin{proof}
We first observe that
\begin{equation}
    \prod_{x\in \Lambda}g(|\varphi_x|^2_2)
    = f(\mathbf{t}) F_{\mathbf{t}}(\varphi)
    \prod_{x\in \Lambda}g(|\varphi_x|^2_2 + 2\mathbf{t}_x),
\end{equation}
with
\begin{equation}\label{eq:defF BFS}
	F_{\mathbf{t}}(\varphi):=\exp\Big(\sum_{x\in \Lambda} \lambda t_x|\varphi_x|^2_2\Big),
    \qquad
    f(\mathbf{t})=\exp\Big(\sum_{x\in \Lambda}\lambda t_x^2+\mu t_x\Big).
\end{equation}
It follows that
\begin{equation}
\label{eq:ZFt}
	Z_{\Lambda,\beta}
    =f(\mathbf{t})
    \int_{(\mathbb R^n)^\Lambda}F_{\mathbf{t}}(\varphi)
    \exp (-\beta H_{\Lambda}(\varphi) ) \mathrm{d}s_{\Lambda,\mathbf{t}}
    (\varphi),
\end{equation}
and therefore
\begin{equation}
    z_{\Lambda,\beta}(\mathbf{t})
    =\frac{Z_{\Lambda,\beta}(\mathbf{t})}{Z_{\Lambda,\beta}}
    =
    \frac{1}{f(\mathbf{t})}
    \frac{1}{\langle F_{\mathbf{t}}(\varphi)\rangle_{\Lambda,\beta,\mathbf{t}}}.
\end{equation}
By Griffiths' (for $n=1$) or Ginibre's (for $n=2$) inequality,
\begin{equation}
	\langle F_{\mathbf{t}}(\varphi)\rangle_{\Lambda,\beta,\mathbf{t}}
    \geq \langle F_{\mathbf{t}}(\varphi)\rangle_{\Lambda,\beta',\mathbf{t}}.
\end{equation}
This implies the desired inequality.
\end{proof}

\begin{Lem}
\label{lem:BFS time decreases 2pt}
Let $\Lambda\subset \mathbb Z^d$ be finite and $\beta\geq 0$. Then, for every $\mathbf{t}\in (\mathbb R^+)^\Lambda$ and every $x,y\in \Lambda$,
\begin{equation}
	\langle\varphi_x^1\varphi_y^1\rangle_{\Lambda,\beta,\mathbf{t}}\leq \langle\varphi_x^1\varphi_y^1\rangle_{\Lambda,\beta}.
\end{equation}
\end{Lem}

\begin{proof}
As in \eqref{eq:ZFt}, with $F_{\mathbf{t}}$ defined by \eqref{eq:defF BFS},
\begin{equation}
	\langle \varphi_x^1\varphi_y^1\rangle_{\Lambda,\beta}
    =\frac{\langle\varphi_x^1\varphi_y^1 F_{\mathbf{t}}(\varphi)\rangle_{\Lambda,\beta,\mathbf{t}}}
    {\langle F_{\mathbf{t}}(\varphi)\rangle_{\Lambda,\beta,\mathbf{t}}}.
\end{equation}
Again, by Griffiths' (for $n=1$) or Ginibre's (for $n=2$) inequality,
\begin{equation}
	\frac{\langle\varphi_x^1\varphi_y^1
    F_{\mathbf{t}}(\varphi)\rangle_{\Lambda,\beta,\mathbf{t}}}
    {\langle F_{\mathbf{t}}(\varphi)\rangle_{\Lambda,\beta,\mathbf{t}}}
    \geq \langle\varphi_x^1\varphi_y^1\rangle_{\Lambda,\beta,\mathbf{t}}.
\end{equation}
This completes the proof.
\end{proof}

\subsubsection{Verification of Definition~\ref{Def:G}}

Let $d\geq 2$.
Although $G$ does not satisfy Definition~\ref{Def:G}$(i)$, this is  harmless, as explained in Remark~\ref{Rem:A}.
The remaining properties listed in Definition~\ref{Def:G} are classical facts
about the $1$- and $2$-component $|\varphi|^4$ model.
The monotonicity in $(ii)$ is a consequence of Griffiths' or Ginibre's inequality, and the differentiability is proved in \cite{Lebo72} for the Ising model and follows more generally for $1$- and $2$-component models using the extension of the Lebowitz inequality in \cite{BFS82}.
 Property~$(iii)$ is inherited from $J$. Property $(iv)$ follows by defining
\begin{equation}
\beta_c:=\inf\Big\{\beta \geq 0: \chi(\beta)=\sum_{x\in \mathbb Z^d}G_\beta(x)=\infty\Big\},
\end{equation}
and using the Simon--Lieb \cite{Simo80,Lieb80} (for $n=1$) or Rivasseau \cite{Riva80} (for $n=2$) inequality as in \cite[Section~2.5]{DT16}. From \cite{DT16}, we additionally obtain that $\{\beta \geq 0: \chi(\beta)=\infty\}$ is a closed set, which forces
\begin{equation}\label{eq:chi infinite for phi4}
\chi(\beta_c)=\infty.
\end{equation}
Finally, $(v)$ holds as a consequence of the aforementioned correlation inequalities, which imply left-continuity\footnote{Here we use
the fact that $G_\beta(x)$ is the increasing limit of the increasing continuous functions $\beta\mapsto \langle \varphi_0^1\varphi_x^1\rangle_{\Lambda,\beta,J}$ as $\Lambda$ approximates $\mathbb Z^d$.} of the map $\beta\mapsto G_\beta(x)$ on $[0,\beta_c]$ for every $x\in \mathbb Z^d$.

\subsubsection{Verification of Assumption~\ref{Ass:G}}

The verification of Assumption~\ref{Ass:G} proceeds by taking the infinite-volume limit of finite-volume versions of the two inequalities.  We first state two lemmas with the finite-volume inequalities, then verify Assumption~\ref{Ass:G}, and finally prove the two finite-volume lemmas.

For $\Lambda\subset \mathbb Z^d$ and $x,y\in \mathbb Z^d$, we write $G_{\Lambda,\beta}(x,y):=\langle \varphi_x\varphi_y\rangle_{\Lambda,\beta}$.
By definition, $G_{\Lambda,\beta}(x,y)=0$ if $x\notin\Lambda$ or $y\notin\Lambda$. We also define
\begin{equation}
F_{\Lambda,\beta}(x,y)=\Big(\sum_{z\in \mathbb Z^d}G_{\Lambda,\beta}(x,z)J_{y-z}\Big)\vee \Big(\sum_{z\in \mathbb Z^d}G_{\Lambda,\beta}(y,z)J_{x-z}\Big),
\end{equation}
so that $F_{\Lambda,\beta}(x,y)=F_{\Lambda,\beta}(y,x)$. When $\Lambda = \mathbb{Z}^d$,
we have $F_{\Lambda} = F$.

The following lemma is a finite-volume version of \eqref{eq:SL assumption}.

\begin{Lem}\label{lem:finite volume finite diff} Let $\Lambda\subset \mathbb Z^d$ be  finite. For every $0\leq \beta'\leq \beta$ and every $x\in \mathbb Z^d$, we have
\begin{equation}\label{eq:phi4G finite}
	G_{\Lambda,\beta}(0,x)
    \leq G_{\Lambda,\beta'}(0,x)
    +(\beta-\beta')\sum_{u,v\in \mathbb Z^d}G_{\Lambda,\beta}(0,u)J_{v-u}G_{\Lambda,\beta}(v,x).
    \end{equation}
\end{Lem}

For the finite-volume version of \eqref{eq:Diff inequ assumption}, we
define
\begin{equation}
    K_{\Lambda,\beta}(x, y) = G_{\Lambda,\beta}(x, y) + \beta F_{\Lambda,\beta}(x, y),
\end{equation}
\begin{align}
	S^{(1)}_{\Lambda,\beta}(0,x)&=
6\lambda \sum_{u,v\in \mathbb Z^d}
\sum_{z\in \mathbb Z^d} G_{\Lambda,\beta}(0,z)G_{\Lambda,\beta}(z,u)J_{v-u}G_{\Lambda,\beta}(v,z)G_{\Lambda,\beta}(z,x),
\\
	S^{(2)}_{\Lambda,\beta}(0,x)&=
3
\max \left(\frac{1}{A^2},1\right)
\sum_{u,v\in \mathbb Z^d}
\sum_{z\in \mathbb Z^d} K_{\Lambda,\beta}(0,z)K_{\Lambda,\beta}(z,u)J_{v-u}K_{\Lambda,\beta}(v,z)K_{\Lambda,\beta}(z,x),
\end{align}
and
\begin{equation}
    S_{\Lambda,\beta}(0,x)=S^{(1)}_{\Lambda,\beta}(0,x)\wedge S^{(2)}_{\Lambda,\beta}(0,x).
\end{equation}

\begin{Lem}\label{lem:finite volume lower diff inequ} Let $\Lambda\subset \mathbb Z^d$ be  finite. For every $\beta\geq 0$ and every $x\in \mathbb Z^d$, we have
\begin{equation}
\label{eq:phi4lb}
	\partial_\beta G_{\Lambda,\beta}(0,x)\geq \sum_{u,v\in \mathbb Z^d}G_{\Lambda,\beta}(0,u)J_{v-u}G_{\Lambda,\beta}(v,x)-S_{\Lambda,\beta}(0,x).
\end{equation}
\end{Lem}

The choice $S^{(1)}$ for $S$ in Lemma~\ref{lem:finite volume lower diff inequ}
comes with a prefactor $\lambda$, which is perfect for our application to weakly-coupled and spread-out $|\varphi|^4$ models.  However, we also wish to derive the counterpart of
Lemma~\ref{lem:finite volume lower diff inequ} for Ising and XY models in
Sections~\ref{sec:Ising}--\ref{sec:XY}, by taking a limit $\lambda \to \infty$.
For this, we will use the choice $S^{(2)}$ instead.

It is straightforward to deduce Assumption \ref{Ass:G} from
Lemmas~\ref{lem:finite volume finite diff}--\ref{lem:finite volume lower diff inequ}.
The lower differential inequality is in terms of
\begin{equation}
\label{eq:Hphi4}
    H_\beta:=H^{(1)}_\beta\wedge H^{(2)}_\beta,
\end{equation}
 \begin{equation}\label{eq:expression H1 spins}
 	H^{(1)}_\beta(x)=6\lambda \delta_0(x)(F_\beta*G_\beta)(x),
 \end{equation}
  \begin{equation}\label{eq:expression H2 spins}
H^{(2)}_\beta(x)=3\max \left(\frac{1}{A^2},1\right)
\big[(\delta_0+\beta J)*(\delta_0+\beta J)\big] (x)(K_\beta*J*K_\beta)(0),
\end{equation}
with
\begin{equation}
    K_\beta=G_\beta+\beta F_\beta=G_\beta*(\delta_0+\beta J).
\end{equation}

\begin{proof}[Proof of \eqref{eq:SL assumption}]
Let $0\leq \beta'\leq \beta<\beta_c$ and $x\in \mathbb Z^d$. By Griffiths' or Ginibre's inequality, for every $u,v\in \mathbb Z^d$, the sequence $G_{\Lambda_k,\beta'}(0,u)J_{v-u}
    G_{\Lambda_k,\beta}(v,x)$ increases monotonically to the limit $G_\beta(u)J_{v-u}G_\beta(x-v)$.
By applying Lemma~\ref{lem:finite volume finite diff} to the sequence of boxes $(\Lambda_k)_{k\geq 1}$, in conjunction with the monotone convergence theorem, we conclude
that  \eqref{eq:SL assumption} holds.
\end{proof}

\begin{proof}[Proof of \eqref{eq:Diff inequ assumption}]
Let $\beta<\beta_c$ and $x\in \mathbb Z^d$. For the infinite-volume limit of the right-hand side of \eqref{eq:phi4lb},
by monotonicity we have
 \begin{equation}
 	\lim_{k\rightarrow \infty}\sum_{u,v\in \mathbb Z^d}G_{\Lambda_k,\beta}(0,u)J_{v-u}G_{\Lambda_k,\beta}(v,x)=(G_\beta*J*G_\beta)(x),
 \end{equation}
 \begin{align}\nonumber
 	\lim_{k\rightarrow \infty}S_{\Lambda_k,\beta}^{(1)}(0,x)&=6\lambda\sum_{u,v,z\in \mathbb Z^d}G_\beta(z)G_\beta(u-z)J_{v-u}G_{\beta}(z-v)G_\beta(x-z)
 	\\&=(G_\beta*H_\beta^{(1)}*G_\beta)(x),
 \end{align}
and
 \begin{align}\nonumber
 	\lim_{k\rightarrow \infty}S_{\Lambda_k,\beta}^{(2)}(0,x)&=3\max \left(\frac{1}{A^2},1\right)  \sum_{u,v,z\in \mathbb Z^d} K_\beta(z)K_\beta(u-z)J_{v-u}K_\beta(z-v)K_\beta(x-z) 	\\&=(G_\beta*H_\beta^{(2)}*G_\beta)(x).
 \end{align}
The function $H_\beta =H^{(1)}_\beta\wedge H^{(2)}_\beta$ satisfies all the hypotheses stated in Assumption \ref{Ass:G}.
To complete the proof, it suffices to know that the finite-volume derivative
on the left-hand side of \eqref{eq:phi4lb} converges to the infinite-volume
derivative. This latter fact is a consequence of correlation inequalities; see \cite{Lebo72}
for the Ising model and \cite{BFS82} for the extension of the Lebowitz inequality
to more general $1$- and $2$-component spins.
This completes the proof.
\end{proof}

\begin{proof}[Proof of Lemma \textup{\ref{lem:finite volume finite diff}}]
Let $\Lambda\subset \mathbb Z^d$ be finite, $0\leq \beta'\leq \beta$, and $x\in \mathbb Z^d$. By Proposition~\ref{prop:BFS},
\begin{equation}\label{eq:pfdphi4 1}
	G_{\Lambda,\beta}(0,x)=\langle \varphi_0^1\varphi_x^1\rangle_{\Lambda,\beta}
    = \sum_{\gamma: 0 \rightarrow x} (\beta J )^{\gamma}
    \int z_{\Lambda,\beta}(\mathbf{t}) \mathrm{d}\nu_{\Lambda,\gamma}(\mathbf{t}).
\end{equation}
We apply the identity $b^{n}-a^{n} =(b-a) \sum_{k = 0}^{n-1}b^{k}a^{n-1-k}$
with $b=\beta$ and $a=\beta'$.
This gives
\begin{align}
\label{eq:phi4Gdecomp}
    G_{\Lambda,\beta}(0,x)
    & =
    \sum_{\gamma: 0 \rightarrow x} (\beta' J )^{\gamma}
    \int z_{\Lambda,\beta}(\mathbf{t}) \mathrm{d}\nu_{\Lambda,\gamma}(\mathbf{t})
    \\ \nonumber  & \qquad +
    (\beta-\beta')
    \sum_{u,v\in \mathbb Z^d}\sum_{\substack{\gamma_1:0\rightarrow u\\\gamma_2:v\rightarrow x}} (\beta' J)^{\gamma_1}  J_{v-u} (\beta J)^{\gamma_2}\int z_{\Lambda,\beta}(\mathbf{t}) \mathrm{d}\nu_{\Lambda,\gamma_1\circ (uv)\circ \gamma_2}(\mathbf{t}),
\end{align}
where $\gamma_1\circ (uv)\circ\gamma_2$ denotes the concatenation of $\gamma_1$, the step $uv$, and $\gamma_2$. By Lemma~\ref{lem:phi4 monot BFS} and Proposition~\ref{prop:BFS},
 the first term on the right-hand side of \eqref{eq:phi4Gdecomp}
is bounded above by the term $G_{\Lambda,\beta'}(0,x)$
appearing in \eqref{eq:phi4G finite}.

For the second term on the right-hand side of \eqref{eq:phi4Gdecomp}, we use the fact
that\footnote{The first equality can be seen as a kind of binomial theorem for the measures $\mathrm{d}\rho_{\Lambda,\gamma}$ on the occupation time $\mathbf{t}$.}
\begin{align}
	\int z_{\Lambda,\beta}(\mathbf{t}) \mathrm{d}\nu_{\Lambda,\gamma_1\circ (uv)\circ \gamma_2}(\mathbf{t})
    & =
    \int z_{\Lambda,\beta}(\mathbf{t}_1+\mathbf{t}_2) \mathrm{d}\nu_{\Lambda,\gamma_1}(\mathbf{t}_1)\mathrm{d}\nu_{\Lambda,\gamma_2}(\mathbf{t}_2)
    \nonumber \\ & =
    \int \mathrm{d}\nu_{\Lambda,\gamma_1}(\mathbf{t}_1)
    z_{\Lambda,\beta}(\mathbf{t}_1)
    \int \mathrm{d}\nu_{\Lambda,\gamma_2}(\mathbf{t}_2)
    \frac{z_{\Lambda,\beta}(\mathbf{t}_1+\mathbf{t}_2) } {z_{\Lambda,\beta}(\mathbf{t}_1)}
    .\label{eq:prooffinitediff phi41}
\end{align}
By Lemma~\ref{lem:BFS time decreases 2pt},
for every fixed $\mathbf{t}_1$,
\begin{equation}
	\sum_{\gamma_2:v\rightarrow x}(\beta J)^{\gamma_2}
    \int \frac{z_{\Lambda,\beta}(\mathbf{t}_1+\mathbf{t}_2) } {z_{\Lambda,\beta}(\mathbf{t}_1)}\D\nu_{\Lambda,\gamma_2}(\mathbf{t}_2)
    =\langle \varphi_v^1\varphi_x^1\rangle_{\Lambda,\beta,\mathbf{t}_1}
    \leq G_{\Lambda,\beta}(v,x).\label{eq:prooffinitediff phi42}
\end{equation}
We use \eqref{eq:prooffinitediff phi41}--\eqref{eq:prooffinitediff phi42} in
\eqref{eq:phi4Gdecomp} and perform the sum
over $\gamma_1$ to obtain the second term on the right-hand side of \eqref{eq:phi4G finite}.
This completes the proof.
\end{proof}

\begin{proof}[Proof of Lemma \textup{\ref{lem:finite volume lower diff inequ}}]
Let $\Lambda\subset \mathbb Z^d$ be a finite set containing $0$. Let $\beta\geq 0$ and $x\in \mathbb Z^d$. We consider the $1$- and $2$-component models simultaneously, by writing
$\varphi=(\varphi^1,\varphi^2)$ with  $\varphi^2=0$ when $n=1$.
To simplify the notation, throughout the proof we write simply $\langle \cdot\rangle=\langle\cdot\rangle_{\Lambda,\beta}$ and generally omit these subscripts.

We use Ursell's four-point functions:
\begin{equation}
     U^{(4)}(0,x,u,v) = \langle \varphi_0^1 \varphi_x^1 \varphi_u^1 \varphi_v^1 \rangle - \langle \varphi_0^1 \varphi_x^1 \rangle \langle \varphi_u \varphi_v^1\rangle-\langle \varphi_0^1 \varphi_u^1 \rangle \langle \varphi_v^1 \varphi_x^1 \rangle  - \langle \varphi_0^1 \varphi_v^1 \rangle \langle \varphi_u^1 \varphi_x^1 \rangle,
\end{equation}
and
\begin{equation}
	\tilde{U}^{(4)}(0,x,u,v)=\langle \varphi_0^1\varphi_x^1\varphi_u^2\varphi_v^2\rangle-\langle \varphi_0^1\varphi_x^1\rangle\langle \varphi_u^2\varphi_v^2\rangle.
\end{equation}
From \eqref{eq:def phi4}, we express the derivative in $\beta$ as:
\begin{equation}
		\partial_\beta\langle \varphi_0^1\varphi_x^1\rangle  			
=
\frac 12 \sum_{u,v\in \Lambda}
J_{v-u} \Big( U^{(4)}(0,x,u,v)+\tilde{U}^{(4)}(0,x,u,v)+\langle \varphi_0^1 \varphi_u^1 \rangle \langle \varphi_v^1 \varphi_x^1 \rangle  + \langle \varphi_0^1 \varphi_v^1 \rangle \langle \varphi_u^1 \varphi_x^1 \rangle \Big)\label{eq:phi4lb0}.
\end{equation}
By definition,
\begin{equation}\label{eq:phi4lb1}
\frac 12 \sum_{u,v\in \Lambda}
J_{v-u}\Big(\langle \varphi_0^1 \varphi_u^1 \rangle \langle \varphi_v^1 \varphi_x^1 \rangle  + \langle \varphi_0^1 \varphi_v^1 \rangle \langle \varphi_u^1 \varphi_x^1 \rangle\Big)=\sum_{u,v\in \mathbb Z^d}G_{\Lambda,\beta}(0,u)J_{v-u}G_{\Lambda,\beta}(v,x),
\end{equation}
which is the first term on the right-hand side of \eqref{eq:phi4lb}.

The subtracted term on the right-hand side of \eqref{eq:phi4lb} arises from
an upper bound on
\begin{equation}
 -U^{(4)}(0,u,v,x)=|U^{(4)}(0,x,u,v)|\quad  \textup{and } -\tilde{U}^{(4)}(0,x,u,v)=|\tilde{U}^{(4)}(0,x,u,v)|
 \end{equation}
(the equalities follow from Lebowitz's inequality \cite{Lebo74}).
We follow \cite[(37)]{Froh82} and use
\begin{equation}\label{eq:phi4lb2}
	|U^{(4)}(0,u,v,x)|\leq \sum_P\sum_{\gamma_1,\gamma_2}^P (\beta J)^{\gamma_1}(\beta J)^{\gamma_2}\int \mathrm{d}\nu_{\gamma_1}(\mathbf{t}^1)\mathrm{d}\nu_{\gamma_2}(\mathbf{t}^2)z(\mathbf{t}^1)z(\mathbf{t}^2)f_\lambda(\mathbf{t}^1,\mathbf{t}^2),
\end{equation}
where
$\sum_P$ is a sum over partitions in pairs (or \emph{pairings}) of $\{0,u,v,x\}$,  $\sum_{\gamma_1,\gamma_2}^P$ is a sum over paths $\gamma_1,\gamma_2$ which respect the partition $P$ (i.e.,, $\gamma_1$ connects the elements of the first pair of the partition, and $\gamma_2$ the elements of the second pair), and
\begin{equation}\label{eq:def f lambda}
	f_{\lambda}(\mathbf{t}^1,\mathbf{t}^2)=1-\exp\Big(-2\lambda \sum_{z\in \mathbb Z^d}\mathbf{t}^1_z\mathbf{t}^2_z\Big).
\end{equation}
We use the elementary bound
\begin{equation}\label{eq:phi4lb3}
	f_{\lambda}(\mathbf{t}^1,\mathbf{t}^2)\leq \Big(2\lambda \sum_{z\in \mathbb Z^d}\mathbf{t}^1_z\mathbf{t}^2_z\Big)\wedge \sum_{z\in \mathbb Z^d}\mathds{1}_{z\in \gamma_1\cap \gamma_2}.
\end{equation}
The term $|\tilde{U}^{(4)}(0,x,u,v)|$ satisfies a bound like \eqref{eq:phi4lb2}, where instead of a sum over all pairings $P$, we only have the term corresponding to the pairing $P=\{\{0,x\},\{u,v\}\}$.

Consider the first alternative on the right-hand side of \eqref{eq:phi4lb3}.
It follows from \cite[(3.13)]{BFS83II} that, for a fixed $z\in \mathbb Z^d$, and $\gamma:a\rightarrow b$,
\begin{equation}
	\sum_{\gamma:a\rightarrow b}(\beta J)^{\gamma}\int \mathrm{d}\nu_{\gamma}(\mathbf{t})z(\mathbf{t})\mathbf{t}_z\leq \langle \varphi_a^1\varphi_z^1\rangle\langle\varphi_z^1\varphi_b^1\rangle.
	\end{equation}
Since there are three pairings of $\{0,u,v,x\}$, this implies that
\begin{multline}
	2\lambda\sum_{z\in \mathbb Z^d}\sum_P\sum_{\gamma_1,\gamma_2}^P(\beta J)^{\gamma_1}(\beta J)^{\gamma_2}\int \mathrm{d}\nu_{\gamma_1}(\mathbf{t}^1)\mathrm{d}\nu_{\gamma_2}(\mathbf{t}^2)z(\mathbf{t}^1)z(\mathbf{t}^2) \mathbf{t}^1_z\mathbf{t}^2_z
	\\\leq 6\lambda \sum_{z\in \mathbb Z^d}\langle \varphi_0^1\varphi_z^1\rangle\langle \varphi_u^1\varphi_z^1\rangle \langle \varphi_v^1\varphi_z^1\rangle \langle \varphi_x^1\varphi_z^1\rangle
.
\label{eq:verif error phi4 1}
\end{multline}
The combination of \eqref{eq:phi4lb2} and \eqref{eq:verif error phi4 1}
leads to
\begin{multline}
\frac 12 \sum_{u,v\in\Lambda}
J_{v-u} \Big( U^{(4)}(0,x,u,v)+\tilde{U}^{(4)}(0,x,u,v)\Big)
\ge
\\
- 6
\lambda \sum_{u,v\in \mathbb Z^d}
\sum_{z\in \mathbb Z^d} G(0,z)G(z,u)J_{v-u}G(v,z)G(z,x).
\end{multline}
This gives the first half of Lemma~\ref{lem:finite volume lower diff inequ}, i.e., the lower bound with $S^{(1)}(0,x)$.

To obtain the lower bound with $S^{(2)}(0,x)$, we bound $f_\lambda$ by the second alternative on the right-hand side of \eqref{eq:phi4lb3}.
Recall that $K(s,t) = G(s,t)+\beta F(s,t)$.  We will prove that
\begin{multline}\label{eq: lemma verif error phi4 3}
	\sum_{z\in \mathbb Z^d}\sum_P\sum_{\gamma_1,\gamma_2}^P(\beta J)^{\gamma_1}(\beta J)^{\gamma_2}\int \mathrm{d}\nu_{\gamma_1}(\mathbf{t}^1)\mathrm{d}\nu_{\gamma_2}(\mathbf{t}^2)z(\mathbf{t}^1)z(\mathbf{t}^2) \mathds{1}_{z\in \gamma_1\cap \gamma_2}
	\\\leq 3 \max \left(\frac{1}{A^2},1\right)  \sum_{z\in \mathbb Z^d} K(0,z)K(z,u)K(v,z)K(z,x).
	\end{multline}
By \eqref{eq:phi4lb2},
this suffices, since it gives
(recall that a bound on $|U^{(4)}|$ also bounds $|\tilde{U}^{(4)}|$)
\begin{multline}
\frac 12 \sum_{u,v\in\Lambda}
J_{v-u} \Big( U^{(4)}(0,x,u,v)+\tilde{U}^{(4)}(0,x,u,v)\Big)
\ge
\\-
3
\max \left(\frac{1}{A^2},1\right)
\sum_{u,v\in \mathbb Z^d}
\sum_{z\in \mathbb Z^d} K(0,z)K(z,u)J_{v-u}K(v,z)K(z,x),
\end{multline}
which is the desired estimate with $S^{(2)}$.

It remains only to prove \eqref{eq: lemma verif error phi4 3}.
For this, it suffices to show that, for every $a,b$,
    \begin{equation}
        \sum_{\gamma: a \rightarrow b }(\beta J)^{\gamma}\int \mathrm{d}\nu_{\gamma}(\mathbf{t})\mathrm{d}z(\mathbf{t}) \mathds{1}_{z\in \gamma} \leq \max\left(\frac{1}{A},1\right) K(a,z) K(z,b). \label{eq:phi4lb6}
    \end{equation}
    We assume that $a,b,z\in \Lambda$ since otherwise the sum on the left-hand side of \eqref{eq:phi4lb6} is equal to $0$. If $a=b=z$, then
    \begin{align}\nonumber
    	\sum_{\gamma: a \rightarrow b }(\beta J)^{\gamma}\int \mathrm{d}\nu_{\gamma}(\mathbf{t})\mathrm{d}z(\mathbf{t}) \mathds{1}_{z\in \gamma}
    &=G(a,a) =\frac{1}{A}G(a,a)G_{\Lambda,0}(a,a)
    	\\&\leq \frac{1}{A}G(a,a)G(a,a)\nonumber =\frac{1}{A}G(a,z)G(z,b).
    \end{align}
    We may therefore assume that $a,b,z$ are not all equal. In this case,
    if $z \in \gamma$ then at least one of the following holds true: $(i)$ there exists $z'$ with $J_{z'-z}>0$, $\gamma_1:a\rightarrow z$, and $\gamma_2:z'\rightarrow b$ such that $\gamma = \gamma_1 \circ (zz') \circ\gamma_2$ (where $(zz')$ denotes the one-step walk from $z$ to $z'$); $(ii)$ there exists $z'$ with $J_{z-z'}>0$, $\gamma_1:a\rightarrow z'$, and $\gamma_2:z\rightarrow b$ such that $\gamma = \gamma_1 \circ (z'z)\circ \gamma_2$. The contributions of these two scenarios to the left-hand side of \eqref{eq:phi4lb6} are respectively denoted $({\rm I})$ and $({\rm II})$.
The term $({\rm I})$ is given by
    \begin{equation}
        ({\rm I}) = \sum_{z'\in \mathbb Z^d}\sum_{\substack{\gamma_1:a\rightarrow z\\\gamma_2:z'\rightarrow x}} (\beta J)^{\gamma_1} \beta J_{z'-z} (\beta J)^{\gamma_2}\int z(\mathbf{t}) \mathrm{d}\nu_{\gamma_1\circ (z,z')\circ \gamma_2}(\mathbf{t}).
    \end{equation}
Proceeding as in the proof of Lemma~\ref{lem:finite volume finite diff}, and using  $J_u=J_{-u}$, we obtain
    \begin{equation}
    \label{eq:IGF}
        ({\rm I}) \leq \sum_{z'\in \mathbb Z^d}G(a,z)\beta J_{z'-z} G(z',b)=\beta G(a,z)\sum_{z'\in \mathbb Z^d}G(b,z')J_{z-z'} \leq \beta G(a,z)F(z,b).
    \end{equation}
    Similarly,
    \begin{equation}
    \label{eq:IIFG}
    	({\rm II})\leq \beta F(a,z)G(z,b).
    \end{equation}
    The bounds \eqref{eq:IGF}--\eqref{eq:IIFG} give \eqref{eq:phi4lb6},
    which we have seen to be sufficient.
    This completes the proof.
\end{proof}

\subsection{The Ising model}
\label{sec:Ising}
To verify Assumption \ref{Ass:G} for the Ising model, we combine Lemmas~\ref{lem:finite volume finite diff}--\ref{lem:finite volume lower diff inequ} with Proposition \ref{prop:cv of phi4 to Ising} below to transfer the desired properties from $|\varphi|^4$ to Ising. Direct proofs can also be obtained via the random current representation \cite{Aize82,Dumi20} of the Ising model.
\subsubsection{The model}

Let $d\geq 2$. Let $\Lambda\subset \mathbb Z^d$ be finite, $\beta\geq 0$, $J$ an admissible interaction
, and $F:\{-1,1\}^\Lambda\rightarrow \mathbb R$. The Ising model on $\Lambda$ is the measure $\langle \cdot \rangle_{\Lambda,\beta}^{\rm Ising}$ on $\{-1,1\}^\Lambda$ given by
	\begin{equation}
		\langle F(\sigma)\rangle_{\Lambda,\beta}^{\rm Ising}:=\frac{1}{Z^{\rm Ising}_{\Lambda,\beta}}\sum_{\sigma \in \{-1,1\}^\Lambda}F(\sigma)\exp(-\beta H_{\Lambda}(\sigma)),
	\end{equation}
	where
	\begin{equation}
		H_{\Lambda}(\sigma)=
-
\frac 12 \sum_{x,y\in \Lambda}
J_{y-x}\sigma_x\sigma_y, \qquad Z_{\Lambda,\beta}^{\rm Ising}=\sum_{\sigma\in \{-1,1\}^\Lambda}\exp(-\beta H_{\Lambda}(\sigma)).
	\end{equation}
It is again a consequence of Griffiths' inequalities that the sequence of measures $\langle \cdot \rangle_{\Lambda,\beta}^{\rm Ising}$ admits a weak limit as $\Lambda \nearrow \mathbb Z^d$. We denote the limit by $\langle \cdot\rangle_\beta^{\rm Ising}$. The two-point function is defined for $\beta\geq 0$ and $x\in \mathbb Z^d$ by
\begin{equation}
	G_\beta(x):=\langle \sigma_0\sigma_x\rangle_{\beta}^{\rm Ising}.
\end{equation}

The $1$-component $\varphi^4$ model and the Ising model are expected to lie in the same universality class.  Support for this conjecture is provided by the fact that
the latter is a limit of the former.  Indeed, the \emph{normalised} single-site distribution
with $\mu=-\lambda$
obeys (in the sense of weak convergence)
\begin{equation}\label{eq:cv of phi4 to Ising} \frac{1}{z_\lambda}\exp\Big(-\frac{\lambda}{4}(\varphi_0^2-1)^2\Big){\rm d}\varphi_0\underset{\lambda \rightarrow \infty}{\longrightarrow} \frac{\delta_{-1}+\delta_1}{2}.
\end{equation}
Easy consequences of \eqref{eq:cv of phi4 to Ising} are formulated in the
next proposition. Recall the definition of $\langle \cdot\rangle_{\Lambda,\beta,J}=\langle \cdot\rangle_{\Lambda,\beta,J,\lambda,\mu}$ from \eqref{eq:def phi4},
and of $A(\lambda,\mu)$ from \eqref{eq:Aphi4}.

\begin{Prop}\label{prop:cv of phi4 to Ising} Let $\Lambda\subset \mathbb Z^d$ be finite and $\beta\geq 0$. Then, for every $x,y\in \Lambda$,
\begin{align}
    \lim_{\lambda \to \infty}
	\langle \varphi_x\varphi_y\rangle_{\Lambda,\beta,J,\lambda,-\lambda}
&=\langle \sigma_x\sigma_y\rangle_{\Lambda,\beta,J}^{\rm Ising},
	\\
\lim_{\lambda \to \infty}
	\partial_\beta\langle \varphi_x\varphi_y\rangle_{\Lambda,\beta,J,\lambda,-\lambda}&
=\partial_\beta\langle \sigma_x\sigma_y\rangle_{\Lambda,\beta,J}^{\rm Ising},
	\\
\lim_{\lambda \to \infty} A(\lambda,-\lambda)&=1.
\end{align}

\end{Prop}

\subsubsection{Verification of Definition~\ref{Def:G}}

For $d\geq 2$, the properties of the Ising model listed in Definition \ref{Def:G} are classical facts that
can be derived using the same arguments as for the $1$-component $\varphi^4$ model.
We omit further details.
We stress that we
define $\beta_c$ as
\begin{equation}
\beta_c:=\inf\Big\{\beta \geq 0: \chi(\beta)=\sum_{x\in \mathbb Z^d}G_\beta(x)=\infty\Big\}.
\end{equation}

\subsubsection{Verification of Assumption~\ref{Ass:G}}

\begin{proof}[Proof of \eqref{eq:SL assumption}] Let $\Lambda\subset \mathbb Z^d$ be finite, $0\leq \beta'\leq \beta<\beta_c$, and $x\in \mathbb Z^d$.
The combination of Lemma~\ref{lem:finite volume finite diff} and Proposition~\ref{prop:cv of phi4 to Ising} leads to
\begin{equation}
	\langle \sigma_0\sigma_x\rangle_{\Lambda,\beta}^{\rm Ising}\leq \langle \sigma_0\sigma_x\rangle_{\Lambda,\beta'}^{\rm Ising}+(\beta-\beta')\sum_{u,v\in \mathbb Z^d}\langle \sigma_0\sigma_u\rangle_{\Lambda,\beta}^{\rm Ising}J_{v-u}\langle \sigma_v\sigma_x\rangle_{\Lambda,\beta}^{\rm Ising}.
\end{equation}
We can then pass to the limit $\Lambda\nearrow \mathbb Z^d$ by using Griffiths' inequalities in the same manner as for the $1$-component $\varphi^4$ model. This concludes the proof.
\end{proof}

We will prove the differential lower bound with
\begin{equation}
\label{eq:HIsing}
	H_\beta(x)=3\big[ (\delta_0+\beta J)*(\delta_0+\beta J)\big] (x)(K_\beta*J*K_\beta)(0),
\end{equation}
where $K_\beta=G_\beta+\beta F_\beta$.
By definition, $H_\beta$ satisfies all the properties listed below \eqref{eq:Diff inequ assumption}.

\begin{proof}[Proof of \eqref{eq:Diff inequ assumption}] Let $\Lambda\subset \mathbb Z^d$ be finite, $0\leq \beta<\beta_c$, and $x\in \mathbb Z^d$.
The combination of Lemma~\ref{lem:finite volume lower diff inequ}
(with $S^{(2)}$) and Proposition~\ref{prop:cv of phi4 to Ising}
(with the contribution from $\tilde U^{(4)}$ omitted) leads to
\begin{multline}
	\partial_\beta\langle \sigma_0\sigma_x\rangle_{\Lambda,\beta}^{\rm Ising}\geq \sum_{u,v\in \mathbb Z^d}\langle \sigma_0\sigma_u\rangle_{\Lambda,\beta}^{\rm Ising}J_{v-u}\langle \sigma_v\sigma_x\rangle_{\Lambda,\beta}^{\rm Ising}
	\\
- 3
\sum_{u,v\in \mathbb Z^d}
\sum_{z\in \mathbb Z^d} K_{\Lambda,\beta}^{\rm Ising}(0,z)K_{\Lambda,\beta}^{\rm Ising}(z,u)J_{v-u}K_{\Lambda,\beta}^{\rm Ising}(v,z)K_{\Lambda,\beta}^{\rm Ising}(z,x),
\end{multline}
where $K^{\rm Ising}_{\Lambda,\beta}$ is defined similarly as for the  $\varphi^4$ model.
We can again take the infinite-volume limit by appealing to \cite{Lebo72}.
This gives
\begin{equation}
	\partial_\beta G_\beta(x)\geq (G_\beta*(J-H^{}_\beta)*G_\beta)(x),
\end{equation}
and completes the proof.
\end{proof}

\subsection{The XY model}
\label{sec:XY}

Assumption \ref{Ass:G} can be verified for the XY model by combining Lemmas~\ref{lem:finite volume finite diff}--\ref{lem:finite volume lower diff inequ} with Proposition \ref{prop:cv of phi4 to XY} below to transfer inequalities from the $2$-component $|\varphi|^4$ model to the XY model.
Direct proofs can be obtained using the BFS expansion of the XY model (see \cite{Lamm23}).

\subsubsection{The model}

	Let $d\geq 2$. Let $\Lambda\subset \mathbb Z^d$ be finite, $\beta\geq 0$, and $J$ an admissible interaction.
Let $\mathbb S^1$ denote the unit circle in $\mathbb R^2$,
and let $\mathrm d \sigma_x$ denote the Haar measure on $\mathbb S^1$.
Let $F: (\mathbb S^1)^\Lambda\rightarrow \mathbb R$.
    The XY model on $\Lambda$ is the measure $\langle \cdot \rangle_{\Lambda,\beta}^{\rm XY}$ on $(\mathbb S^1)^\Lambda$ given by
	\begin{equation}
		\langle F(\sigma)\rangle_{\Lambda,\beta}^{\rm XY}:=\frac{1}{Z^{\rm XY}_{\Lambda,\beta}}\int_{(\mathbb S^1)^\Lambda} F(\sigma)\exp(-\beta H_{\Lambda}(\sigma)) \prod_{x\in \Lambda}\mathrm d \sigma_x,
	\end{equation}
	where
	\begin{equation}
		H_{\Lambda}(\sigma)=-\sum_{\{x,y\}\subset \Lambda}J_{y-x}(\sigma_x \cdot \sigma_y), \qquad Z_{\Lambda,\beta}^{\rm XY}=\int_{(\mathbb S^1)^\Lambda}\exp(-\beta H_{\Lambda}(\sigma)) \prod_{x\in \Lambda}\mathrm d \sigma_x.
	\end{equation}
By Ginibre's inequalities, the sequence of measures $\langle \cdot \rangle_{\Lambda,\beta}^{\rm XY}$ admits a weak limit as $\Lambda \nearrow \mathbb Z^d$. We denote it by $\langle \cdot\rangle_\beta^{\rm XY}$. The two-point function is defined for $\beta\geq 0$ and $x\in \mathbb Z^d$ by
\begin{equation}
	G_\beta(x):=\langle \sigma^1_0\sigma^1_x\rangle_{\beta}^{\rm XY}.
\end{equation}
The $2$-component $|\varphi|^4$ model and the XY model are expected to lie in the same universality class.
As in \eqref{eq:cv of phi4 to Ising}, there is convergence of the
$|\varphi|^4$ single-spin distribution:
\begin{equation}\label{eq:cv of phi4 to XY}
\frac{1}{z_\lambda}\exp\Big(-\frac{\lambda}{4}(|\varphi_0|_2^2-1)^2\Big){\rm d}\varphi_0\underset{\lambda \rightarrow \infty}{\longrightarrow} \mathrm d \sigma_0.
\end{equation}
Easy consequences of \eqref{eq:cv of phi4 to XY} are formulated in the
next proposition. Recall the definition of $\langle \cdot\rangle_{\Lambda,\beta,J}=\langle \cdot\rangle_{\Lambda,\beta,J,\lambda,\mu}$ from \eqref{eq:def phi4},
and of $A(\lambda,\mu)$ from \eqref{eq:Aphi4}.

\begin{Prop}\label{prop:cv of phi4 to XY}
Let $\Lambda\subset \mathbb Z^d$ be finite and $\beta\geq 0$. Then, for every $x,y\in \Lambda$,
\begin{align}
	\lim_{\lambda\to\infty}
\langle \varphi^1_x\varphi^1_y\rangle_{\Lambda,\beta,\lambda,-\lambda}&=\langle \sigma^1_x\sigma^1_y\rangle_{\Lambda,\beta}^{\rm XY},
	\\
\lim_{\lambda\to\infty}
	\partial_\beta\langle \varphi^1_x\varphi^1_y\rangle_{\Lambda,\beta,\lambda,-\lambda}&= \partial_\beta\langle \sigma^1_x\sigma^1_y\rangle_{\Lambda,\beta}^{\rm XY},
	\\
\lim_{\lambda\to\infty}
A(\lambda,-\lambda)&= \frac12.\label{eq: cv A XY}
\end{align}

\end{Prop}

\subsubsection{Verification of Definition~\ref{Def:G}}

For $d\geq 2$, the properties of the XY model listed in Definition \ref{Def:G} are classical facts that can be
derived using the  same arguments as for the $2$-component $|\varphi|^4$ model.
We omit the details. Again we define $\beta_c$ as
\begin{equation}
\beta_c:=\inf\Big\{\beta \geq 0: \chi(\beta)=\sum_{x\in \mathbb Z^d}G_\beta(x)=\infty\Big\}.
\end{equation}

\subsubsection{Verification of Assumption~\ref{Ass:G}}

The infinite-volume limit of the derivative is again justified
using correlation inequalities \cite{Lebo72,BFS82}.
We omit the details, which proceed as they do for the Ising model.  The limiting
value of $A$ in \eqref{eq: cv A XY} gives rise to an extra factor $4$, with the result
that
\begin{equation}
\label{eq:HXY}
	H_\beta(x)=12 \big[ (\delta_0+\beta J)*(\delta_0+\beta J) \big] (x) (K_\beta*J*K_\beta)(0),
\end{equation}
with $K_\beta=G_\beta+\beta F_\beta$.

\subsection{Lattice trees}
\label{sec:LT}

\subsubsection{The model}
\label{sec:LTmodel}

Let $d > 2$, $J$ be an admissible kernel, and let $E = \{\{x,y\}: J_{y-x}>0 \}$.
A \emph{lattice tree} is a finite acyclic subgraph of
the infinite graph $(\mathbb Z^d, E)$.
We write $|T|$ for the number of bonds (also called edges) in $T$.
For $p \ge 0$ and for a lattice tree $T$, let
\begin{equation}
    (p J)^T = \prod_{\{u,v\} \in T}p J_{v-u}.
\end{equation}
Let $\mathcal{T}_{0,x}$ denote the set of lattice trees containing
the vertices $0$ and $x$, and let $\mathcal T_0 = \mathcal T_{0,0}$.
The \emph{two-point function}, \emph{one-point function}, and \emph{susceptibility}
are defined by
\begin{align}
    \rho_p(x) = \sum_{T \in \mathcal{T}_{0,x}} (pJ) ^{T},
    \qquad
    g_p = \sum_{T \in \mathcal{T}_{0}} (pJ) ^{T} = \rho_p(0),
    \qquad
    \hat\chi(p) = \sum_{x\in \mathbb Z^d} \rho_p(x).
\end{align}
We also define
\begin{align}
    \hat\xi(p)^2 &= \frac{1}{\hat\chi(p)}\sum_{x \in \mathbb{Z}^d} |x|_2^2
    (J*\rho_p)(x).
\end{align}
 A standard subadditivity argument  \cite{BFG86,HS90b}
shows that the susceptibility is finite below a critical point $p_c\in (0,\infty)$,
and that it diverges at $p_c$ at least as fast as
\begin{equation}
\label{eq:chihat-diverges}
    \hat\chi(p) \gtrsim (p_c-p)^{-1/2}.
\end{equation}
Since a lattice tree $T$ has $|T|+1$ vertices, $g_p$ and $\hat \chi(p)$ are related by
\begin{align}
\label{eq:dg}
    \frac{\mathrm{d}(pg_p)}{\mathrm{d}p} & = \sum_{T \in \mathcal{T}_{0}}(|T|+1) (pJ)^T
    = \sum_{T \in \mathcal{T}_{0}}\sum_{x\in T} (pJ)^T =\hat\chi(p).
\end{align}

In dimensions $d>8$, lattice trees should have critical exponents
$\gamma = \frac 12$ and $\nu = \frac 14$ \cite{HS90b,HS92c},
so this does not fit immediately into our black box which
is designed for situations where $\gamma =1$ and $\nu = \frac 12$.
To fix this mismatch, we make the following change of variables.
For $p \in [0,p_c)$, let
\begin{equation}
    \beta =\beta(p) = pg_p.
\end{equation}
We denote the inverse
of the strictly increasing function $p\mapsto \beta(p)$
by $p=p(\beta)$.

The critical point is defined by $ \beta_c= \sup_{p \le p_c}\beta(p)$.
 A priori, it is possible that $\beta_c=\infty$
because we do not initially know that the critical one-point function $g_{p_c}$ is finite.
The possibility that $\beta_c=\infty$ is allowed in Assumption~\ref{Ass:G}.
A posteriori, we do learn that $\beta_c$ is finite (when $d>8$ and $\sigma_J$ is large).

Let $\deg_T(x)$ denote the number of bonds in the lattice tree $T$ that
are incident to the vertex $x\in T$.
We consider the following reduced set of trees:
\begin{equation}
    \mathcal T^1_{0,0} = \{0\}, \qquad
    \mathcal{T}^1_{0,x} = \{T \in \mathcal{T}_{0,x} | \deg_{T}(x) = 1 \}
    \;\; \text{if $x \neq 0$},
\end{equation}
and define a modified \emph{two-point function} by
\begin{equation}
    G_{\beta}(0) = 1, \qquad
    G_\beta(x) =
    \sum_{T \in \mathcal{T}^1_{0,x}} (p(\beta)J) ^{|T|} \;\;\text{if $x \neq 0$}.
\end{equation}

We will show that $G_\beta$ satisfies Assumption~\ref{Ass:G},  and
satisfies Assumption~\ref{Ass:E}
when $d>8$ and $\sigma_J$ is
sufficiently large.  This implies that the conclusion of
Theorem~\ref{thm:maintheorem-betac} applies
to $G_\beta$ with
\begin{equation}
    \chi(\beta) = \sum_{x\in \mathbb Z^d} G_\beta(x),
    \qquad
    \xi(\beta)^2 = \frac{1}{\chi(\beta)}\sum_{x\in \mathbb Z^d}|x|_2^2 G_\beta(x).
\end{equation}
In particular, $\beta_c<\infty$.
We also show that $\chi(\beta_c)=\infty$, so
the conclusions of Theorem~\ref{Thm:gamma-nu} also apply to $G_\beta$.

Those conclusions do not immediately refer to the original parametrisation
of the model in terms of $p$, rather than $\beta$.
Therefore, before verifying
Assumptions~\ref{Ass:G} and \ref{Ass:E}, we show how to draw conclusions for
$\rho(p)$, $g_p$, $\hat\chi(p)$ and $\hat\xi(p)$.

\begin{Thm}
\label{Thm:LTp}
Let $d>8$. For every admissible $J$ with $R_J$ sufficiently large, for every $p <p_c$, and with $E=O(R_J^{-d})$,
\begin{gather}
\label{eq:LTrhobd}
        \rho_p(x)  \lesssim
     \delta_0(x)
    + \frac{1}{\sigma_J^d}
    \left(\frac{\sigma_J}{\sigma_J \vee|x|} \right)^{d-2-\varepsilon}
    \exp\left(-\bfc\frac{|x|}{\xi(\beta)}\right),
    \\
\label{eq:LTchibd}
    \hat\chi(p) \asymp \frac{1}{(p_c-p)^{1/2}},
    \\
\label{eq:LTxibd}
     \frac{1}{(p_c-p)^{\frac 14 -E}}
     \lesssim \hat\xi(p)
     \lesssim \frac{1}{(p_c-p)^{\frac 14 +E}},
     \\
\label{eq:LTgbd}
     g_{p_c}-g_p \asymp  (p_c-p)^{1/2} .
\end{gather}
\end{Thm}

\begin{Rem}
It is a consequence of \eqref{eq:betacbd} that $1 \le \beta_c \le 1+O(E)$.
In fact, this occurs due to a cancellation in $p_cg_{p_c}$,
because $p_c\sim e^{-1}$ and $g_{p_c}\sim e$ under the hypotheses
of Theorem~\ref{Thm:LTp}, and also for the nearest-neighbour model for $d$ sufficiently large \cite{MS13,KS24,Penr94}.
\end{Rem}

\begin{proof}[Proof of Theorem~\textup{\ref{Thm:LTp}}]
We assume that $G_\beta$ has been proven to satisfy Assumptions~\ref{Ass:G} and \ref{Ass:E}
for kernels $J$ with range $R_J$ sufficiently large, so
that the conclusion of Theorem~\ref{thm:maintheorem-betac} applies to $G_\beta$.

A tree $T \ni x$ can be decomposed into at most $(2R_J+1)^d$ trees in $\mathcal{T}^1_{x,x}$,
since the degree of $x$ in $(\mathbb Z^d,E)$ is $(2R_J+1)^d-1$.  Therefore,
together with a trivial lower bound,
we have
\begin{equation}\label{eq:comp rho G}
    G_{p g_p}(x) \le \rho_p(x) \leq G_{p g_p}(x)
    \Big( 1+ \max_{z:J_z>0}G_{pg_p}(z) \Big)^{(2R_J+1)^d}
    .
\end{equation}

By Theorem~\ref{thm:maintheorem-betac},
$G_{p g_p}(z) \leq  \bfC \sigma_J^{-d}$.
Since $R_J \asymp \sigma_J$ for every admissible kernel (here we use Definition~\ref{Def:J}$(iii)$), it follows that
\begin{equation}
    G_{p g_p}(x) \le \rho_p(x) \lesssim  G_{p g_p}(x) \le
     \delta_0(x)
    + \frac{\bfC}{\sigma_J^d}
    \left(\frac{\sigma_J}{\sigma_J \vee|x|} \right)^{d-2-\varepsilon}
    \exp\left(-\bfc\frac{|x|}{\xi(\beta)}\right).
\end{equation}
This proves \eqref{eq:LTrhobd}.

Equation \eqref{eq:comp rho G} also shows that
\begin{equation}
    \chi(pg_p) \le \hat\chi(p) \lesssim \chi(pg_p),
    \qquad
    \hat\xi(p) \asymp \xi(pg_p).
\end{equation}
In particular, the relation between $\hat\chi$ and $\chi$ implies that
$\beta_c=p_cg_{p_c}$, and then \eqref{eq:chihat-diverges} implies that $\chi(\beta_c)=\infty$.   This tells us that the conclusions of Theorem~\ref{Thm:gamma-nu} apply to $\chi(\beta)$, $\xi(\beta)$, and $\beta_c$.
In particular, with $E=O(R_J^{-d})$ small, we
see that
 \begin{align}
 \label{eq:chihatLT}
    \frac{1}{p_cg_{p_c}-pg_p}
    &\le \hat{\chi}(p)
    \lesssim
    \frac{1}{p_cg_{p_c}-pg_p },
\\
    \left( \frac{1}{p_cg_{p_c}-pg_p} \right)^{1-E}
    &\lesssim
    \frac{\hat{\xi}(p)^2}{\sigma_J^2}
    \lesssim
    \left( \frac{1}{p_cg_{p_c}-pg_p} \right)^{1+E}
    .
\label{eq:xihatLT}
 \end{align}
To obtain the bounds \eqref{eq:LTchibd}--\eqref{eq:LTgbd} from the above,
we need to compare $p_c g_{p_c}-p g_p$ and $p_c-p$.
For this, we apply \eqref{eq:dg} to obtain
\begin{equation}
    \frac{ \mathrm{d} p(\beta)}{\mathrm{d}\beta} = \frac{1}{\hat{\chi}(p(\beta) )}.
\end{equation}
With \eqref{eq:chihatLT}, this gives
\begin{equation}
    \beta_c-\beta\lesssim
     \frac{ \mathrm{d} p(\beta)}{\mathrm{d}\beta}
     \leq \beta_c-\beta,
\end{equation}
and then integration over the interval $[\beta,\beta_c]$ gives
\begin{equation}
    (\beta_c-\beta)^2\lesssim p(\beta_c)-p(\beta) \le \frac{1}{2} (\beta_c-\beta)^2.
\end{equation}
Since $\beta   = pg_p$, this implies that
\begin{equation}
\label{eq:LTpg}
    \sqrt{2}(p_c-p)^{1/2}\leq p_c g_{p_c}-p g_p \lesssim (p_c-p)^{1/2}.
\end{equation}
With \eqref{eq:chihatLT}--\eqref{eq:xihatLT}, this proves \eqref{eq:LTchibd}--\eqref{eq:LTxibd}, and hence that the critical
exponents
for the original problem are $\gamma= \frac 12$ and $\nu = \frac 14 \pm O(E)$.
It also follows from \eqref{eq:LTpg} that
\begin{equation}
    \frac{\sqrt{2}}{p_c}(p_c-p)^{1/2}\Big[ 1 - \frac{g_p}{\sqrt{2}}(p_c-p)^{1/2} \Big]
    \le
    g_{p_c}-g_p
    \lesssim \frac{1}{p_c} (p_c-p)^{1/2},
\end{equation}
which proves \eqref{eq:LTgbd}.
This completes the proof.
\end{proof}

\subsubsection{Verification of Definition~\ref{Def:G}}

It follows from its definition that $G_\beta$ satisfies the following properties: $G_0=\delta_0$, $G_\beta$ is monotone and differentiable
 (by the inverse function theorem) for $\beta\in [0,\beta_c)$, and $G_\beta$ is $\mathbb Z^d$-symmetric.

To see that the function $x \mapsto G_\beta(x)$ decays exponentially
for each fixed $\beta\in [0,\beta_c)$,
we observe that
a tree containing $0,x$ must contain at least $|x|/R_J$ bonds.
Given $p<p_c$, let $q=\frac12 (p+p_c)$.  Then $q<p_c$ and
\begin{equation}
    G_\beta(x) \le \rho_{p(\beta)}(x) \le
    \sum_{\substack{T \in \mathcal{T}_{0,x} \\ |T| \ge |x|/R_J}} (pJ)^T
    \le
    \left(\frac{p}{q}\right)^{|x|/R_J} \sum_{T \in \mathcal{T}_{0,x}} (qJ)^T
    \le
    \left(\frac{p}{q}\right)^{|x|/R_J} \hat\chi(q).
\end{equation}
Since $\hat{\chi}(q)<\infty$, this proves the desired exponential decay.

Finally, the limit $\lim_{\beta \uparrow \beta_c}G_{\beta}(x) = G_{\beta_c}(x)$
exists by monotone convergence.

\subsubsection{Verification of Assumption~\ref{Ass:G}}

Since $G_\beta(0)=G_{\beta'}(0)=1$,
there is nothing to prove for $x=0$.  We therefore assume that $x\neq 0$.
We begin with a decomposition of  a tree $T \in \mathcal{T}^1_{0,x}$
into subtrees.
The \emph{backbone} $\Gamma_{0,x}$ of $T$ is
the unique path in $T$ between $0$ and $x$.
In particular, $\Gamma_{0,0}=\{0\}$.
We order the backbone vertices by setting $y<z$
for $y,z \in \Gamma_{0,x}$ if $y$ is closer to $0$ than $z$ (for the graph metric on $T$).
Similarly, we order the backbone bonds:
the least bond is adjacent to $0$ and the
last bond is adjacent to $x$.

Deletion of the backbone bonds from $T$ leaves
behind non-intersecting connected subgraphs of $T$
which we refer to as \emph{ribs} $R_0, \ldots, R_{x}$.
Given a directed bond $e = (u,v) \in \Gamma_{0,x}$ with $u \leq v$,
we define $R_e$ to be the union of the rib $R_u$ containing $u$
and the bond $\{u,v\}$.
With these definitions, $T$ is the bond-disjoint union of the $R_e$, as $e$ ranges over the
backbone bonds of $T$. The requirement that $T\in \mathcal{T}^1_{0,x}$ ensures that
when $e_n$ is its last edge, the tree $R_{e_n}$ includes  the vertex $x$ which is
the entire rib $R_x$, so nothing is omitted from $T$ in the union of all $R_e$.
Thus, we have
\begin{equation}
\label{eq:Tdecomp}
    T = \bigsqcup_{e \in \Gamma_{0,x}} R_e ,
    \qquad |T| = \sum_{e \in \Gamma_{0,x}}|R_{e} |
\end{equation}
For an illustration, see Figure~\ref{fig:backbone}.

The verification of Assumption~\ref{Ass:G} is
a small variant of the skeleton inequality arguments
of \cite{BFG86,HS90b,TH87}, as follows.

\begin{figure}[t]
\begin{center}
\includegraphics[scale = 1.1]{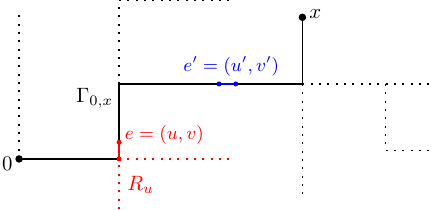}
\caption{An example of a tree $T \in \mathcal{T}^1_{0,x}$ with
its backbone $\Gamma_{0,x}$ in black, and its ribs $R_y$ $(y\in \Gamma_{0,x})$ in dotted lines.
The tree $R_e$ is coloured red.  The tree $R_{e'}$ is coloured blue;
in this case $R_{u'}$ is the single vertex $\{u'\}$.}
\label{fig:backbone}
\end{center}
\end{figure}

\begin{proof}[Proof of \eqref{eq:SL assumption}.]
As mentioned above, we assume that $x\neq 0$.
For $0\leq \beta'<\beta<\beta_c$, let $p = p(\beta)$ and $p'= p(\beta')$.
The identity
\begin{equation}
    \prod_{i=1}^n a_i -  \prod_{i=1}^n b_i
    =
    \sum_{i=1}^n \Big( \prod_{\ell<i}b_\ell\Big)\Big( \prod_{j>i}a_j\Big)(a_i-b_i)
\end{equation}
holds for any real numbers $a_i,b_i$ (an empty product equals $1$).  In particular, by \eqref{eq:Tdecomp},
\begin{equation}
    (pJ)^{T}-(p'J)^{T} = \sum_{e \in \Gamma_{0,x}} (p'J)^{\sqcup_{f <e} R_f}
    (pJ)^{\sqcup_{g>e} R_g} \left((pJ)^{R_e}-(p'J)^{R_e}\right).
\end{equation}
Given a (directed) backbone bond $e=(u,v)$,
we define trees $T_e^- \in \mathcal{T}_{0,u}^1$ and $T_e^+ \in \mathcal{T}_{v,x}^1$ by
\begin{equation}
    T_e^- = \bigsqcup_{f <e} R_f,  \qquad  T_e^+ = \bigsqcup_{g >e} R_g.
\end{equation}
An example of this decomposition is illustrated in Figure~\ref{fig:decomptrees}. With this notation, we have
\begin{equation}
    (pJ)^{T}-(p'J)^{T}
    = \sum_{e \in \Gamma_{0,x}} (p'J)^{T_e^{-}}
    \left((pJ)^{R_e}-(p'J)^{R_e}\right)
    (pJ)^{T_e^{+}} .
\end{equation}
After summation over
all $T\in \mathcal{T}^1_{0,x}$, this gives
\begin{equation}
\label{eq:LT1}
    G_{\beta}(x)-G_{\beta'}(x)
    =
    \sum_{T\in \mathcal{T}^1_{0,x}}
    \sum_{e \in \Gamma_{0,x}} (p'J)^{T_e^{-}} \left((pJ)^{R_e}-(p'J)^{R_e}\right) (pJ)^{T_e^{+}}
    .
\end{equation}

\begin{figure}[h]
\begin{center}
\includegraphics[scale = 1.1]{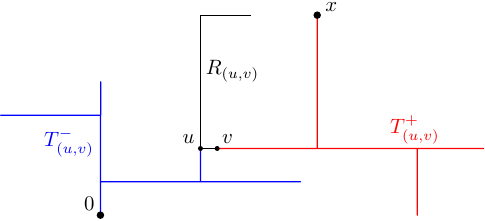}
\caption{An example of a tree $T \in \mathcal{T}^1_{0,x}$, backbone bond $e=(u,v)$, and the decomposition into $T^-_e$, $R_e$, $T^+_e$.
If $u=0$ then $T^-_{(u,v)}$ is empty, and if $v=x$ then
$T^+_{(u,v)}$ is empty.}
\label{fig:decomptrees}
\end{center}
\end{figure}

Fix $e=(u,v)$ with $u<v$.
In the summation on the right-hand side of \eqref{eq:LT1} there are avoidance
constraints between $T_e^-$, the rib $R_u$, the directed edge $(u,v)$, and $T_{(u,v)}^+$. More precisely, it is required that they are bond disjoint and that their union forms a lattice tree.
We write $\1_{C}(T^-,R,(u,v),T^+)$ for the indicator that these constraints hold.  We can then reorganise the summation in \eqref{eq:LT1} as
\begin{align}
    &G_\beta(x)-G_{\beta'}(x)
    \\ \nonumber & \qquad = \sum_{u,v \in \mathbb{Z}^d}
    \sum_{\substack{T^- \in \mathcal{T}^1_{0,u} \\ R \in \mathcal{T}_{u,u}\\ T^+ \in \mathcal{T}^1_{v,x}} }
    (p'J)^{T^{-}} J_{v-u}\left(p(pJ)^{R}-p'(p'J)^{R}\right) (pJ)^{T^{+}} \1_{C}(T^-,R,(u,v),T^+)
    .
\label{eq: maindecompLT}
\end{align}
To prove  \eqref{eq:SL assumption}, we simply bound the indicator function by $1$ and obtain
\begin{align}
    G_\beta(x)-G_{\beta'}(x) &\le
    (G_{\beta'}*(pg_p-p'g_{p'})J*G_\beta)(x)
    \nonumber \\ & =
    (\beta-\beta')(G_{\beta'}*J*G_\beta)(x)
    .
\end{align}
This completes the proof of \eqref{eq:SL assumption}.
\end{proof}

\begin{proof}[Proof of \eqref{eq:Diff inequ assumption}]
We divide \eqref{eq: maindecompLT} by $p-p'$ and then take the limit $p' \rightarrow p$.
By the chain rule and the fact that $\tfrac{\mathrm{d}}{\mathrm{d}p}p(pJ)^R=(1+|R|)(pJ)^R$, we obtain
\begin{align}
     \frac{\partial G_\beta(x)}{\partial \beta}
     \hat{\chi}(p)
     = \sum_{u,v \in \mathbb{Z}^d}
     \sum_{\substack{T^- \in \mathcal{T}^1_{0,u} \\ R \in \mathcal{T}_{u,u}\\ T^+ \in \mathcal{T}^1_{v,x}} }
     (pJ)^{T^{-}} J_{v-u} (1+|R|) (pJ)^{R}(pJ)^{T^{+}} \1_{C}(T^+,T^-,R,(u,v)).
\end{align}
Since the number of vertices in $R$ is $1+|R|$ (i.e., $1+|R|=\sum_{y\in \mathbb Z^d}\mathds{1}_{y\in R}$), this gives
\begin{equation}\label{eq:expression derivative g lattice trees}
    \frac{\partial G_\beta(x)}{\partial \beta}
    =
    \frac{1}{\hat{\chi}(p)}
    \sum_{y \in \mathbb{Z}^d}
    \sum_{u,v \in \mathbb{Z}^d}
    \sum_{\substack{T^- \in \mathcal{T}^1_{0,u} \\ R \in \mathcal{T}_{u,y}
    \\ T^+ \in \mathcal{T}^1_{v,x}}}
    (pJ)^{T^{-}} J_{v-u} (pJ)^{R}(pJ)^{T^{+}} \1_{C}(T^+,T^-,R,(u,v)).
\end{equation}
We then write $\1_C$ as $1-(1-\1_C)$. As in the derivation of \eqref{eq:SL assumption} above, the contribution involving ``$1$'' in \eqref{eq:expression derivative g lattice trees} is equal to
\begin{equation}
    \frac{1}{\hat\chi(p)} \hat\chi(p) (G_{\beta}*J*G_\beta)(x)
    =
    (G_{\beta}*J*G_\beta)(x),
\end{equation}
which is the leading term in \eqref{eq:Diff inequ assumption}.

The factor $1-\mathds{1}_C$ enforces certain intersections to occur.
There are three possible (and non-exclusive) types of intersections:  $T^-$ and $R$ intersect, $R$ and $T^+$ intersect, or $T^-$ and $T^+$ intersect. See Figure~\ref{fig:LTH} for an illustration. Therefore, the contribution coming from the $1-\1_C$ term in \eqref{eq:expression derivative g lattice trees} is
bounded by $(\mathrm{I})+(\mathrm{II})+(\mathrm{III})$, where
\begin{align}
    (\mathrm{I}) = \frac{1}{\hat{\chi}(p)}
    \sum_{y \in \mathbb{Z}^d}
    \sum_{u,v \in \mathbb{Z}^d}
     \sum_{\substack{w \in \mathbb{Z}^d\\ w \neq u} }
    \sum_{\substack{T^- \in \mathcal{T}^1_{0,u} :\,  T^-\ni w \\  R \in \mathcal{T}_{u,y} :\, R\ni w   \\ T^+ \in \mathcal{T}^1_{v,x}}} (pJ)^{T^{-}} J_{v-u} (pJ)^{R}(pJ)^{T^{+}} , \\
    (\mathrm{II}) = \frac{1}{\hat{\chi}(p)}
   \sum_{y \in \mathbb{Z}^d}
    \sum_{u,v \in \mathbb{Z}^d}
    \sum_{w \in \mathbb{Z}^d }
    \sum_{\substack{T^- \in \mathcal{T}^1_{0,u}\\ R \in \mathcal{T}_{u,y} :\, R\ni w   \\ T^+ \in \mathcal{T}^1_{v,x} :\, T^+\ni w   }} (pJ)^{T^{-}} J_{v-u} (pJ)^{R}(pJ)^{T^{+}}
    ,\\
    (\mathrm{III})  = \frac{1}{\hat{\chi}(p)}
    \sum_{y \in \mathbb{Z}^d}
    \sum_{u,v \in \mathbb{Z}^d}
    \sum_{w \in \mathbb{Z}^d }
    \sum_{\substack{T^- \in \mathcal{T}^1_{0,u} :\, w \in T^-\\
    R \in \mathcal{T}_{u,y}\\
    T^+ \in \mathcal{T}^1_{v,x} :\:w \in T^+}} (pJ)^{T^{-}} J_{v-u} (pJ)^{R}(pJ)^{T^{+}}.
\label{eq: 3errors}
\end{align}
The restriction $w\neq u$ in (I) occurs because the
intersection between $T^-$ and $R$
must be in addition to
the common point $u$ that both of these trees are required to contain simply
by definition.

\begin{figure}[htb]
\begin{center}
\includegraphics[scale = 0.8]{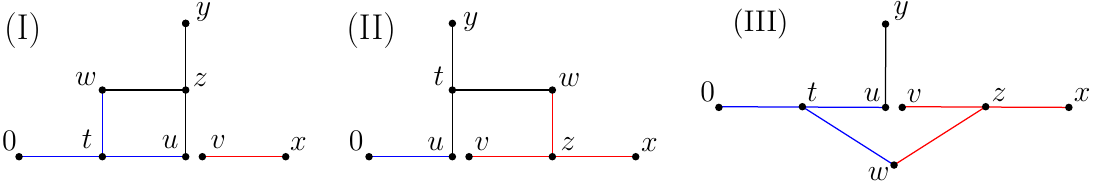}
\caption{An illustration of the error terms $({\rm I})$, $({\rm II})$, $({\rm III})$. For clarity, branches
not contributing to intersections are not shown.}
\label{fig:LTH}
\end{center}
\end{figure}

\begin{figure}[htb]
\begin{center}
\includegraphics[scale = 0.85]{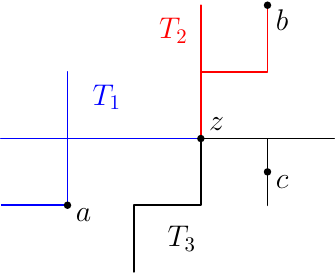}
\caption{An example of a decomposition for the tree graph inequality.}
\label{fig:treegraph}
\end{center}
\end{figure}

To estimate these terms, we use a tree-graph inequality.
If a tree $T$ contains three points $a,b,c$, then there exists a unique point $z$ such that the three points are bond-disjointly connected to $z$ in $T$. We can therefore divide $T$ into
three subtrees: $T_1 \in \mathcal{T}_{a,z}^1$, $T_2 \in \mathcal{T}_{b,z}^1$ and $T_3 \in \mathcal{T}_{c,z}$.
An example illustrating this decomposition is given in Figure~\ref{fig:treegraph}.
After removing the avoidance constraint between the subtrees,
we find that the sums over $T^-$ and $T^+$ in
\eqref{eq: 3errors} can be bounded above as
 \begin{align}
    \sum_{T^- \in \mathcal{T}^1_{0,u} :\, w \in T^-} (pJ)^{T^{-}} &\leq \sum_{z \in \mathbb{Z}^d}\sum_{\substack{T_1  \in \mathcal{T}^1_{0,z} \\T_2  \in \mathcal{T}^1_{w,z}
    \\
    T_3\in \mathcal{T}^1_{z,u}}}  (pJ)^{T_1}  (pJ)^{T_2}  (pJ)^{T_3}
    \leq \sum_{z \in \mathbb{Z}^d} G_{\beta}(z)G_{\beta}(z-w) G_{\beta}(u-z),
    \\
    \sum_{T^+ \in \mathcal{T}^1_{v,x} :\, w \in T^+} (pJ)^{T^{+}}& \leq
    \sum_{t \in \mathbb{Z}^d}\sum_{\substack{T_1  \in \mathcal{T}^1_{v,t}
    \\
    T_2  \in \mathcal{T}^1_{w,t}  \\ T_3\in \mathcal{T}^1_{t,x}}}  (pJ)^{T_1}  (pJ)^{T_2}  (pJ)^{T_3}
    \leq  \sum_{t  \in \mathbb{Z}^d} G_{\beta}(t-v)G_{\beta}(t-w) G_{\beta}(x-t).
 \end{align}
The factor $1/\hat{\chi}(p)$ in $(\mathrm{III})$ is cancelled by the sums over $R$ and $y$.
Inserting the above bounds into $(\mathrm{III})$ yields \begin{equation}
    (\mathrm{III})\leq \Big(G_\beta *\Big[[G_{\beta} * J *G_\beta]\cdot[G_\beta*G_\beta] \Big]*G_\beta\Big)(x).
 \end{equation}
An almost identical calculation for the terms $(\mathrm{I})$ and $(\mathrm{II})$ shows that the sum $(\mathrm{I})+(\mathrm{II})+(\mathrm{III})$ is bounded by $(G_\beta*H_\beta*G_\beta)(x)$, where
\begin{equation}
\label{eq:HLT}
    H_{\beta} = \Big(\big[G_{\beta} \cdot [G_{\beta}*G_{\beta}*G_{\beta}] - \delta_0 \big] *J
    +[J*G_{\beta}]\cdot [G_{\beta}*G_{\beta}*G_{\beta}]
    +[G_{\beta}*J*G_{\beta}] \cdot[ G_{\beta}*G_{\beta}]\Big).
\end{equation}
A detail for the bound on $(\mathrm{I})$ is that the square in its depiction in Figure~\ref{fig:LTH}
cannot have all four of its vertices be identical in the intersection of $T^-$ and $R$, since the sum is constrained by $w\neq u$.
The subtracted delta function eliminates the main contribution to this scenario. This concludes the proof.
\end{proof}

\subsection{Verification of Assumption~\ref{Ass:E} for each application}
\label{sec:AssE}

We now verify Assumption~\ref{Ass:E} for
each of the models discussed in Sections~\ref{sec:SAW}--\ref{sec:LT}.
Recall from \eqref{eq:definition error} the definition
\begin{equation}
	E(\beta) =
    \sup_{0\leq t\leq \beta}\left(\Vert H_t \Vert_1 + \frac{\Vert |x|_2^2\cdot H_t \Vert_1}{\xi (t)^2}\right),
\end{equation}
where $H_\beta$ is the model-dependent function that occurs in
Assumption~\ref{Ass:G}.
Assumption~\ref{Ass:E} asserts that
\begin{equation}
\label{eq:AssEconc}
    E(\beta(\deltamain)) < \deltamain,
\end{equation}
where $\deltamain$ is given by Theorem~\ref{thm:maintheorem}.
By definition,
\begin{equation}
	E(\beta) \le E_0(\beta)+ E_2(\beta)
\end{equation}
with
\begin{equation}
\label{eq:Eapp}
     E_0(\beta)=
     \sup_{0\leq t\leq \beta} \Vert H_t \Vert_1
     , \quad
      E_2(\beta)=
     \sup_{0\leq t\leq \beta} \frac{\Vert |x|_2^2\cdot H_t \Vert_1}{\xi (t)^2}
    .
\end{equation}

The function $H_\beta$ is identified
for self-avoiding walk in \eqref{eq:HSAW},
for continuous-time self-avoiding walk in \eqref{eq:HCTWSAW},
for percolation in \eqref{eq:Hperc},
for $|\varphi|^4$ in \eqref{eq:expression H1 spins},
for the Ising model in \eqref{eq:HIsing},
for the XY model in \eqref{eq:HXY},
and for lattice trees in \eqref{eq:HLT}.
Explicitly, with $K_\beta=G_\beta+\beta F_\beta$,
\begin{align}
\label{eq:HSAWII}
    H^{\rm SAW}_\beta &=\lambda \delta_{0}(G_\beta*F_\beta) , \\
\label{eq:HCTWSAWII}
    H_\beta^{\rm CTWSAW} &= 2\lambda \delta_{0}(G_\beta*F_\beta),\\
\label{eq:Hphi4II}
    H^{|\varphi|^4}_\beta &=6\lambda \delta_0 (F_\beta*G_\beta) , \\
\label{eq:HIsingII}
    H_\beta^{\rm Ising} &=3\big[ (\delta_0+\beta J)*(\delta_0+\beta J)\big] (K_\beta*J*K_\beta)(0),\\
\label{eq:HXYII}
    H_\beta^{\rm XY} &=12 \big[ (\delta_0+\beta J)*(\delta_0+\beta J) \big] (K_\beta*J*K_\beta)(0),\\
\label{eq:HpercII}
    H^{{\rm{perc}}}_\beta  &= G_\beta \cdot (G_\beta*F_\beta) , \\
\label{eq:HLTII}
    H^{{\rm{LT}}}_{\beta}
    &= (G_{\beta} \cdot [G_{\beta}*G_{\beta}*G_{\beta}] - \delta_0 )*J
    + F_{\beta}\cdot (G_{\beta}*G_{\beta}*G_{\beta})
    \nonumber \\ & \qquad\qquad + (G_{\beta}*F_{\beta}) \cdot ( G_{\beta}*G_{\beta}) .
\end{align}
By \eqref{eq:bound on beta(delta)} and the fact that
$\deltamain < \frac 12$, we have $\beta(\deltamain) < 2$, so
factors of $\beta$ are not problematic.
In all cases, $E_0(\beta)=\|H_\beta\|_1$, but we do not know that the ratio in the definition of $E_2(\beta)$ in \eqref{eq:Eapp} is monotone in $t$, so we cannot ignore the supremum for $E_2$.

Since we have verified Assumption~\ref{Ass:G}, Theorem~\ref{thm:maintheorem} tells us that, for every $x\in\mathbb{Z}^d$ and every $\beta
\in [0,\beta(\deltamain)]$, we have the bound
\begin{equation}
\label{eq:AssEhyp}
	G_\beta(x)\leq \delta_0(x)
    + \frac{ \bfC}{\sigma_J^d}
    \left(\frac{\sigma_J}{\sigma_J \vee|x|} \right)^{d-2-\varepsilon}
    \exp\left(-\bfc\frac{|x|}{\xi(\beta)}\right).
\end{equation}
Under Assumption~\ref{Ass:G},
we also have the bound \eqref{eq:Fmainbd}, which states that
for every $x\in\mathbb{Z}^d$ and every $\beta
\in [0,\beta(\deltamain)]$,
\begin{equation}
\label{eq:EhypF}
	F_\beta(x)\leq
    \frac 12 \frac{\bfC}{\sigma_J^d}
    \left(\frac{\sigma_J}{\sigma_J \vee|x|} \right)^{d-2-\varepsilon}
    \exp\left(-\bfc\frac{|x|}{\xi(\beta)}\right).
\end{equation}
The constants $\bfc,\bfC,\deltamain$ depend only on $c_0,\varepsilon,d$.
In particular, they do not depend on the admissible kernel $J$ nor on any
particular model in our applications.
For nearest-neighbour models (for which $\sigma_J=1$) with an inherent small
parameter $\lambda$, we will prove that $E(\beta(\deltamain)) < O(\lambda)$.
This establishes \eqref{eq:AssEconc} once we take $\lambda$ sufficiently small.
For spread-out models (which have large $\sigma_J$), we will instead prove
that \eqref{eq:AssEhyp} implies that $E(\beta(\deltamain)) < O(\sigma_J^{-d})$.
This establishes \eqref{eq:AssEconc} once we take $\sigma_J$ sufficiently large.

The $L^1$ norm of $H_\beta$ for self-avoiding walk and $|\varphi|^4$ is an open
bubble diagram, and for Ising and XY something quite similar is true.
The $L^1$ norm of $H^{\rm perc}_\beta$ is the open triangle
diagram, and the $L^1$ norm of $H^{\rm LT}_\beta$ involves the open square
diagram.
The early mathematical development of the theory of high-dimensional
critical phenomena was based on the observation that when the critical exponent
$\eta$ is equal to its mean-field value $\eta=0$, the critical bubble, triangle and
square diagrams are respectively finite above dimensions $4$, $6$ and $8$.
This was an important step in the understanding, in the 1980s, that $4$, $6$ and $8$ are the upper critical
dimensions for the models \cite{Aize82,Froh82,AN84,BFG86}.
The finiteness of the diagrams above the upper critical dimension was first
proved in various contexts for self-avoiding walk \cite{BS85,HS92a}, spin
systems \cite{Aize82,Froh82}, percolation \cite{HS90a}, and lattice trees \cite{HS90b}.
As outlined in Section~\ref{sec:previous-work},
our contribution is to give a
new and unified approach to the subject.

\subsubsection{Self-avoiding walk and $|\varphi|^4$}
\label{sec:SAWII}

Fix any $\beta \in [0,\beta(\deltamain)]$.
By \eqref{eq:HSAWII}--\eqref{eq:Hphi4II}, it suffices to consider
\begin{equation}
	H_\beta(x) =\lambda  (G_\beta*F_\beta)(0)  \delta_{0}(x) ,
\end{equation}
since the factors $2$ or $6$ are of no significance.
In this case, $E_2(\beta)=0$ and
$E(\beta)= \lambda (G_\beta*F_\beta)(0)$.

Let $\bar G_\beta = G_\beta -\delta_0$.  Then
\begin{equation}
    G_\beta*F_\beta = F_\beta + \bar G_\beta*F_\beta.
\end{equation}
By \eqref{eq:EhypF}, $F_\beta(0)\lesssim \sigma_J^{-d}$.
Also, by \eqref{eq:AssEhyp} and \eqref{eq:EhypF},
\begin{align}
    (\bar G_\beta*F_\beta) (0)
    & \lesssim
    \frac{1}{\sigma_J^d}\frac{1}{\sigma_J^d}
    \sum_{x\in \mathbb Z^d}
    \left(\frac{\sigma_J}{\sigma_J \vee|x|} \right)^{2d-4-2\varepsilon}.
\end{align}
For dimensions $d>4+2\varepsilon$,
the sum is bounded for the nearest-neighbour $J$ (for which $\sigma_J=1$), so $E(\beta) =
\lambda
(G_\beta*F_\beta)(0) = O(\lambda)$.
This bound is uniform in $\beta$, and in particular it holds at
$\beta(\deltamain)$.

For the spread-out case, by Riemann sum approximation,
if $d>4+2\varepsilon$ then
\begin{align}
\label{eq:ESAWbd}
    (\bar G_\beta*F_\beta) (0)
    & \lesssim
    \frac{1}{\sigma_J^d}
    \int_{\mathbb R^d} \frac{1}{1\vee |u|^{2d-4-2\varepsilon}} \mathrm{d}u
    \lesssim
    \frac{1}{\sigma_J^d}
    ,
\end{align}
so $E(\beta)=O(\sigma_J^{-d})$.  Again, the bound holds at
$\beta(\deltamain)$.

\subsubsection{Ising and XY models}

Let $t\leq \beta(\deltamain)\leq 2$.
By \eqref{eq:HIsingII}--\eqref{eq:HXYII},
it suffices to consider $H_t= \big[ (\delta_0+t J)*(\delta_0+t J) \big] (K_t*J*K_t)(0)$.
Its zeroth and second moments satisfy
\begin{align}
    \sum_{x\in \mathbb Z^d}H_t(x)
    & =\sum_{x\in \mathbb Z^d}\big[ \delta_{0}(x)+2t J_x+ t^2(J*J)(x)  \big]  (K_t*J*K_t)(0)
    \nonumber \\
    & \le
    \big[ 1+4+4    \big]  (K_t*J*K_t)(0),
\\
    \frac{1}{\xi(t)^2}\sum_{x\in \mathbb Z^d}|x|^2_2 H_t(x)
    & =
    \frac{1}{\xi(t)^2}\sum_{x\in \mathbb Z^d}|x|^2_2\big[ \delta_{0}(x)+2t J_x+ t^2(J*J)(x)  \big]  (K_t*J*K_t)(0)
    \nonumber \\
    & \le
    \frac{1}{\xi(t)^2}
    \big[ 1+4\sigma_J^2 + 8\sigma_J^2   \big]  (K_t*J*K_t)(0).
\end{align}
By Proposition~\ref{prop: rough bounds chi xi}, $\xi(t) \ge \xi(0)=\sigma_J$.
Also, by definition of $K_t$,
\begin{equation}
\label{eq:KJKspin}
    (K_t*J*K_t)(0)
    =
    (G_t*F_t)(0) + 2t(F_t*F_t)(0) + t^2 (F_t *J*F_t)(0).
\end{equation}
It therefore suffices to prove that the right-hand side of \eqref{eq:KJKspin} is bounded by a multiple of
$\sigma_J^{-d}$.
The first two terms are indeed bounded by a multiple of
$\sigma_J^{-d}$, exactly as in Section~\ref{sec:SAWII}.

For the last term, it suffices to show that
\begin{equation}
\label{eq:JFwant}
    (J*F_t)(x) \lesssim
    \frac{1}{\sigma_J^d}
    \left(\frac{\sigma_J}{\sigma_J \vee|x|} \right)^{d-2-\varepsilon},
\end{equation}
since then the computations of Section~\ref{sec:SAWII} can be applied.
By \eqref{eq:EhypF},
\begin{equation}\label{eq:JFwant'}
    (J*F_t)(x) \lesssim
    \frac{1}{\sigma_J^d}\sum_{y\in \mathbb Z^d}
    J_y \left(\frac{\sigma_J}{\sigma_J \vee|x-y|} \right)^{d-2-\varepsilon}
    .
\end{equation}
We divide the sum according to whether (i) $|y|\le |x|/2$, or
(ii) $|y|\ge |x|/2$.
In case (i), up to a constant we may replace $|x-y|$ by $|x|$ and then
bound the resulting sum over $y$ by $1$, so this contribution does satisfy the bound \eqref{eq:JFwant}.
In case (ii), since $J$ has finite range $R_J$, the summand is nonzero only if $|x|\le 2R_J \le 2c_0^{-1}\sigma_J$, where $c_0\in (0,1]$ is given by Definition~\ref{Def:J}.
The contribution to the sum \eqref{eq:JFwant'} from case (ii) is therefore bounded by
\begin{equation}
    \frac{1}{\sigma_J^d}\cdot 1
    =
    \frac{1}{\sigma_J^d}
    \left( \frac{ 2c_0^{-1}\sigma_J}{2c_0^{-1}\sigma_J\vee |x|}\right)^{d-2-\epsilon}
    \le
    \frac{1}{\sigma_J^d}
    \left( \frac{ 2c_0^{-1}\sigma_J}{ \sigma_J\vee |x|}\right)^{d-2-\epsilon}.
\end{equation}
This completes the proof.

\subsubsection{Bernoulli percolation}

For $E_0(t)$, we use
\begin{align}
    \|H_t\|_1 & = (G_t*G_t*F_t)(0) = F_t(0) + 2(\bar G_t*F)(0) + (\bar G_t*\bar G_t*F_t)(0).
\end{align}
The first term on the right-hand side is $O(\sigma^{-d})$ by \eqref{eq:EhypF},
and the second term  is $O(\sigma_J^{-d})$ by \eqref{eq:ESAWbd}.  For the third
term, we again use Riemann sum approximation to see that,
for $d>6+3\varepsilon$,
\begin{align}
\label{eq:Epercbd}
    (\bar G_t* \bar G_t* F_t) (0)
    & \lesssim
    \frac{1}{\sigma_J^d}
    \int_{\mathbb R^d} \int_{\mathbb R^d}\left(
    \frac{1}{1\vee |u|}
    \frac{1}{1\vee |v|}
    \frac{1}{1\vee |u-v|}\right)^{d-2-\varepsilon}
    \mathrm{d}u \, \mathrm{d}v
    \lesssim
    \frac{1}{\sigma_J^d}
    .
\end{align}
This shows that $E_0(t)=O(\sigma_J^{-d})$.

We did not use the exponential decay for $E_0(t)$, but for $E_2(t)$ we will,
as follows.
Let
\begin{equation}
    W_t(x) = \frac{|x|^2}{\xi(t)^2} G_t(x) = \frac{|x|^2}{\xi(t)^2} \bar G_t(x).
\end{equation}
Then
\begin{equation}
    E_2(t) = \sum_{x\in \mathbb Z^d} W_t(x)(G_t*F_t)(x) =(W_t*G_t*F_t)(0).
\end{equation}
Since $s^2e^{-\bfc s}$ is bounded for $s \in [0,1]$, it follows from \eqref{eq:AssEhyp}
that
\begin{equation}
\label{eq:Kexpbd}
    W_t(x) \lesssim
    \frac{1}{\sigma_J^d}
    \left(\frac{\sigma_J}{\sigma_J \vee|x|} \right)^{d-2-\varepsilon}.
\end{equation}
The factors in $(W_t*G_t*F_t)(0)$ therefore obey the same estimates as the factors
in \eqref{eq:Epercbd}, so we also have $E_2(t)=O(\sigma^{-d})$.

\subsubsection{Lattice trees}

By definition of $H_t$ in \eqref{eq:HLTII}, and with $\bar G_t = G_t -\delta_0$,
\begin{align}
    E_0(t) &= (G_t^{*4}(0)-1) + 2(F_t*G_t^{*3}(0))
    \nonumber \\ & =
    3\bar G_t(0) + 6 \bar G_t^{*2}(0) + 3\bar G_t^{*3}(0) + \bar G_t^{*4}(0)
    \nonumber \\ & \qquad
     +2F_t(0) +6(F_t*\bar G_t)(0) + 6 (F_t * \bar G_t^{*2})(0) + 2(F_t * \bar G_t^{*3})(0).
\label{eq:ELTbd}
\end{align}
A bound $O(\sigma_J^{-d})$ was obtained above for the first three
terms in each of second and third lines of \eqref{eq:ELTbd},
assuming $d>6+3\varepsilon$.  The last terms on each of those lines
obey the same upper bound,
which is a Riemann sum approximation to
\begin{equation}
    \frac{1}{\sigma_J^d}
    \int_{\mathbb R^d}  \int_{\mathbb R^d}  \int_{\mathbb R^d}
    \left(
    \frac{1}{1\vee |u|} \frac{1}{1\vee |v-u|}\frac{1}{1\vee |w-v|}\frac{1}{1\vee |w|}
    \right)^{d-2-\varepsilon}
    \mathrm{d}u\, \mathrm{d}v\, \mathrm{d}w.
\end{equation}
The integral converges for $d>8+4\varepsilon$, so $E_0(t)=O(\sigma_J^{-d})$.

The estimate for $E_2(t)$ follows similarly, again using \eqref{eq:Kexpbd}.  We present the details only for the contribution from the first term in \eqref{eq:HLTII}.  For this term, we seek a bound $O(\sigma_J^{-d})$ on
\begin{equation}
    \frac{1}{\xi(t)^2}
    \sum_{x,y\in \mathbb Z^d}
    |x|_2^2 G_{\beta}(y) [G_{\beta}*G_{\beta}*G_{\beta}](y) J(x-y).
\end{equation}
We can replace $|x|_2^2$ by $|y|_2^2 + |x-y|_2^2$, since the cross term gives zero contribution to the sum by the $\mathbb Z^d$-symmetry.
For $|y|_2^2$, we use $\sum_{x\in \mathbb Z^d}J(x-y)=1$ and bound  $|y|_2^2G_t(y)/\xi(t)$ using \eqref{eq:Kexpbd}.  This brings us to
an expression that we have treated already for $E_0(t)$.
For $|x-y|_2^2$, we use
\begin{equation}
    \frac{1}{\xi(t)^2} \sum_{x\in \mathbb Z^d} |x-y|_2^2 J(x-y)
    = \frac{\sigma_J^2}{\xi(t)^2} = \frac{\xi(0)^2}{\xi(t)^2} \le 1
    ,
\end{equation}
due to the monotonicity of $\xi$ proved in Proposition~\ref{prop: rough bounds chi xi}.  The remaining sum over $y$ has been shown in the bound on $E_0(t)$ to be $O(\sigma_J^{-d})$.  This completes the proof.

\appendix

\section{Random walk theorems}
\label{appendix: rw}

In this appendix, we prove the averaged and pointwise
anti-concentration and Green function estimates stated in
Theorem~\ref{thm:estimate RW} and
Proposition~\ref{Prop:Green}, respectively.
We prove Theorem~\ref{thm:estimate RW} in Section~\ref{sec:regrw},
and use it to prove Proposition~\ref{Prop:Green} in Section~\ref{sec:Green-pointwise}.
An important preliminary ingredient is the anti-concentration estimate
proved in Section~\ref{sec:Esseen}.

\subsection{Esseen's anti-concentration estimate}
\label{sec:Esseen}

A fundamental ingredient in our proof of Theorem \ref{thm:estimate RW} is the following result from Esseen \cite{Esse66}.
For its statement, for any $M\geq 0$ and $z\in \mathbb R^d$, we introduce the set
\begin{equation}
D(z;M):=\{x\in \mathbb R^d:
\min_{1 \le i \le d}|x_i-z_i|\leq M\}.
\end{equation}
 Recall from \eqref{eq:B-box-def} that for $y\in \mathbb R^d$ we set $B_M(y)=\{z\in \mathbb R^d: |z-y|\leq M\}$, and we also write $B_M=B_M(0)$.
A random variable $Y=(Y^{(1)},\ldots,Y^{(d)})$ on $\mathbb R^d$ is called \emph{sign invariant} if
the $2^d$
random vectors $(\pm Y^{(1)},\ldots,\pm Y^{(d)})$ have the same distribution.

\begin{Thm}[\hspace{1pt}{\cite[Theorem~3]{Esse66}}]
\label{thm:esseen}
Let $d\geq 1$. There exists $C=C(d)>0$ such that the following holds. Let $n\geq 1$, and let $Y_1,\ldots,Y_n$ be mutually independent and sign-invariant random variables
on $\mathbb R^d$. Set $S_n:=Y_1+\ldots+Y_n$. Then, for every $M\geq 1$,
\begin{equation}
	\sup_{y\in \mathbb R^d}\mathbb P[S_n\in B_M(y)]\leq \frac{C}{\Big(\displaystyle\sum_{k=1}^n\big(1-\sup_{z\in \mathbb R^d}\mathbb P[Y_k\in D(z;M)]\big)\Big)^{d/2}}.
\end{equation}
\end{Thm}

As stated, Theorem~\ref{thm:esseen} is not well-suited for regular random walks. Indeed, to use it directly would require showing that for a generic $(\creg,\Creg)$-regular random walk $X$ there exists $M>0$ such that $\mathbb P[X_1\notin D(z;M)]\geq c_0(d,\creg,\Creg)>0$, uniformly in $z$. This property is false in general: it does not hold for the simple random walk on $\mathbb Z^d$ with $z=0$ and $M=1$. Nevertheless, as we will see below in \eqref{eq: nonlazyness}, $(\creg,\Creg)$-regular random walks satisfy a closely related bound of the form $\mathbb P[X_1\notin B_M]\geq c_1(d,\creg,\Creg)>0$, for an $M$ depending on the variance $\sigma^2$ of $X_1$. The purpose of the next result is
to adapt Theorem~\ref{thm:esseen} to apply under this alternative (weaker) assumption.

We say that a random variable $Y$ on $\mathbb R^d$ is \emph{symmetric} if
the $2^d d!$
random vectors $(\pm Y^{(\pi(1))},\ldots,\pm Y^{(\pi(d))})$, where $\pi$ ranges over all permutations of $\{1,\ldots,d\}$, have the same distribution. This is a stronger condition than being sign invariant. A random walk $X=(X_k)_{k\geq 0}$ on $\mathbb R^d$ is called \emph{symmetric} if $X_1$ (and therefore $X_{i+1}-X_i$ for every $i\geq 2$) is symmetric.

\begin{Coro}\label{cor: coro of esseen} Let $d\geq 1$ and $a_0>0$.
Let $X=(X_k)_{k\geq 0}$ be a symmetric random walk started at $0$. Let $M\geq 1$ and assume that
\begin{equation}
	\mathbb P[X_1\notin B_M]\geq a_0.
\end{equation}
Then there exists $C=C(a_0,d)>0$ such that, for every $n\geq 1$,
\begin{equation}\label{eq:target esseen}
	\sup_{y\in \mathbb R^d}\mathbb P[X_n\in B_M(y)]\leq \frac{C}{n^{d/2}}.
\end{equation}
\end{Coro}

\begin{proof}
By Theorem~\ref{thm:esseen}, it would be sufficient to show that there exists $c>0$ such that
\begin{equation}
	\sup_{z\in \mathbb R^d}\mathbb P[X_1\in D(z;M)]\leq 1-c.
\end{equation}
However, with our hypotheses, this is not necessarily true.
Instead, we prove that there exists $c_1=c_1(a_0,d)>0$ such that
\begin{equation}\label{eq:pbe1}
	\sup_{z\in \mathbb R^d}\mathbb P[X_d\in D(z;M)]\leq 1-c_1,
\end{equation}
as follows.

For $1\leq i \leq d$, let $\mathcal E_i:=\{x\in \mathbb R^d: x_i> M, \: x_{j}\geq 0 \: \text{if}\; j\neq i\}$. Since $X_1$ is symmetric, we have
\begin{equation}
	\mathbb P[X_1\in \mathcal E_i]\geq \frac{a_0}{2d\cdot 2^{d-1}}=:\varepsilon.
\end{equation}
Again by symmetry,
\begin{equation}
    \max_{z\in \mathbb R^d}\mathbb P[X_d\in D(z;M)]
    =\max_{z \in \mathbb R^d: \: \max_{1 \le i \le d}z_i \le 0}
    \mathbb P[X_d\in D(z;M)].
\end{equation}
Finally, if $z\in \mathbb R^d$ with $z_i\leq 0$ for all $i$, then
\begin{equation}
	\mathbb P[X_d\notin D(z;M)]\geq \mathbb P\Big[\bigcap_{i=1}^d \{X_{i}-X_{i-1}\in \mathcal E_i\}\Big]\geq \varepsilon^d=:c_1.
\end{equation}
This proves \eqref{eq:pbe1}.

We now prove \eqref{eq:target esseen}. If $n\leq d-1$, we just bound the probability by $1$. Assume that $n\geq d$. In this case, we set $\ell:=\lfloor n/d\rfloor$,
and for $1\leq k \leq \ell$ define
\begin{equation}
	Y_k:= \sum_{j=(k-1)d+1}^{kd}(X_j-X_{j-1}), \qquad Y_{\ell+1}:=X_n-(Y_1+\ldots+Y_\ell),
\end{equation}
so that
\begin{equation}
	X_n=Y_1+\ldots+Y_{\ell+1}.
\end{equation}
The random variables $Y_1,\ldots,Y_{\ell+1}$ are mutually independent and sign invariant. Also, for every $1\leq k\leq \ell$, $Y_k$ has the same law as $X_d$. We apply Theorem \ref{thm:esseen} to the random variables $Y_1,\ldots,Y_{\ell+1}$ and find $C_1=C_1(d)>0$ such that for every $n\geq d$,
\begin{align}
	\sup_{y\in \mathbb R^d}\mathbb P[X_n\in B_M(y)]&\leq \frac{C_1}{\Big((1-\mathbb P[Y_{\ell+1}\in D(z;M)])+\displaystyle\sum_{k=1}^\ell\big(1-\sup_{z\in \mathbb R^d}\mathbb P[Y_k\in D(z;M)]\big)\Big)^{d/2}}\notag
	\\&\leq \frac{C_1}{\Big(\ell \big(1-\sup_{z\in \mathbb R^d}\mathbb P[X_d\in D(z;M)]\big)\Big)^{d/2}}\leq \frac{C_2}{n^{d/2}},
\end{align}
where in the second line we used \eqref{eq:pbe1}, and where $C_2=C_2(a_0,d)>0$. As a result, for every $n\geq 1$,
\begin{equation}
	\sup_{y\in \mathbb R^d}\mathbb P[X_n\in B_M(y)]\leq \frac{C}{n^{d/2}},
\end{equation}
where $C:=d^{d/2}\vee C_2$. This concludes the proof.
\end{proof}

\subsection{Regular random walk: proof of Theorem~\ref{thm:estimate RW}}
\label{sec:regrw}

We now prove Theorem~\ref{thm:estimate RW}.
Recall from Definition~\ref{def:regularrw} that a
$(\cM,\CM)$-\emph{regular} random walk on $\mathbb R^d$
has $X_0=0$,
variance $\sigma^2=\mathbb E[|X_1|_2^2]<\infty$,
is invariant under
permutation of coordinates and/or replacement of a coordinate by its negative, and
its moment generating function $M(t) = \mathbb E [e^{t(\mathbf{e}_1 \cdot X_1)/\sigma}]$ obeys
\begin{align}
    M(\cM) \le \CM.
\end{align}

It follows from the elementary inequality \eqref{eq: general bound laplace transform-bis}
that, for $0 \le t \le \cM$,
\begin{equation}
\label{eq:t0def}
    M(t) \le 1 + \frac{t^2}{2d} + \frac{t^4}{\cM^4}M(\cM)
    \le
    1 + \frac{t^2}{4}  \left( \frac{2}{d} +\frac{4t^2\CM}{\cM^4} \right).
\end{equation}
Therefore, for $d>2$, we can choose
\begin{equation}
    t_0=\frac{\cM^2}{2}\sqrt{\frac{d-2}{d\CM}} \wedge \cM\wedge 1\in (0,1]
\end{equation}
to achieve
\begin{equation}
\label{eq:t0def}
    M(t) \le 1+\frac{t^2}{4}\quad \text{for all} \quad t \in [0,t_0].
\end{equation}
We prove the anti-concentration part of Theorem~\ref{thm:estimate RW},
in the next proposition.

\begin{Prop}
\label{prop:anti}
{\rm (Anti-concentration inequality).} Let $d>2$.
There is a constant $\CAC$ (depending on $d,\creg,\Creg$) such that
for every $m\ge 1$, every $y\in \mathbb R^d$, and every $\tau \in [0,t_0]$,
\begin{equation}
\label{eq:add-to-i}
    \mathbb P[X_m\in B_\sigma(y)]
    \leq
    \frac{\CAC}{m^{d/2}}
    e^{m\tau^2/8}
    e^{- \tau |y|/ 2\sigma  }
    .
\end{equation}
\end{Prop}
We will apply two special cases of \eqref{eq:add-to-i}, obtained
from two choices of $\tau$.  The first is
\begin{align}
\label{eq:it0}
    \mathbb P[X_m\in B_\sigma(y)]
    &\leq
    \frac{\CAC e^{t_0^2/8}}{m^{d/2}}
    \exp\left(-t_0 \frac{|y|}{2 \sigma \sqrt{m}} \right)
      \qquad (\tau =t_0/\sqrt{m}).
\end{align}
The anti-concentration bound \eqref{eq:ac-estimate} of Theorem~\ref{thm:estimate RW} follows immediately from
\eqref{eq:it0}, once we require that
\begin{equation}
\label{eq:crw-anti-concentration}
    \crw \le \frac{t_0}{2}, \qquad \Crw \ge \CAC e^{t_0^2/8}.
\end{equation}
For the second choice,
given $s \in [0,1]$, let $\tau =s \wedge t_0$.  Then
$m^2\tau /8 \le m^2 s/8$ and  $-\tau \le - t_0 s$ (since $0\leq s\le 1$ and $0\leq t_0\le 1$), so
\eqref{eq:add-to-i} implies that
\begin{align}
\label{eq:is}
    \mathbb P[X_m\in B_\sigma(y)]
    &\leq
    \frac{\CAC }{m^{d/2}} e^{ms^2/8}
    \exp\left(- \frac{t_0s|y|}{2\sigma}\right)
      \qquad (\tau = s \wedge t_0  ).
\end{align}

The constants $C_i$ below depend only on $\creg$, $\Creg$, and $d$. Before proving Proposition \ref{prop:anti}, we start with a weaker estimate.

\begin{Lem}\label{lem: weak anticonc} Let $d\geq 1$. There exists $C=C(\creg,\Creg,d)>0$ such that, for every $m\geq 1$ and every $y\in \mathbb R^d$,
\begin{equation}
	\mathbb P[X_m\in B_\sigma(y)]\leq \frac{C}{m^{d/2}}.
\end{equation}
\end{Lem}
\begin{proof} Let $\tilde X_1=X_1/\sigma$ and $\delta>0$. By the Cauchy--Schwarz inequality,
\begin{equation}
	1=\mathbb E[|\tilde{X}_1|^2]\leq \delta^2+\mathbb E[\mathds{1}_{|\tilde{X}_1|> \delta}|\tilde{X}_1|^2]\leq \delta^2+\sqrt{\mathbb P[|\tilde{X}_1|> \delta]\mathbb E[|\tilde{X}_1|^4]}.
\end{equation}
This gives
\begin{equation}
\frac{(1-\delta^2)^2}{\mathbb E[|\tilde{X}_1|^4]}\leq  \mathbb P[|X_1|> \delta\sigma].
\end{equation}
By symmetry,
for any $t \ge 0$ we have
 \begin{align}
 \label{eq:2dM}
    \mathbb E[ e^{t |\tilde X_1|}]
    & \le
    2d \, \mathbb E \big[e^{t (\mathbf{e}_1 \cdot\tilde  X_1)} \1_{|\tilde{X}_1|= (\mathbf{e}_1 \cdot\tilde  X_1)} \big]
    \le
    2d \, M(t),
 \end{align}
so
\begin{equation}
	\frac{\cM^4}{4!}\mathbb E[|\tilde X_1|^4]\leq \mathbb E[e^{\cM |\tilde X_1|}]
    \leq 2d\CM.
\end{equation}
Therefore, by choosing $\delta=\delta(\cM,\CM,d)$ small enough, we obtain
\begin{equation}
\label{eq: nonlazyness}
	\mathbb P[X_1\notin B_{\delta\sigma}]\geq \delta.
\end{equation}
The inequality \eqref{eq: nonlazyness} is a form of ``non-lazyness'' of the random walk.

Now, thanks to Corollary \ref{cor: coro of esseen} (applied to $X$, $a_0=\delta$, and $M=\delta\sigma$), there exists $C_1>0$ which depends on $\creg$, $\Creg$, and $d$, such that for every $m\geq 1$ and every $y\in \mathbb R^d$,
\begin{equation}\label{eq:proof lemma 1 rw 1}
	\mathbb P[X_m\in B_{\delta \sigma}(y)]\leq \frac{C_1}{m^{d/2}}.
\end{equation}
Then, after adding the contributions to the big box from small boxes, \eqref{eq:proof lemma 1 rw 1} gives the existence of $C_2=C_2(\creg,\Creg,d)>0$ such that, for every $m\geq 1$,
\begin{equation}\label{eq:proof lemma 1 rw 2}
	\sup_{y\in \mathbb R^d}\mathbb P[X_m\in B_{\sigma}(y)]\leq \frac{C_2}{m^{d/2}}.
\end{equation}
This concludes the proof.
\end{proof}

We are now in a position to prove Proposition \ref{prop:anti}.

\begin{proof}[Proof of Proposition~\textup{\ref{prop:anti}}]

By Lemma \ref{lem: weak anticonc}, there exists $C>0$ such that for every $m\geq 1$ and every $y\in \mathbb R^d$,
\begin{equation}\label{eq:proof rw 1}
	\mathbb P[X_m\in B_{\sigma}(y)]\leq \frac{C}{m^{d/2}}.
\end{equation}
To improve \eqref{eq:proof rw 1} with an exponential factor, we proceed as follows.
Suppose that $X_m\in B_{\sigma}(y)$.
If $|y|\le \sigma$ then we can insert the exponential factor in
the upper bound \eqref{eq:proof rw 1} at the cost of increasing $C$, so
we assume now that $|y|>\sigma$.

Let $k=\lceil m/2 \rceil$.  Either $|X_k| \le (|y|-\sigma)/2$ or $|X_k|\ge (|y|-\sigma)/2$.
In the former case, $|X_m-X_k|\ge(|y|-\sigma)/2$.  Therefore,
\begin{align}
    \mathbb P[X_m\in B_{\sigma}(y)]
    & \le
    \mathbb P\Big[X_m\in B_{\sigma}(y) \:\Big|\: |X_m-X_k|\ge (|y|-\sigma)/2\Big]
    \mathbb P[|X_m-X_k|\ge (|y|-\sigma)/2]
    \nonumber \\ & \;\; +
    \mathbb P\Big[X_m\in B_{\sigma}(y) \:\Big|\: |X_k|\ge(|y|-\sigma)/2\Big]
    \mathbb P[|X_k|\ge(|y|-\sigma)/2]
    .
\label{eq:proof_rw_1.5}
\end{align}
For the conditional probabilities, we first observe that
$(X_0,X_1,\ldots,X_m)$ has the same law as $(X_m-X_m,\ldots, X_m-X_0)$.
With $k'=\lfloor m/2\rfloor$, this gives
\begin{equation}
	\mathbb P\Big[X_m\in B_{\sigma}(y) \:\Big|\: |X_m-X_k|\ge (|y|-\sigma)/2\Big]
    =\mathbb P\Big[X_m\in B_{\sigma}(y) \:\Big|\: |X_{k'}|\ge (|y|-\sigma)/2\Big].
\end{equation}
This allows the two terms in \eqref{eq:proof_rw_1.5} to be bounded in the same way.
By the Markov property and \eqref{eq:proof rw 1},
\begin{equation}
    \mathbb P\Big[X_m\in B_{\sigma}(y) \:\Big|\: |X_{k'}|\ge (|y|-\sigma)/2\Big]
    \le
    \frac{C}{(m-k')^{d/2}}.
\end{equation}

With \eqref{eq:proof rw 1}, and with $\ell=k=\lceil m/2 \rceil$ or $\ell =k' = m-k=
\lfloor m/2 \rfloor$ giving the
maximum of the two options on the right-hand side, this gives
the existence of $C_1$ such that
\begin{equation}
\label{eq:proof rw 4}
	\mathbb P[X_m \in B_\sigma(y)]
    \leq
    \frac{C_1}{m^{d/2}}\mathbb P[|X_{\ell}|\geq (|y|-\sigma)/2].
\end{equation}
Let $Z_j$ denote the first coordinate of $X_j$.
By symmetry, $\mathbb E[Z_1^2]=d^{-1}\sigma^2$ and
\begin{equation}
\label{eq: proof consequence regularity 1}
	\mathbb P[|X_{\ell}|\geq (|y|-\sigma)/2]
    \leq (2d) \mathbb P[Z_{\ell} \geq (|y|-\sigma)/2].
\end{equation}
The moment generating function $M(t)=\mathbb E[\exp(tZ_1/\sigma)]$ is well defined for $|t|\leq\cM$.
Let $t_0$ be chosen in \eqref{eq:t0def} so that
$M(t)\leq 1+\frac{t^2}{4}$ for all $|t|\leq t_0$.
By Markov's inequality, if $\tau \in [0,t_0]$ then
\begin{equation}
\label{eq:proof rw 5}
	\mathbb P[Z_{\ell}\geq (|y|-\sigma)/2]
    \leq [M(\tau)]^{\lceil m/2 \rceil}\exp\left(-\frac{\tau}{2}\frac{|y|-\sigma}{\sigma }\right)
    \le e^{t_0/2}\exp\left(\frac{\tau^2 m}{8} -\frac{\tau}{2}\frac{|y|}{\sigma }\right).
\end{equation}
The combination of
\eqref{eq:proof rw 4} and \eqref{eq:proof rw 5}
gives the existence of $\CAC$ and completes the proof.
\end{proof}

The following theorem is a restatement of the Green function
estimate \eqref{eq:bound massive Green function regular walk} of Theorem~\ref{thm:estimate RW}. In order to achieve the constant $\crw$ in \eqref{eq:bound massive Green function regular walk},
we now strengthen the demand in \eqref{eq:crw-anti-concentration} to require that
\begin{equation}
    \crw \le \frac{t_0}{4\sqrt{d}}.
\end{equation}
The proof of Theorem~\ref{thm:estimate RW-app}
uses the following elementary fact.
Given $\lambda >0$, for $x \ge 0$ let   $f(x)=x^{-d/2}e^{-\lambda/\sqrt{x}}$.
Since $f'(x)=f(x)[-\frac{d}{2x}+ \frac{\lambda}{2x^{3/2}}]$, the function
$f$ has a unique maximum at $x=(\lambda/d)^2$, and
\begin{equation}
\label{eq:fmax}
    f(x_0)
    \le f(x) \le f((\lambda/d)^2)=(d/\lambda)^de^{-d}  \quad \text{if $x_0\leq x\le (\lambda/d)^2$.}
\end{equation}

\begin{Thm}[Green function estimate for regular random walks]
\label{thm:estimate RW-app}
Let $d>2$.
For every $\cM,\CM>0$, there exist $\mathsf{C}_{\textup{RW}}=\Crw(\cM,\CM,d)$ and $\crw=\crw(\cM,\CM,d)>0$ such that, for every $(\cM,\CM)$-regular random walk $X$ (started at $0$) on $\mathbb R^d$ of law $\mathbb P$,
Green function $\mathbb G$, and variance $\sigma^2$, every $\mu \in [0,1]$, and every $y\in \mathbb Z^d$,
	\begin{equation}
\label{eq:Green function regular walk}
		\mathbb G_\mu (B_\sigma(y))
    \leq \Crw\left(\frac{\sigma}{\sigma\vee |y|}\right)^{d-2} \exp\left(-\sqrt{1-\mu}\frac{t_0}{4\sqrt{d}}\frac{|y|}{\sigma}\right) .
	\end{equation}
\end{Thm}

\begin{proof}
We first consider $\mu=1$, for which there is no exponential decay.
For  $|y|<2 \sigma$, the desired bound follows from \eqref{eq:it0} together with the fact
that $\sum_{m\geq 1}\tfrac{1}{m^{d/2}}$ is finite when $d>2$.
For the case $|y|\geq 2\sigma$,
we consider large and small $m$ separately.
By Lemma \ref{lem: weak anticonc},
\begin{equation}
	\sum_{m\geq (|y|/\sigma)^2}\mathbb P[X_m\in B_\sigma(y)]
	\leq
	C_1 \sum_{m\geq (|y|/\sigma)^2}m^{-d/2}	
	\leq
	C_2\left(\frac{\sigma}{|y|}\right)^{d-2}.
\end{equation}
For small $m$, we use
\eqref{eq:it0} and apply \eqref{eq:fmax}
to see that
\begin{equation}
	\mathbb P[X_m\in B_\sigma(y)]
    \leq \frac{C_3}{m^{d/2}}\exp\left(-c \frac{|y|}{ \sigma \sqrt{m}}\right)
    \leq C_4\left(\frac{\sigma}{|y|}\right)^d.
\end{equation}
Therefore,
\begin{equation}
	\sum_{ m \leq (|y|/\sigma)^2}\mathbb P[X_m\in B_\sigma(y)]\leq C_5 \left(\frac{\sigma}{|y|}\right)^{d-2}.
\end{equation}
This proves \eqref{eq:Green function regular walk} for the case $\mu=1$.

Let
\begin{equation}
    a = \frac{t_0}{2} \frac{|y|}{\sigma},
    \qquad
    s =\sqrt{1-\mu}.
\end{equation}
For $as \le K$ (any fixed $K>0$) the exponential factor in
\eqref{eq:Green function regular walk} plays no role,
so \eqref{eq:Green function regular walk} holds in this case for all $\mu$ by
using monotonicity in $\mu$ with the result for $\mu=1$.
It therefore suffices to assume in the following that
$as$ is bounded below by whatever value is
convenient.
Fix $s\in (0,1]$.
We again consider large and small $m$ separately, but
now with a different division between large and small.

Consider the contribution to the Green function due to
$m \le d/s^2$.  In this case, we simply use $\mu \le 1$, and then apply
\eqref{eq:it0}.  This gives
\begin{align}
    \sum_{m \le d/s^2} \mu^m \mathbb P[X_m\in B_\sigma(y)]
    & \le
    \CAC e^{t_0^2/8}
    \sum_{m \le d/s^2}   \frac{1}{m^{d/2}} e^{-a/\sqrt{m}}.
\label{eq:m-le-ds2}
\end{align}
We apply \eqref{eq:fmax} a second time
to see that the terms in the sum on the above right-hand side are bounded above
by their value at $m=(a/d)^2$.
We can assume that $d/s^{2} \le (a/d)^2$, because this is a statement
that $as$ is bounded below and we have already dealt with the case when
$as$ is bounded above.  In this case, the terms in the sum on the
right-hand side of \eqref{eq:m-le-ds2} are bounded above by their value for $m=d/s^2$,
and we obtain the desired estimate via
\begin{align}
    \sum_{m \le d/s^2}   \frac{1}{m^{d/2}} e^{-a/\sqrt{m}}
    & \lesssim
    s^{-2} s^{d} e^{-as/\sqrt{d}}
    \\ \nonumber &=
    s^{d-2} e^{-as/2\sqrt{d}}(as)^{-(d-2)}[(as)^{d-2}e^{-as/2\sqrt{d}}]
    \\ \nonumber & \lesssim
    a^{-(d-2)}e^{-as/2\sqrt{d}} .
\label{eq:small-m}
\end{align}

We now turn to the remaining case, which is $m \ge d/s^2 $.
We use $\mu^m =(1-s^2)^m \le e^{-s^2m}$ and apply \eqref{eq:is}.
Therefore $\mu^m e^{s^2m/8} \le 1$, and we find that
\begin{align}
    \sum_{m \ge d/s^2} \mu^m \mathbb P[X_m\in B_\sigma(y)]
    & \le
    \CAC e^{-as/2}
    \sum_{m \ge d/s^2}   \frac{1}{m^{d/2}}
    \\ \nonumber & \lesssim
    e^{-as/2} s^{d-2}
    \lesssim a^{-(d-2)} e^{-as/4},
\end{align}
where the last inequality is as in \eqref{eq:small-m}.

This completes the proof.
\end{proof}

The following is an example of a regular random walk which does not
obey a pointwise version of \eqref{eq:Green function regular walk}.
The averaging in Theorem~\ref{thm:estimate RW-app} therefore plays an important role.

\begin{Ex}
\label{ex:rw}
Let $d>2$ and $N >1$ be integers.
Consider the random walk on $\mathbb Z^d$ whose transition probabilities are given, for $1\leq i \leq d$, by
$\mathbb P[X_1=\pm N\mathbf{e}_i ]=\frac{1}{2d}$.
Then  $\sigma=N$ and
\begin{equation}
    M(s) = \mathbb E[e^{s (\mathbf{e}_1\cdot X_1)/\sigma}] = \frac{d-1}{d}+ \frac{\cosh(s)}{d} .
\end{equation}
Therefore $M(1) \le 1+e$ for every $N$, so the
walk is uniformly regular in $N$.
At criticality,
a non-averaged estimate of \eqref{eq:Green function regular walk} would have
to state that
\begin{equation}
	\mathbb G_1(0,N\mathbf{e}_1)
    \leq \frac{\Crw}{\sigma^d}\left(\frac{\sigma}{\sigma \vee |N\mathbf{e}_1|}\right)^{d-2}
    =
    \frac{\Crw}{N^d}.
\end{equation}
This cannot hold for every $N$ because
$\mathbb G_1(0, N\mathbf{e}_1)\geq \frac{1}{2d}$.
\end{Ex}

\subsection{Pointwise estimates: proof of Proposition~\ref{Prop:Green}}
\label{sec:Green-pointwise}

We now apply Theorem~\ref{thm:estimate RW-app} to prove its pointwise counterpart
for any random walk on $\mathbb Z^d$ whose transition function is given by
an admissible kernel, as defined in Definition~\ref{Def:J}.

Let $J$ be an admissible kernel.
In particular, there is a constant $c_0>0$ such that
\begin{equation}\label{eq: new req J}
    c_0R_J \le \sigma_J,
    \qquad
    J_{x} \le c_0^{-1}R_J^{-d} \quad (x\in \mathbb Z^d).
\end{equation}
Let $\mathbb P_J$ denote the law of the random walk $(X_k)_{k\geq 0}$ started at $0$ and of step distribution given by $J$, and let $\sigma_J^2:=\mathbb E_J[|X_1|_2^2]$. For $\mu\in [0,1]$ and $x\in \mathbb Z^d$, the Green
function is $\mathbb C_\mu(x)=\sum_{m\geq 0}\mu^m \mathbb P_J[X_m=x]$, and we
define the moment generating function $M_J(s)=\mathbb E_J[\exp(s(\mathbf{e}_1)\cdot X_1 )]$.

We restate Proposition~\ref{Prop:Green} here as Proposition~\ref{Prop:Greenbd}.
For the nearest-neighbour $J$, it provides a different proof of the Green function
estimate of \cite[Proposition~2.1]{Slad23_wsaw}.
For the specific spread-out $J$ of \eqref{eq:Jso},
Proposition~\ref{Prop:Greenbd} improves the Green function estimate of
\cite[Proposition~B.1]{LS25a},  which
has less precise control in terms of $\sigma_J$.

\begin{Prop}
\label{Prop:Greenbd}
Let $d>2$.
There exist $\cGreen,\CGreen>0$, which depend on $d$ and $c_0$ but not on $J$, such that for every $\mu\leq 1$, every $m \ge 1$, and every $x\in \mathbb Z^d$,
     \begin{align}
    \label{eq:antix-app}
	&\mathbb P_{J}[X_m =x ]
    \leq \frac{\CGreen}{\sigma_J^d}\frac{1}{m^{d/2}}
    \exp\left(-\cGreen \frac{|x|}{\sigma_J\sqrt{m}}\right),
    \\
\label{eq:GJxbd-bis}
	&\mathbb C_\mu (x)
    \leq \delta_0(x)+ \frac{\CGreen}{\sigma_J^d}
    \left( \frac{\sigma_J}{\sigma_J \vee|x| } \right)^{d-2}
    \exp\left(-\cGreen\sqrt{1-\mu}\frac{|x|}{\sigma_J } \right).
\end{align}
\end{Prop}

\begin{proof}
To begin, we notice that
\begin{equation}\label{eq:gfestimate1new}
	M_J(2/\sigma_J)
    =
    \sum_{x\in \Lambda_R}\exp(2(\mathbf{e}_1\cdot x )/\sigma_J)J_x
    \leq
    \exp(2R_J/\sigma_J)\leq \exp(2/c_0),
\end{equation}
so the random walk is $(2,\exp(2/c_0))$-regular, uniformly in $R_J\geq 1$.
We can therefore conclude from Theorem~\ref{thm:estimate RW-app} that,
with the $d$- and $c_0$-dependent
constants $(\crw,\Crw)$ determined by $(\creg,\Creg)=(2,\exp(2/c_0))$,
and for every $y\in\mathbb Z^d$,
\begin{equation}\label{eq:gfestimate3new}
	\mathbb C_\mu(\Lambda_{\sigma_J}(y))
    \leq \Crw\left(\frac{\sigma_J}{\sigma_J\vee |y|}\right)^{d-2}
    \exp\left(-\crw\sqrt{1-\mu}\frac{|y|}{\sigma_J}\right).
\end{equation}

We first prove \eqref{eq:GJxbd-bis}, by
``unaveraging'' \eqref{eq:gfestimate3new}. For this,
we start by using $J_x\leq c_0^{-1} R_J^{-d}$ (by \eqref{eq: new req J}) and the definition of $R_J$
to see that, for $m\geq 1$,
\begin{align}
\label{eq:unav1}
    \mathbb P_J[X_{m}=x] &\leq \sum_{u\in \Lambda_{R_J}(x)\setminus\{x\}} \mathbb P_J[X_{m-1}=u]J_{u-x}
   \leq  \frac{1}{c_0R_J^d}\mathbb P_J[X_{m-1}\in \Lambda_{R_J}(x)].
\end{align}
The Green function therefore obeys
\begin{align}
    \mathbb C_\mu(x) &= \delta_{0}(x) +\sum_{m \ge 1} \mu^m \mathbb P_J[X_{m}=x]
    \le
    \delta_{0}(x)
    + \frac{1}{c_0 R_J^d} {\mathbb C}_\mu(\Lambda_{R_J}(x))
    .
\end{align}
With \eqref{eq:gfestimate3new}, this is almost what we need, but there
is a mismatch between the boxes of size $R_J$ and $\sigma_J$.
To deal with this, we let $b=\lceil c_0^{-1}\rceil$
and $B=(2b+1)^d$, so that, by \eqref{eq:gfestimate3new}
\begin{align}
    {\mathbb C}_\mu(\Lambda_{R_J}(x))
    &\le
    \sum_{j\in \Lambda_{b}}{\mathbb C}_\mu(\Lambda_{\sigma_J}(x+j\sigma))
    \\ &
    \le
    B \Crw \sup_{j\in \Lambda_{b}}
    \left(\frac{\sigma_J}{\sigma_J\vee |x+j\sigma_J|}\right)^{d-2}
    \exp\left(-\crw\sqrt{1-\mu}\frac{|x+j\sigma_J|}{\sigma_J}\right).
    \nonumber
\end{align}
Suppose first that $|x| \le 2 b\sigma_J$.  Then,
\begin{equation}
    \frac{\sigma_J}{\sigma_J\vee |x+j\sigma_J|}
    \le 1
    \le
    2b
    \frac{\sigma_J}{\sigma_J\vee |x|},
\end{equation}
\begin{equation}
    \exp\left(-\crw\sqrt{1-\mu}\frac{|x+j\sigma_J|}{\sigma_J}\right)
    \le 1
    \le
    e^{\crw 2b}
    \exp\left(-\crw\sqrt{1-\mu}\frac{|x|}{\sigma_J}\right).
\end{equation}
These give the desired bound for the case $|x| \le 2 b\sigma_J$, with suitable constants
$\cGreen,\CGreen$.
Suppose instead that $|x| \ge 2 b\sigma_J$.  In this case, for every $j\in \Lambda_b$,
\begin{equation}
    |x+j\sigma_J| \ge |x| - b\sigma_J \ge \frac 12 |x|,
\end{equation}
which implies the desired bound \eqref{eq:GJxbd-bis} with suitable constants
$\cGreen,\CGreen$.

We prove \eqref{eq:antix-app} similarly.
For $m=1$ it is straightforward, so we assume $m \ge 2$.
We start with \eqref{eq:unav1} and apply the averaged anti-concentration
bound of \eqref{eq:it0}, and $\sigma_J \le R_J$, to obtain
\begin{align}
    \mathbb P_J[X_{m}=x]
    & \leq \frac{B}{c_0
    \sigma_J^d}
    \frac{\CAC e^{t_0^2/8}}{(m-1)^{d/2}}
    \sup_{j\in\Lambda_{b}}
    \exp\left(-t_0 \frac{|x+j\sigma_J|}{2\sigma_J \sqrt{m-1}}\right)
    \nonumber \\ & \le
    \frac{B}{c_0
    \sigma_J^d}
    \frac{2^{d/2}\CAC e^{t_0^2/8}}{m^{d/2}}
    e^{2t_0 b   }
    \sup_{j\in\Lambda_{b}}
    \exp\left(-
    \frac{t_0}{4} \frac{|x|}{\sigma_J \sqrt{m}}\right).
\end{align}
This gives \eqref{eq:antix-app} and, after relaxing the values of $\cGreen,\CGreen$
if necessary, completes the proof.
\end{proof}

\section{Convolution estimates}\label{appendix:conv estimates}

In this appendix, we collect two convolution estimates.
First, we prove Lemma~\ref{lemma:fstarg}, which we restate as
the following lemma.

\begin{Lem}
Let $a,b,c_1,c_2,\sigma,\xi >0$, $\mu >0$ and $\varepsilon \in [ 0,1]$.
Suppose that the functions $f,g: \mathbb Z^d \to [0,\infty)$ satisfy
\begin{align}
    f(x)
    &\le
    c_1\frac{1}{\sigma^d} \left( \frac{\sigma}{\sigma\vee |x|}\right)^{d-2-\varepsilon}
    e^{-a|x|/\xi},
    \\
    g(\Lambda_\xi(x))
    &\le
    c_2\left( \frac{\xi}{\xi\vee |x|}\right)^{d-2}
    e^{-b|x|/\xi}.
\end{align}
Then, there is a constant $C_{a,\mu}$ such that, for every $|x| \ge 2 (\sigma \vee \xi)$,
\begin{equation}
\label{eq:fstargapp}
    \sum_{y \notin \Lambda_{ \mu \xi}(0)} f(y)g(x-y)
    \le \frac{c_1}{\sigma^2|x|^{d-2}} \Big(\frac{|x|}{\sigma}\Big)^\varepsilon
    \left(
    2^{d} \|g\|_1  e^{-a|x|/2\xi}
    +
    c_2C_{a,\mu}\Big(\frac{\xi}{|x|}\Big)^\varepsilon e^{-b|x|/2\xi}
    \right).
\end{equation}
\end{Lem}

\begin{proof}
We first observe that we may assume that $\mu\in (0,1)$.
Indeed, if $\mu \ge 1$ then we can bound the left-hand side of \eqref{eq:fstargapp}
by the left-hand side with $\mu = \frac 12$.  That new left-hand side is bounded by the corresponding right-hand side, which suffices.

The first term in the upper bound on the convolution arises from
\begin{align}
    \sum_{y \notin \Lambda_{|x|/2}(0)} f(y)g(x-y)
    & \le
    \frac{c_1}{\sigma^d}
    \left( \frac{\sigma}{ |x|/2}\right)^{d-2-\varepsilon}e^{-a|x|/2\xi} \sum_{y \notin \Lambda_{|x|/2}(0)}  g(x-y)
    \nonumber \\ & \le
\frac{c_1 }{\sigma^{2}|x|^{d-2}}\Big(\frac{|x|}{\sigma}\Big)^\varepsilon 2^d
    e^{-a|x|/2\xi}\|g\|_1.
\label{eq:fg1}
\end{align}

The remaining contribution to \eqref{eq:fstargapp} is due to the sum over $y \in \Lambda_{|x|/2}(0)
\setminus \Lambda_{\mu\xi}(0)$.
We do not have a
pointwise hypothesis on $g$, so averaging is required.
For $k \ge 0$, we define the annulus $A_k:= \{z\in \mathbb Z^d: 2^{k} \mu \xi < |z| \le  2^{k+1}\mu\xi\}$. We decompose the annulus
$\Lambda_{|x|/2}(0) \setminus \Lambda_{\mu \xi }(0)$
into the annuli $A_k$, with $0 \le k \le \log_2(|x|/2\mu \xi)-1$.
This gives
\begin{equation}
    \sum_{y \in \Lambda_{|x|/2}(0)} f(y)g(x-y)
    \le
    \sum_{k=0}^{\log_2(|x|/2 \mu\xi)-1}\sum_{y\in A_k }
    f(y)g(x-y).
\end{equation}
By hypothesis, if $k\geq 0$ and $y \in A_k$, we have the uniform bound
\begin{equation}
    f(y)
    \le
    \frac{c_1}{\sigma^d}
    \left(\frac{\sigma}{\sigma \vee 2^{k} \xi \mu}\right)^{d-2-\varepsilon} e^{-a 2^{k} \mu\xi/\xi}  \le  \frac{c_1}{\sigma^{2+\varepsilon}}
    \left(\frac{1}{ 2^{k} \xi \mu}\right)^{d-2-\varepsilon}
    e^{-a 2^{k} \mu} .
\end{equation}
Also, since $\Lambda_{\mu \xi}(x-y)\subset \Lambda_{ \xi}(x-y)$,
\begin{align}
    \sum_{y\in A_k }g(x-y)
    & \le
    2^{kd} \sup_{y\in A_k}g(\Lambda_{\mu \xi}(x-y))
    \le
    2^{kd} \sup_{y\in A_k}g(\Lambda_{\xi}(x-y)).
\end{align}
It follows from
the hypothesis on $g$, and the fact that $|y|\le |x|/2 $,
that
\begin{equation}
    \sup_{y\in A_k}g(\Lambda_{\xi}(x-y))
    \le
    c_2\left( \frac{\xi}{|x|/2}\right)^{d-2}
    e^{-b|x|/2\xi}.
\end{equation}
We abbreviate the notation for the annulus of interest
by defining $A:=\{y \in \Lambda_{|x|/2}(0)
\setminus \Lambda_{\mu\xi}(0)\}$.
Altogether, we find that
\begin{align}
    \sum_{y \in A} f(y)g(x-y)
&\le
    \frac{c_1}{\sigma^{2+\varepsilon}}
    \left(\frac{1}{|x|} \right)^{d-2-\varepsilon}
    \left(\frac{\xi}{|x|}\right)^{\varepsilon}
    c_2   e^{-b|x|/2\xi}
    \nonumber \\ & \quad \times
    \Big[\left(\frac{2}{\mu}\right)^{d-2}\sum_{k\geq 0}  e^{-a 2^k  \mu}
    2^{k(2+\varepsilon)}\Big]
    .
\end{align}
The factor in square brackets is bounded by a constant $C_{a,\mu}$ which depends on $a,\mu,d$.
We therefore obtain the desired upper bound
\begin{equation}
    \frac{c_1}{\sigma^2|x|^{d-2}}\left(\frac{|x|}{\sigma}\right)^\varepsilon c_2 C_{a,\mu}\left( \frac{\xi}{|x|}\right)^{\varepsilon} e^{-b|x|/2\xi} .
\end{equation}
This completes the proof.
\end{proof}

Next, we prove Lemma~\ref{lemma:fstarf}, which we restate as follows.

\begin{Lem}
Let $p,a>0$.
For $i=1,2$, suppose that $f_i \in \ell^1(\mathbb Z^d)$
satisfy $0 \le f_i(x) \le a(1\vee |x|)^{-p}$ for all $x\in\mathbb Z^d$.
Let $k\ge 1$. Then
\begin{equation}
    (f_1*f_2)(x)
    \le
    \frac{a}{(1\vee |x|)^p}
    \Big(  \frac{1}{k^{p}}  \|f_1\|_1
    + 2^{p} \sum_{y \in \Lambda_{k|x|}(0)}(f_1(y)+f_2(y)) \Big).
\end{equation}
\end{Lem}

\begin{proof}
We divide the sum $\sum_y f_1(x-y)f_2(y)$ into three parts:
\begin{equation}
    y \in \Lambda_{|x|/2}(x), \qquad y \in \Lambda_{k|x|}(0)\setminus \Lambda_{|x|/2}(x),
    \qquad y \notin \Lambda_{k|x|}(0).
\end{equation}
For the first case, we have $|y| \ge |x|/2$, so by hypothesis,
and by $\Lambda_{|x|/2}(0) \subset \Lambda_{k|x|}(0)$,
\begin{align}
    \sum_{y \in \Lambda_{|x|/2}(x)} f_1(x-y)f_2(y)
    & \le
    \frac{2^p a}{(1\vee |x|)^p} \sum_{y \in \Lambda_{k|x|}(0)}f_1(y).
\end{align}
For the second case, $|x-y|\ge |x|/2$, so
\begin{align}
    \sum_{y \in \Lambda_{k|x|}(0)\setminus \Lambda_{|x|/2}(x)} f_1(x-y)f_2(y)
    & \le
    \frac{2^p a}{(1\vee |x|)^p} \sum_{y \in \Lambda_{k|x|}(0)}f_2(y).
\end{align}
For the third case, $|y|\ge k|x|$, so
\begin{align}
    \sum_{y \notin \Lambda_{|x|}(x)} f_1(x-y)f_2(y)
    & \le
    \frac{a}{k^p(1\vee |x|)^p} \|f_1\|_1.
\end{align}
This completes the proof.
\end{proof}

\section*{Acknowledgements}
\noindent
We thank Christophe Garban and Yvan Velenik for useful comments.
 This project has received funding from the Swiss National Science
Foundation and the NCCR SwissMAP. HDC acknowledges the support from the Simons
collaboration on localization of waves. RP acknowledges the support of the Swiss National Science Foundation through a Postdoc.Mobility grant.
 The work of GS was supported in part by the Natural Sciences and Engineering Research Council of Canada (NSERC), Grant no. GR010086.
GS thanks the
Mathematics Section
of the University of Geneva for kind hospitality during a visit
at the start of this work.

\bibliographystyle{plain}
\bibliography{bib}

\end{document}